\documentclass[hidelinks,11pt]{article}
\usepackage{amsmath,amssymb,latexsym,mathdots,comment,enumerate}
\usepackage{graphicx}
\usepackage[all,cmtip]{xy}
\usepackage{color}

\usepackage[utf8]{inputenc}
\usepackage[T1]{fontenc}
\usepackage{lmodern}
\usepackage{mathtools, nccmath, textcomp} 

\usepackage{hyperref}

\DeclareMathSymbol{*}{\mathbin}{symbols}{"03}
\def\Z{{\mathbb Z}}

\def\SL{{\rm SL}}
\def\GL{{\rm GL}}
\def\Gal{{\rm Gal}}

\def\ndist{{\rm ndist}}
\def\main{{\rm main}}

\def\scusp{{\rm scusp}}
\def\dcusp{{\rm dcusp}}
\def\Sym{{\rm Sym}}

\def\B{{\mathcal B}}
\def\SO{{\rm SO}}

\def\b{{\rm Sq}}

\def\P{{\mathbb P}}
\def\Disc{{\rm Disc}}
\def\diag{{\rm diag}}
\def\s{{\rm stab.}}

\def\irr{{\rm irr}}
\def\gen{{\rm gen}}

\def\dist{{\rm dist}}

\def\red{{\rm red}}
\def\Vol{{\rm Vol}}
\def\R{{\mathbb R}}

\def\F{{\mathbb F}}
\def\FF{{\mathcal F}}
\def\cI{{\mathcal I}}
\def\cZ{{\mathcal Z}}
\def\cM{{\mathcal M}}
\def\cK{{\mathcal K}}
\def\cS{{\mathcal S}}
\def\RR{{\mathcal R}}

\def\LLM{{\L(M)}}
\def\LL1{{\L(1)}}
\def\Q{{\mathbb Q}}

\def\G{{\mathbb G}}

\def\D{{\mathcal D}}
\def\C{{\mathcal C}}
\def\U{{\mathcal W}}
\def\W{{\mathcal W}}

\def\Z{{\mathbb Z}}
\def\P{{\mathbb P}}
\def\F{{\mathbb F}}
\def\Q{{\mathbb Q}}
\def\C{{\mathbb C}}

\def\L{{\mathcal L}}

\def\w{{\rm {(2)}}}
\def\s{{\rm {(1)}}}

\def\max{{\rm max}}

\def\top{{\rm top}}
\newtheorem{Theorem}{Theorem}
\newtheorem{theorem}{Theorem}[section]
\newtheorem{corollary}[theorem]{Corollary}
\newtheorem{lemma}[theorem]{Lemma}
\newtheorem{proposition}[theorem]{Proposition}
\newtheorem{remark}[theorem]{Remark}
\newenvironment{proof}{\noindent {\bf Proof:}}{$\Box$ \vspace{2 ex}}

\title{\vspace{-.5in}Squarefree values of polynomial discriminants II}

\author{Manjul Bhargava, Arul Shankar, and Xiaoheng Wang}

\begin{document}

\maketitle

\vspace{-.2in}
\begin{abstract}
We determine the density of integral binary forms of given degree that
have squarefree discriminant, proving for the first time that the
lower density is positive.
Furthermore, we determine the density of integral binary forms that cut out maximal orders in number fields. The latter proves, in particular, an ``arithmetic Bertini theorem''
conjectured by {Poonen}  for {$\P^1_\Z$}. 

Our methods also allow us to prove that there are $\gg X^{1/2+1/(n-1)}$
number fields of degree~$n$ having associated Galois group~$S_n$ and
absolute discriminant less than $X$, improving the best previously
known lower bound of $\gg X^{1/2+1/n}$.

Finally, our methods correct an error in and thus resurrect earlier (retracted) results of Nakagawa on lower bounds for
the number of totally unramified $A_n$-extensions of quadratic number fields of bounded discriminant. 
\end{abstract}

\setcounter{tocdepth}{2}


\vspace{-.1in}
\section{Introduction}

In the first article \cite{BSWsq} of this two-part series, we proved that when monic integer polynomials $f(x)=x^n+a_1x^{n-1}+\cdots+a_n$ of fixed degree $n$ are ordered by
$\max\{|a_1|,\ldots,|a_n|^{1/n}\}$, a positive proportion have squarefree discriminant. The purpose of this article is to prove the analogous result for integral binary $n$-ic
forms.  

Recall that the {\it discriminant} $\Delta(f)$ of a binary
$n$-ic form over a field $K$ is a homogeneous
polynomial of degree $2n-2$ in the coefficients of $f$, whose
nonvanishing is equivalent to $f$ having~$n$ distinct linear factors
over an algebraic closure $\overline{K}$ of $K$.  We order integral binary
$n$-ic forms $f(x,y)=a_0x^n+a_1x^{n-1}y+\cdots+a_ny^n$ by their {\it height} $H(f)$ given by
$H(f):=    
\max\{|a_0|,\ldots,|a_n|\},$
i.e., the maximum of the absolute values of the coefficients. 
Then a natural question is: when ordered by height, what is the density of integral binary $n$-ic forms whose
discriminant is squarefree? For $n=2$, classical methods in sieve
theory yield the answer. For $n=3$ and $n=4$, results of
Davenport--Heilbronn \cite{DH} and the first and second authors
\cite{BS2}, respectively, answer the question in the related setting in
which we consider $\GL_2(\Z)$-orbits on binary $n$-ic forms. However,
for $n\geq 5$, it has not previously been known whether this density
exists or even whether the lower density is positive.
In this paper, we prove:

\begin{Theorem}\label{polydisc2}
Let $n\geq 2$ be an integer. When integral binary $n$-ic forms $f(x,y)=a_0x^n+a_1x^{n-1}y+\cdots+a_n y^n$ are ordered by $H(f):=\max\{|a_0|,\ldots,|a_n|\}.
$, the density of forms having squarefree discriminant exists
and is equal to 
\begin{equation*}
\begin{array}{lll}
\displaystyle\mfrac12
\prod_{p>2}\Bigl(1-\frac{1}{p}\Bigr)
\Bigl(1+\frac{1}{p}-\frac{1}{p^3}\Bigr)
&\approx \;\;38.97\%&\mbox{ if }n=2;
\\[.175in] \displaystyle\mfrac38
\prod_{p>2}\Bigl(1-\frac{1}{p}\Bigr)^2\Bigl(1+\frac{1}{p}\Bigr)^2&\approx \;\;24.64\%
&\mbox{ if }n=3;
\\[.175in] \displaystyle\mfrac38
\prod_{p>2}\Bigl(1-\frac{1}{p}\Bigr)^2
\Bigl(1+\frac{2}{p}-\frac{2}{p^4}+\frac{1}{p^5}\Bigr)&\approx\;\;21.18\%
&\mbox{ if }n=4;
\\[.175in] \displaystyle\mfrac38
\prod_{p>2}\Bigl(1-\frac{1}{p}\Bigr)^2
\Bigl(1+\frac{1}{p}\Bigr)\Bigl(1+\frac{1}{p}-\frac{1}{p^2}\Bigr)&\approx \;\;20.83\%
&\mbox{ if }n\geq5.
\end{array}
\end{equation*}
\end{Theorem}

To any nonzero integral binary $n$-ic form $f(x,y) = a_0x^n + \cdots +
a_ny^n$, we may naturally attach a rank-$n$ ring $R_f$ (see~Birch--Merriman~\cite{BM}, Nakagawa~\cite{Nakagawa1}, and Wood~\cite{Wood}), defined as follows when $a_0\neq 0$. Let $\theta$
denote the image of $x$ in $K_f:=\Q[x]/(f(x,1))$. Let $R_f$ be the free
rank-$n$ $\Z$-submodule of $K_f$ generated by $1, \,a_0\theta, \,a_0\theta^2+a_1\theta,\,\ldots,\,a_0\theta^{n-1}+\cdots+a_{n-1}\theta$.
Then $R_f$ is in fact closed under multiplication and forms a ring whose 
discriminant is equal to the discriminant of $f(x)$.  
Our next result determines the density of irreducible integral binary forms $f$ for which $R_f$ is the maximal order in its field of fractions.

\begin{Theorem}\label{polydiscmax2}
Let $n\geq 2$ be an integer. When irreducible integral binary $n$-ic forms $f(x,y)=a_0x^n+a_1x^{n-1}y+\cdots+a_ny^n$ are ordered by
$H(f):=\max\{|a_0|,\ldots,|a_n|\}$, the density of forms $f$ such that $R_f$ is the ring of integers in its field of fractions exists and is equal to 
\begin{equation*}
  \begin{array}{lll} \displaystyle\prod_p
    \Big(1-\frac{1}{p^2}-\frac{1}{p^3}+\frac{1}{p^4}\Big)
    &\approx \;\;53.59\%&\mbox{ if }n=2;\\[.1725in] \zeta(2)^{-1}\zeta(3)^{-1} &\approx \;\;50.57\%&\mbox{ if }n\geq3.\end{array}
\end{equation*}
\end{Theorem}

In particular, Theorem~\ref{polydiscmax2}  yields the first unconditional Bertini theorem for arithmetic schemes of dimension $\geq 2$ as conjectured by Poonen~\cite[\S5]{BPBertini}.  Indeed, for a quasiprojective subscheme $X$ of $\P^n_\Z$ that is regular of dimension $m$, Poonen
conjectured that the density of hyperplane sections of $X$ that are  regular of dimension $m-1$ should equal $\zeta_X(m+1)$, where $\zeta_X$  denotes the zeta function of $X$. 
Since the subscheme of $\P^1_\Z$ cut out by an integral binary $n$-ic form $f$ is regular if and only if $R_f$ is maximal,  and the zeta function of $\P^1_\Z$ is given by $\zeta_{\P^1_\Z}(s)=\zeta(s)\zeta(s-1)$, we have 
$\zeta_{\P^1_\Z}(\dim(\P^1_\Z)+1)^{-1}=\zeta(2)^{-1}\zeta(3)^{-1}$. Therefore, Theorem \ref{polydiscmax2} yields an unconditional proof of \cite[Theorem 5.1]{BPBertini} for the case $X=\P^1_\Z$ with the usual ``box ordering'' on the forms defining the hyperplane sections. In fact, we prove the stronger result that for every {\it fixed} $n\geq 3$, the density of regular binary $n$-ic forms is $\zeta(2)^{-1}\zeta(3)^{-1}$, while arithmetic Bertini only claims this in the limit as~$n\to\infty$.

As a further application of our
methods, we obtain the following theorem:

\begin{Theorem}\label{monogenic}
For each $n\geq 3$, the number of isomorphism classes of number fields
of degree~$n$ with associated Galois group $S_n$ and absolute  discriminant less than 
$X$ is $\gg X^{1/2+1/(n-1)}$.
\end{Theorem}
Our lower bound in Theorem~\ref{monogenic} on the number of degree-$n$ $S_n$-number
fields of absolute discriminant less than $X$  improves the previous  best-known lower
bound of $X^{1/2+1/n}$ obtained in \cite{BSWsq}.
We note that the number fields constructed in Theorem~\ref{monogenic} can all be taken to have squarefree discriminant.

Our results also correct an error in, and thus resurrect, all the results of Nakagawa~\cite{Nakagawa1} and~\cite{Nakagawa2} that had been subsequently retracted in \cite{Nakagawa1E} and \cite{Nakagawa2E}. Specifically, the retracted theorems \cite[Theorems~3--4]{Nakagawa1} and \cite[Theorem~2]{Nakagawa2} regarding binary forms and  $A_n$-extensions of quadratic fields can now be taken to be true. In particular, we obtain:

\begin{Theorem}\label{thm:unr}
For $n\geq 3$, the total number of unramified $A_n$-extensions of real $($resp., imaginary$)$ quadratic fields $F$, across all such $F$ such that $|\Disc(F)|<X$, is $\gg  X^{(n+1)/(2n-2)}$.
\end{Theorem}
Theorem~\ref{thm:unr} yields the best-known lower bounds on the number of unramified $A_n$-extensions of quadratic fields when $n>5$. For improved bounds in the cases $n\leq 5$, see~\cite[Theorem~1.4]{geosieve}. For the best-known bounds on the number of quadratic fields of bounded discriminant admitting an unramified $A_n$-extension, see~Kedlaya~\cite[Corollary~1.4]{Kedlaya}.  
Other related works include  Uchida~\cite{Uchida}, Yamamoto~\cite{Yamamoto}, and Yamamura~\cite{Yamamura}.

The main technical ingredient required to prove all the above results is a ``tail estimate'' which shows that not too many
discriminants of integral binary $n$-ic forms $f$ are divisible by
$p^2$ when $p$ is large relative to the discriminant of $f$ (here,
large means larger than $H(f)$, say). It is these tail estimates that were missing in Nakagawa's work.
For a prime $p$, and an integral binary $n$-ic form $f$ such that
$p^2\mid \Delta(f)$, we say that $p^2$ {\it strongly divides}
$\Delta(f)$ if $p^2\mid \Delta(f + pg)$ for every integral binary
$n$-ic form $g$; otherwise, we say $p^2$ {\it weakly divides}
$\Delta(f)$.  For any squarefree integer~$m>0$, let $\U_m^\s$
(resp., $\U_m^\w$) denote the set of integral binary $n$-ic forms
whose discriminants are strongly divisible (resp., weakly divisible)
by $p^2$ for every prime factor $p$ of $m$.  

We prove the following tail estimates:
\begin{Theorem}\label{thm:mainestimate} 
For an integer $n\geq 3$, a positive real number $M$ and any $\epsilon>0$, we have:
\begin{equation*}
\begin{array}{rlcl}
\displaystyle \rm{(a) } &
\displaystyle
\#\bigcup_{\substack{m>M\\ m\;\mathrm{ squarefree}}}
    \{f\in\U_m^\s:H(f)<X\}&=&
\displaystyle
    O_\epsilon\Bigl(\frac{X^{n+1+\epsilon}}{M}+X^{n}\Bigr);
\\[.2in]

\displaystyle \rm{(b) } &
\displaystyle
\#\bigcup_{\substack{m>M\\
m\;\mathrm{ squarefree} }}\{f\in\U_m^\w:H(f)<X\}&=&
\displaystyle
O_\epsilon\Bigl(\frac{X^{n+1+\epsilon}}{M}+X^{n+1-1/(5n)+\epsilon}\Bigr),\mbox{ if }2\nmid n;
\\[.2in]
\displaystyle \rm{(c) } &
\displaystyle
\#\bigcup_{\substack{m>M\\
m\;\mathrm{ squarefree} }}\{f\in\U_m^\w:H(f)<X\}&=&
\displaystyle
O\Bigl(\frac{X^{n+1+1/(88n^5)}}{\sqrt{M}}+X^{n+1-1/(88n^6)}\Bigr),\mbox{ if }2\mid n.
\end{array}
\end{equation*}
\end{Theorem}

The estimate in the strongly divisible Case (a) of
Theorem~\ref{thm:mainestimate} follows from geometric techniques,
namely, the quantitative version of the Ekedahl geometric sieve as developed by the first author \cite{geosieve}.  The estimates in the weakly divisible Cases
(b) and (c) of Theorem~\ref{thm:mainestimate} are considerably more
difficult (particularly (c)), and we describe their proofs in the next section. 
Our tail estimate in fact allows us to prove Theorems~\ref{polydisc2} and \ref{polydiscmax2} with
power-saving error terms: 
\begin{Theorem}\label{thm:powersave}
Let $V_n=\Sym^n(2)$ denote the space of binary $n$-ic forms.
Define $\eta_n$ to be $1/(5n)$ when $n$ is odd and $1/(88n^6)$ when $n$ is even.
Then 
\begin{equation*}
\begin{array}{ccl} \displaystyle \#\{f\in V_n(\Z) : H(f)<X \mbox{ and
      $\Delta(f)$ squarefree}\}&\!\!=\!\!& \alpha_n\cdot (2X)^{n+1} +
    O_\epsilon(X^{n+1-\eta_n+\epsilon});\\[.125in] \displaystyle \#\{f\in
    V_n(\Z) : H(f)<X \mbox{ and $R_f$ maximal}\}&\!\!=\!\!&
    \beta_n\cdot (2X)^{n+1} + O_\epsilon(X^{n+1-\eta_n+\epsilon}).
\end{array}
\end{equation*}
\end{Theorem}
These power saving bounds have applications towards
level-of-distribution questions when counting integral binary $n$-ic
forms $f$ of bounded height with $\Delta(f)$ squarefree (resp., $R_f$ maximal)
satisfying splitting conditions at finitely many primes. Such
level-of-distribution results in turn have applications towards a host
of problems in analytic number theory, such as studying statistics of
Artin $L$-functions attached to binary $n$-ic forms and proving lower
bounds on the number of degree-$n$ number fields which are ramified
only at a bounded number of primes, among many others.

We remark that our methods imply that the analogues of all of the
above results also hold when local conditions are imposed at
finitely many places (including at infinity); the orders of magnitudes
in these theorems remain the same, provided
that no local conditions are imposed that force the sets being counted
in Theorems~\ref{polydisc2} and~\ref{polydiscmax2} to be empty.

Finally, the methods introduced in \cite{BSWsq} and in the current article have applications beyond just squarefree values of polynomial discriminants.  They have been recently adapted in~\cite{BhargavaHo2} to determine the density of squarefree discriminants of elliptic curves over $\Q$ having two marked rational points. Other applications include determining the density of conductors in some families of elliptic curves \cite{ECC} and the density of squarefree values taken by $a^4+b^3$ (\cite{SanW}).

\section{Outline of proof}

As mentioned in the introduction, the uniformity estimate in Theorem \ref{thm:mainestimate} is the key to deducing Theorems \ref{polydisc2}, \ref{polydiscmax2}, and \ref{thm:powersave} via a  squarefree sieve. Case (a) of Theorem \ref{thm:mainestimate} follows directly from the results in \cite{geosieve}. Case (b), which pertains to odd degrees $n$, can be proven using methods similar to those developed in our previous work \cite{BSWsq}. However, these methods fail to work for Case (c), which pertains to even degrees $n$, and a number of new ideas are required to handle this case. It is the proof of this case to which the bulk of our paper is devoted; it requires, in particular, the introduction of a new technique in the geometry of numbers that we believe may also be useful in other contexts. 

In this section, we give a detailed outline of the proof of Case (b) pertaining to odd $n$. We then explain why this strategy breaks down (quite spectacularly!) when $n$ is even, and finally we describe the new techniques required to complete the proof of Theorem \ref{thm:mainestimate}(c).

\subsubsection*{Sketch of the proof of the tail estimate for odd $n$}

Our proof of Theorem~\ref{thm:mainestimate}(b) makes use of the representation of $G=\SL_n$
on the space $W=2\otimes\Sym_2(n)$ of pairs $(A,B)$ of symmetric $n\times n$ matrices, studied
in detail in \cite{Wood1,pointless,BGW,BGW2}.  The group $G$ acts on $W$ via 
$\gamma\cdot (A,B)=(\gamma A\gamma^t,\gamma B\gamma^t$) for $\gamma\in
G$ and $(A,B)\in W$.  We define the {\it invariant binary form} of an
element $(A,B)\in W$ by $$\smash{f_{A,B}(x,y) = (-1)^{n(n-1)/2}\det(Ax -
By).}$$ Then $f_{A,B}$ is a binary $n$-ic form satisfying
$f_{\gamma(A,B)}=f_{A,B}$. Moreover, the ring of polynomial
invariants for the action of $G$ on $W$ is freely generated by the
coefficients of the invariant binary form (reference? it is in some classification of coregular spaces). Define the {\it
  discriminant} $\Delta(A,B)$ and {\it height} $H(A,B)$ of an element
$(A,B)\in W$ by $\Delta(A,B)=\Delta(f_{A,B})$ and
$H(A,B)=H(f_{A,B})$.

The first step of our proof is the construction, for every 
squarefree integer $m>0$, of a map
\begin{equation*}
\sigma_m:\U_m^\w\to W(\Z),
\end{equation*}
such that $f_{\sigma_m(f)}(x,y)=f(x,y)$ for every $f\in \U_m^\w$.  In
our construction, the image of $\sigma_m$ in fact lies in $W_0(\Z)$,
where $W_0$ is the subspace of $W$ consisting of pairs of matrices
whose top left $g\times g$ blocks are $0$, where $n=2g+1$. The action of the group $G$
does not preserve $W_0$, and we take $G_0$ to be the maximal parabolic
subgroup of $G$ that does preserve $W_0$. When the discriminant
polynomial $\Delta\in\Z[W]$ is restricted to $W_0$, it is no longer
irreducible but rather is divisible by the square of a polynomial
$Q\in\Z[W_0]$. This polynomial $Q$ is a relative invariant for the
action of $G_0$ on $W_0$. Its significance is that,
by construction of $\sigma_m$, every element in the the image of
$\sigma_m$ has $Q$-invariant equal to $m$. To prove Part (b) of
Theorem \ref{thm:mainestimate}, it therefore suffices to estimate the
number of $G_0(\Z)$-orbits on $W_0(\Z)$ having height less than $X$
and $Q$-invariant greater than $M$.

Bounding the number of these orbits is complicated by the fact that
$G_0$ is not reductive. We are rescued by using the full action of
$G(\Z)$ on $W(\Z)$.  This necessitates expanding the definition of the
$Q$-invariant from $W_0(\Z)$ to all ``distinguished'' elements of $W(\Z)$.  An element $(A,B)\in W(\Z)$ is {\it distinguished}
if $A$ and $B$ have a common isotropic $g$-dimensional subspace
defined over $\Q$. Thus every
element in $W_0(\Z)$ (and thus every element in the image of
$\sigma_m$) is distinguished. The $Q$-invariant, though defined
initially on $W_0$, can be extended as a function on the set of all
triples $(A,B,\Lambda)$, where $(A,B)\in W(\Z)$ is distinguished, and
$\Lambda$ is a common isotropic subspace of $A$ and~$B$. For all but a
negligible number of distinguished elements $(A,B)\in W(\Z)$, $A$ and
$B$ have exactly one common isotropic subspace $\Lambda$ defined over $\Q$. Thus
we may define a $G(\Z)$-invariant function~$Q$ on the set of
distinguished pairs $(A,B)\in W(\Z)$ outside a negligible number of
them. It then  \pagebreak suffices to bound the number of $G(\Z)$-orbits on
distinguished elements in $W(\Z)$ having bounded height and large
$Q$-invariant.

To obtain such a bound, we construct fundamental domains 
for the action of $G(\Z)$ on elements in $W(\R)$ with height
less than $X$. Such a fundamental domain has a natural partition into
three parts that we term the {\it main body}, the {\it shallow cusp}, and the {\it
  deep cusp}. We have little control over the
$Q$-invariants of elements in the main body and the shallow cusp. However, it
is known~\cite[Proposition 4.3]{HSV} 
that there are a negligible number of integral
elements in the shallow cusp. Meanwhile, distinguished elements occur rarely in
the main body, a fact we prove via the Selberg sieve.

Finally, the deep cusp lies in $W_0$, where an upper
bound for the $Q$-invariant can be obtained. Imposing the condition
that this upper bound is greater than $M$, and counting the number of such points in the deep cusp using the averaging method of \cite{quintic}, gives the desired saving for the number of 
elements in the deep cusp having $Q$-invariant larger than $M$. Combining the estimates for the main
body, the shallow cusp, and the deep cusp yields Part (b) of Theorem
\ref{thm:mainestimate}.

\subsubsection*{Sketch of the proof of the tail estimate for even $n$}

With $W$ again denoting the space of pairs of symmetric $n\times n$ matrices, we
may attempt to proceed in the same manner as in the case of odd $n$, by 
constructing a map \begin{equation*}
\smash{\sigma_m:\U_m^\w\to W(\Z)}
\end{equation*}
such that \smash{$f_{\sigma_m(f)}(x,y)=f(x,y)$ for every $f\in \U_m^\w$}.  However, such a map does not exist in the case that $n$ is even!  Indeed, there exist integral binary $n$-ic forms $f(x,y)$ that cannot be expressed as $\det(Ax -
By)$---even up to sign---for any integral $n\times n$ symmetric matrices $A$ and $B$.  This phenomenon was extensively studied in \cite{pointless,BGW,BGW2}. 
It is in this sense that the strategy to prove Theorem~\ref{thm:mainestimate}(b) for odd $n$ fails spectacularly for even $n$---and at the very first step.

We address this issue by replacing
\smash{$f(x,y)\in\W_m^{(2)}$} by $xf(x,y)$, which is a reducible binary
$(n+1)$-ic form whose discriminant, at least generically, remains weakly divisible by
$m^2$. For these forms $xf(x,y)$, we can use the lift $\sigma_m$
constructed in the odd case. However, since $xf(x,y)$ has vanishing
$y^{n+1}$ term, the image of $\sigma_m$ lies within the set of pairs
$(A,B)$ where $B$ is singular.

The singularity of $B$ introduces additional difficulties with respect to both the
algebraic and the analytic aspects of the proof. On the algebraic side,
the main new problem is that distinguished elements $(A,B)$ with $B$
singular have at least two values for the $Q$-invariant, since they
share at least two different common isotropic $(g+1)$-dimensional
subspaces, where $n=2g+2$. So it is no longer well-defined to impose the condition that
$Q$ is large. Imposing the condition that the maximum value of $Q$ is
large does not yield sufficient savings to prove an analogue of Theorem~\ref{thm:mainestimate}(b). We
thus instead construct a new invariant, termed $q$, such that for all but a negligible number of elements $(A,B)$ in the image of our map $\sigma_m$, the invariant $q$ is the
minimum value taken by $Q$, and it satisfies $q(\sigma_m(xf(x,y)))=\pm m$.

As in the odd degree case, we once again construct  fundamental
domains $\FF_X$ for the action of $G(\Z)$ on $W(\R)$ with height less
than $X$, and partition such a domain into three parts: the main body,
the shallow cusp, and the deep cusp. However, we must now only count integer elements $(A,B)$ where $B$ is singular. The beautiful work
of Eskin and Katznelson \cite{EK} provides asymptotics for the number of
singular symmetric matrices in homogenously expanding domains. But this
work is not directly applicable to our case since we need to estimate
the number of singular symmetric matrices $B$ in {\it skewed}
domains. To achieve this, we provide a simplification of the proof of the
upper bounds in \cite{EK}, at the cost of some extra $\log$ factors,
which gives us a flexible method by which to obtain upper bounds on
the number of singular symmetric matrices in arbitrarily skewed
domains.

\pagebreak 

Accounting for the singularity of the $B$'s introduces complications
in each region of the fundamental domain. In the main body, the lack
of an exact count with a power-saving error term means we cannot
directly apply a Selberg sieve to bound the number of distinguished
elements. Instead, we fiber over the singular matrices $B$ and apply
the Selberg sieve to bound the number of possible $A$'s. This requires
us to prove new density estimates on the number of distinguished
elements $(A,B)$ over $\F_p$, when  $B$ is fixed.

Furthermore, unlike in the odd degree case, we no longer have an automatic
power-saving on the number of pairs $(A,B)\in W(\Z)$ lying in the shallow cusp
of the fundamental domain and where $B$ is singular. As we go closer to
the deep cusp, there are regions in which imposing the
condition that $B$ is singular yields no saving whatsoever. To obtain the
required bounds, we isolate this region of the shallow cusp and prove that
integral elements $(A,B)$ in them either satisfy $\Delta(A,B)=0$ or
$|q(A,B)|$ is small.

Finally, for the deep cusp of $\FF$, we once again use the condition
that the $q$-invariant is large to obtain a power saving. Unlike
the situation with the $Q$-invariant in the odd-degree case, the invariant $q$ in the even degree case behaves more wildly and is much
harder to control. This is because $q$ is not a polynomial in the
coefficients of $W_0$ but rather is a minimum of the different
possible values of $Q$. In fact, there are regions within the
deep cusp where the $q$-invariant of elements $(A,B)$ are not 
small. However, we show that these regions correspond to an archimedean condition on the invariant binary form $f$ of $(A,B)$,
namely, that the discriminant of $f$ is much smaller than is typical
for the height bound on $f$. Separately bounding the number of such
binary forms yields the desired result.

\subsubsection*{Organization of the paper}

This paper is organized as follows. We begin in \S3 by
recalling the arithmetic invariant theory for the representations
$W_n:=2\times\Sym_2(n)$ of $\SL_n$ and $2\otimes g\otimes(g+1)$ of
$\SL_2\times\SL_g\times\SL_{g+1}$. In particular, we define the
fundamental invariants $Q$ and $q$. We then construct our maps from \smash{$\W_{m}^{(2)}$} into $W_n(\Z)$ when $n$ is odd and into $W_{n+1}(\Z)$ when $n$ is even.

The analytic parts of the paper are carried out in \S4--6. In \S4, we prove the tail estimates of Theorem
\ref{thm:mainestimate} for odd degrees $n$ using geometry-of-numbers
techniques. In \S5, we carry out the necessary groundwork to count the
number of singular symmetric matrices that lie in skewed
domains. Using these results, we prove the tail
estimates for even degrees $n$ in \S6, completing the proof of Theorem
\ref{thm:mainestimate}.  In \S7, we deduce the main results, Theorems~1--4, from the tail  estimates using a squarefree sieve, although the exact constants occurring in Theorems~1 and 2 remain conditional upon certain local density computations. Finally, in the Appendix, we compute the local densities of integral binary
$n$-ic forms whose discriminants are indivisible by~$p^2$
(resp., whose associated rings are maximal at $p$), thereby completing the proofs of
Theorems \ref{polydisc2} and~\ref{polydiscmax2}.

\section{Invariant theory on spaces associated to binary $n$-ic forms}

Fix a positive integer $n$ and consider the space
$V_n=\Sym^n(2)$ of binary $n$-ic forms of degree $n$. The group
$\SL_2$ acts on $V_n$ via linear change of variables: we have
$\gamma\cdot f(x,y):=f((x,y)\cdot\gamma)$ for $\gamma\in\SL_2$ and
$f\in V_n$. 

Let $W_n=2\otimes\Sym_2(n)$ denote the space of pairs of
$n\times n$ symmetric matrices $(A,B)$. The group
$\SL_2\times\SL_n$ acts on $(A,B)$ via
\begin{equation*}
(\gamma_2,\gamma_n)\cdot (A,B)=(\gamma_n A\gamma_n^t,\gamma_n B\gamma_n^t)\cdot\gamma_2^t.
\end{equation*}
There is a natural map  $W_n\to V_n$ given by
\begin{equation}\label{eqres}
  \begin{array}{rcl}
    (A,B)&\mapsto & f_{A,B}:= (-1)^{n(n-1)/2}\det(Ax-By),
  \end{array}
\end{equation}
sending an element of $W_n$ to its {\it invariant binary $n$-ic form}.
The ring of $\SL_n(\C)$-invariant polynomials on $W_n(\C)$ is freely
generated by the coefficients of the invariant binary $n$-ic form.

\subsection{Arithmetic invariant theory for the representation $2\otimes\Sym_2(n)$ of $\SL_n$}\label{sec:2.1}

First, let $n=2g+1$ be an odd integer with $g\geq 1$. We recall some of the arithmetic invariant theory of the representation $W:=W_n$ of
 $\SL_n$ and its map (\ref{eqres}) to $V:=V_n;$ see \cite{BGW} for
more details. 

Let $k$ be a field of characteristic not $2$.  For a binary $n$-ic form $f(x,y)=a_0x^n + \cdots + a_ny^n\in V(k)$ with
$\Delta(f)\neq0$ and $a_0\neq 0$, let $C_f$ denote the smooth
hyperelliptic curve $z^2=f(x,y)y$ of genus $g$ viewed as a curve in
the weighted projective space $\P(1,1,g+1)$. Let $J_f$ denote the
Jacobian of $C_f$.  Then the stabilizer of an element $(A,B)\in W(k)$
with invariant binary form $f(x,y)$ is isomorphic to $J_f[2](k)$. 
The
set of $\SL_n(k)$-orbits on $W(k)$ with invariant binary form
$f(x,y)$ maps injectively into $H^1(k,J_f[2])$. An element $(A,B)$ (or an
$\SL_n(k)$-orbit) is {\it distinguished} if
$\Delta(A,B)\neq 0$ and there exists a $g$-dimensional subspace defined
over $k$ that is isotropic with respect to both $A$ and $B$. If
$(A,B)$ is distinguished, then its $\SL_n(k)$-orbit corresponds to the identity element
of $H^1(k,J_f[2])$, and the set of these $g$-dimensional subspaces is
in bijection with $J_f[2](k)$.

Let $W_{0}\subset W$ be the subspace of pairs of matrices
whose top left $g\times g$ blocks are zero. Then elements $(A,B)$ in
$W_{0}(k)$ with nonzero discriminant are all distinguished since the
$g$-dimensional subspace $Y_g$ spanned by the first $g$ basis vectors
is isotropic with respect to both $A$ and $B$. Moreover, every
distinguished element of $W(k)$ is $\SL_n(k)$-equivalent to some
element in $W_{0}(k)$ since $\SL_n(k)$ acts transitively on the set
of $g$-dimensional subspaces of $\P^{n-1}(k)$. Let $G_0$ be the
maximal parabolic subgroup of $\SL_n$ consisting of elements $\gamma$
that preserve $Y_g$.  Elements of $W_{0}$ have block matrix form
\begin{equation}\label{eq:G_01}
  (A,B)=\left(\Bigl(\begin{array}{cc}0 & A^\top\\ (A^\top)^t & A_1\end{array}\Bigr),
    \Bigl(\begin{array}{cc}0 & B^\top\\ (B^\top)^t & B_1\end{array}\Bigr)\right),
\end{equation}
where $A^\top$, $B^\top$ are $g\times (g+1)$ matrices and $A_1$, $B_1$
are $(g+1)\times(g+1)$-symmetric matrices.  Meanwhile, elements of
$G_0$ have the block matrix form
\begin{equation}\label{eq:G_0}
\gamma=\Bigl(\begin{array}{cc}\gamma_1 & 0\\ n & \gamma_2
\end{array}\Bigr)\in\Bigl(\begin{array}{cc}\GL_g & 0\\ M_{(g+1)\times g} & \GL_{g+1}
\end{array}\Bigr).
\end{equation}
An element $\gamma\in G_0$ acts on the top right $g\times (g+1)$ block
of elements of $W_{0}$ by
\begin{equation*}\gamma(A^\top,B^\top) =
  (\gamma_1A^\top\gamma_2^t,\gamma_1B^\top\gamma_2^t)
\end{equation*}
where we use the superscript ``top'' to denote the top right $g\times
(g+1)$ block of an $n\times n$ symmetric matrix. The action of $G_0$ on $W_{0}$ restricts to an action on
the space $U_g:=2\otimes g\otimes(g+1)$ of pairs of $g\times
(g+1)$-matrices, Moreover, the unipotent radical
$M_{(g+1)\times g}$ of $G_0$ acts trivially on $U_g$. We study the
invariant theory for this action more closely in the next subsection.

We will also need some results in the case when $n=2g+2$ is even in
Section \ref{sec:moniceven} (specifically in the proof of Lemma
\ref{lemndist}). Let $f(x,y) = a_0x^n + \cdots + a_ny^n\in V(k)$
with $\Delta(f)\neq 0$ and $a_0\neq 0$. Let $L = k[x]/(f(x))$. Let
$V_f(k)$ denote the set of $(A,B)\in W_n(k)$ with $f_{A,B} =
f(x,y)$. Then $V_f(k)$ is nonempty if and only if $f\in
k^{\times2}N_{L/k}(L^\times)$. Note in particular that if $f(x,y)\in
V(\R)$ is negative definite, so that $L = \R[x]/(f(x))\simeq
\C^{n/2}$ and $a_0<0$, then $V_f(\R)$ is empty. 
On the other hand, if
$k$ is a finite field of characteristic not $2$, then $V_f(k)$ is always nonempty and the number of $\SL_n(k)$-orbits equals the number of even degree factorizations of $f(x,y)$ over $k$.

\subsection{The representation $2\otimes g\otimes (g+1)$
  of $\SL_2\times\SL_g\times\SL_{g+1}$ and the $Q$-invariant}\label{sQ}

In this section, we collect some algebraic facts about the
representation $U_g;=2\otimes g\otimes(g+1)$ of the group
$H_g:=\SL_2\times\SL_g\times\SL_{g+1}$. We start with the following proposition.
\begin{proposition}
The representation $U_g$ of $\G_m\times H_g$ is {\it prehomogeneous}, i.e.,
the action of $\G_m\times H_g$ on $U_g$ has a single Zariski open
orbit. Furthermore, the stabilizer in $H_g(\C)$ of an element in the
open orbit of $U_g(\C)$ is isomorphic to $\SL_2(\C)$.
\end{proposition}
\begin{proof}
We prove this by induction on $g$.  The assertion is clear for $g=1$,
where the representation is that of $\G_m\times \SL_2\times \SL_2$ on
$2\times 2$ matrices; the single relative invariant in this case is
the determinant, and the open orbit consists of nonsingular matrices.
For higher $g$, we note that $U_g$ is a {\it castling transform} of
$U_{g-1}$ in the sense of Sato and Kimura~\cite{SK}. As a result, the
orbits of $\G_m\times \SL_2\times\SL_g\times\SL_{g-1}$ on $2\otimes g
\otimes (g-1)$ are in natural one-to-one correspondence with the orbits
of $\G_m\times\SL_2\times\SL_g\times\SL_{g+1}$ on $2\otimes g\otimes
(2g-(g-1))=2\otimes g\otimes(g+1)$, and under this correspondence, the
open orbit in $U_{g-1}$ maps to an open orbit in $U_g$
(cf. \cite{SK}).  Thus all the representations $U_g$ for the action of
$\G_m\times H_g$ are prehomogeneous.

Note that castling transforms preserve stabilizers over $\C$. Since
the generic stabilizer for the action of $H_1(\C)$ on $U_1(\C)$ is
clearly isomorphic to $\SL_2(\C)$, it follows that this remains the
generic stabilizer for the action of $H_g(\C)$ on $U_g(\C)$ for all
$g\geq 1$.
\end{proof}

Since castling transforms also preserve polynomial invariants and their irreducibility (\cite[Proposition 18]{SK}), it
follows that the ring of polynomial invariants for this action of $H_g$ on $U_g$ is generated by an irreducible polynomial. We now give an explicit description of
this invariant.

Write an element in $U_g=2\times g\times
(g+1)$ as a pair $(A^\top,B^\top)$ of $g\times(g+1)$ matrices.  For
$1\leq i\leq g+1$, let $A_i$ and $B_i$ denote the $g\times g$-matrices
obtained from $A^\top$ and $B^\top$, respectively, by deleting the
$i$th column.  Define the binary $(g+1)$-ic form $f_i(x,y)$ to be
$(-1)^{i+1}\det(A_ix-B_iy)$. Consider the $(g+1)\times(g+1)$ matrix
$C$ whose $(i,j)$-entry is the $j$th-coefficient of $f_i(x,y)$. Taking
the determinant of $C$ yields a polynomial $Q=Q(A^\top,B^\top)$ in the
coordinates of $U_g$. The polynomial~$Q$ is the {\it hyperdeterminant} of the $2\times g \times (g+1)$ matrix $(A^\top,B^\top)$ (cf.\ \cite[Theorem 3.18]{GKZ}).  As a
consequence, it is irreducible and invariant under the action of $H_g$
on $U_g$ and thus generates the ring of polynomials for the action of $H_g$ on $U_g$.

Let $n=2g+1$ again be an odd integer.
We return to the representation $W_{0}$ of $G_0$. Given an element
$(A,B)\in W_{0}$, recall that we obtain an element $(A^\top,B^\top)\in
U_g$ by taking the top right $g\times(g+1)$ blocks of $A$ and $B$.  We
define the $Q$-{\it invariant} of $(A,B)\in W_{0}$ as the
$Q$-invariant of~$(A^\top,B^\top)$:
\begin{equation}\label{eqQB}
Q(A,B):=Q(A^\top,B^\top).
\end{equation}
Then the $Q$-invariant is a relative invariant for $G_0$. More
precisely, for any $\gamma\in G_0$ in the block matrix form
\eqref{eq:G_0}, we have
\begin{equation}\label{eq:weightG_0}
Q(\gamma\cdot (A,B)) = \det(\gamma_1)^{g+1}\det(\gamma_2)^gQ(A,B)=\det(\gamma_1)Q(A,B),
\end{equation}
since $\det(\gamma_1)\det(\gamma_2)=1$.  If $\gamma\in G_0(\Z)$, then
we have $\det(\gamma_1)=\det(\gamma_2)=\pm1$. Hence the absolute value
$|Q|$ of $Q$ is an invariant for the action of $G_0(\Z)$ on
$W_{0}(\Z)$.

\subsection{Divisibility properties of $\Delta$ when restricted to $W_{0}$}\label{sec:2.3}

Let $n=2g+1$ be an odd integer. Write the coordinates on 
$W_{0}$ as   $a_{ij},b_{ij}$ with $i,j$ in the appropriate ranges. Let~$R$ denote the ring of regular functions of $W_0$ over
$\Z$, i.e., $R=\Z[W_0]=\Z[a_{ij},b_{ij}]$. 
Consider the discriminant polynomial $\Delta\in R$
given by $\Delta(A,B):=\Delta(f_{A,B})$. In this section, we
prove that $Q^2\mid\Delta$ as polynomials in $R$, along with another
useful divisibility result.

Let $Z$ be the closed subvariety of $W_{0}$ consisting of elements
$(A,B)$ with $\Delta(A,B)=0$, and let $Y\subset Z$ denote the closed
subvariety of $W_{0}$ consisting of elements $(A,B)$ such that
$f_{A,B}$ is either divisible by the cube of a binary form with
degree $\geq 1$ or the square of a binary form with degree $\geq
2$. Both of these varieties $Y$ and $Z$ are defined over $\Z$ and are
clearly $\SL_2\times G_0$-invariant.  

Our first
result states that the variety in $W_{0}$ cut out by $Q=0$ does not
lie in $Y$.
\begin{proposition}\label{prop:notstrong}
 Let $(A,B) = ((a_{ij})_{ij},(b_{ij})_{ij}) \in W_0(R)$ be the generic element.
 Then
\begin{equation*}
  (A,B){\rm{ \;mod\; }}Q\not\in
  Y(R/(Q)).
\end{equation*}
\end{proposition}
\begin{proof}
Fix an odd prime $p$. Let $f(x,y)$ be an element of $V(\Z)$, such
that the reduction of $f(x,y)$ modulo $p$ factors as $x^2h(x,y)$, where
$h$ is irreducible. In particular, $f(x,y)$ mod $p$ is not divisible
by either the cube of a binary form with degree $\geq 1$, or the
square of a binary form with degree $\geq 2$. Let $(A_f,B_f)\in
W_{0}(\Z)$ be an element with invariant binary $n$-ic form equal to $f$ and $Q(A_f,B_f)
= p$. Such an element $(A_f,B_f)$ is constructed in the next
subsection.

Let $\pi:R\rightarrow \Z$ denote the specialization map assigning
integer values to $a_{ij},b_{ij}$ such that
\begin{equation*}
  \pi(A,B) = (A_f,B_f).
\end{equation*}
Then $\pi(Q) = p$ and so $\pi$ induces a map $R/(Q) \rightarrow
\F_p$. Since $(A_f,B_f)\mbox{ mod }p\notin Y(\F_p)$, we see that
$(A,B)\mbox{ mod }Q\notin Y(R/(Q))$.
\end{proof}

The next lemma, which follows from a direct computation, gives the
$Q$-invariant for elements in $W_{0}$ having a specific form.
\begin{lemma}\label{lemQcomp}
Let $k$ be a field and let $(A,B)\in W_{0}(k)$ be an element such
that the top right $g\times(g+1)$ blocks of $(A,B)$ are of the
following form:
\begin{equation}\label{eqQcomp}
(A^\top,B^\top)=\left(\left(\begin{array}{cccccc}
0&0&\cdots&0&0&a_1\\0&&&&a_2&*\\0&&&a_3&*&*\\ \vdots&&
\reflectbox{$\ddots$}&&\vdots&\vdots\\0&a_g&\cdots&*&*&*\end{array}
\right),\left(\begin{array}{cccccc}0&
\cdots&0&0&b_1&0\\&&&b_2&0&0\\&&b_3&&0&0\\&
\reflectbox{$\ddots$}&&&\vdots&\vdots\\b_g&&&&0&0\end{array}\right)\right).
\end{equation}
Then
\begin{equation*}
Q(A,B)=\pm a_1^ga_2^{g-1}\cdots a_gb_1b_2^2\cdots b_g^g.
\end{equation*}
\end{lemma}

Next, we have the following proposition that gives a normal form
for elements $(A,B)\not\in Y$ whose $Q$-invariant is $0$.
\begin{proposition}\label{prop:normal}
Let $k$ be a field. Let $(A,B)$ be an element of $W_{0}(k)\backslash
Y(k)$ such that $Q(A,B)=0$. Then $(A,B)$ is $\SL_2(k)\times
G_0(k)$-equivalent to an element of the form $(A',B')$ where the top
right $g\times(g+1)$ blocks of $A'$ and $B'$ are given by
\begin{equation}\label{eq:normal}
(A'^\top,B'^\top) =
\left(\left(\begin{array}{cccccc}0&0&
\cdots&0&0&a_1\\0&&&&a_2&*\\0&&&a_3&*&*\\ \vdots&&
\reflectbox{$\ddots$}&&\vdots&\vdots\\0&a_g&\cdots&*&*&*
\end{array}\right),\left(\begin{array}{cccccc}
0&\cdots&0&0&0&0\\&&&b_2&0&0\\&&b_3&&0&0\\&
\reflectbox{$\ddots$}&&&\vdots&\vdots\\b_g&&&&0&0\end{array}
\right)\right),
\end{equation}
where $a_1,\ldots,a_g,b_2,\ldots,b_g\in k^\times.$ In the displayed
matrices above, any empty entry is $0$.
\end{proposition}
\begin{proof} 
The action of $G_0(k)$ allows us to perform simultaneous row
operations and simultaneous column operations on $(A^\top,B^\top)$. As
a first step, we perform column operations to ensure that the
rightmost column of $B^\top$ is $0$. Next, recall that the
$Q$-invariant of $(A,B)$ is the determinant of the $(g+1)\times (g+1)$
matrix $C$, whose rows come from the coefficients of the $g\times g$
minors of $A^\top x-B^\top y$. It follows that row operations on
$(A^\top,B^\top)$ leave $C$ unchanged, while adding $\alpha$ times the
$i$-th columns of $A^\top,B^\top$ to the $j$-th column has the effect
of adding $\alpha$ times the $j$-th row of $C$ to the $i$-th row of
$C$ and leaving the rest unchanged. Since
$\det(C)=Q(A^\top,B^\top)=0$, it follows that by adding multiples of
the last columns of $A^\top,B^\top$ to the other columns, we may
assume that the last row of $C$ is $0$. Denoting the $g\times g$
matrices obtained by removing the last columns of $A^\top$ and
$B^\top$ by $M$ and $N$, respectively, we have $\det(Mx - Ny) = 0$.

We next claim that by performing simultaneous row and column
operations on $(M,N)$, we may bring $M$ and $N$ in the form of the
first $g$ columns of $A'^{\,\top}$ and $B'^{\,\top}$, respectively, for
$(A'^{\,\top},B'^{\,\top})$ as given in \eqref{eq:normal} with $b_i\neq 0$
for all $2\leq i\leq g$.  Since $\det(M)=0$, after
appropriate column operations we may assume that the first column of
$M$ is~$0$. Now the first column of $N$ cannot be identically $0$ for
otherwise, the invariant binary form of $(A,B)$ has a factor of the
form $h(x,y)^2$ with $\deg h = g$, contradicting $(A,B)\notin
Y(k)$. By applying row operations, we may ensure that the bottom left
entry of $N$ is $b_g\neq 0$ and the rest of the first column of $N$ is~$0$. We then use this nonzero cofficient $b_g$ to clear out the rest
of the bottom row of $N$ (without changing~$M$).

Let $M_1$ and $N_1$
denote the top right $(g-1)\times (g-1)$ block of $M$ and $N$. Then
$\det(Mx - Ny) = (-1)^{g}b_g y\det(M_1x - N_1y).$ Hence $\det(M_1x -
N_1y) = 0$ and the first column of $M_1$ can be made $0$. As in the
previous case, all the coefficients of the first column of $N_1$ can
be made $0$ except for the bottom left entry, which is $b_{g-1}\neq
0$. We then clear out the bottom row of $N_1$ as before. Proceeding in
this way, we transform the first $g-1$ columns of $M$ and $N$ to be in
the required form. Since the $b_i$'s are nonzero for $2\leq i\leq g$,
and since $\det(Mx-Ny)=0$, it follows that the top right coefficients
of $M$ and $N$ are $0$, completing the proof of the claim.

Note that this transformation of $M$ and $N$ did not change the last
column of $B^\top$, which remains $0$.  Thus to complete the
proof of Proposition \ref{prop:normal}, it remains to show that
$a_i\neq 0$ for $1\leq i\leq g$. Since the first row and column of
$B'$ are $0$, we see that $x^2a_1^2\mid f_{A',B'}$. Hence $a_1\neq 0$. Suppose for  contradiction that $i=2,\ldots,g$ is
the smallest index such that $a_i = 0$. Then we may clear out the
$i$-th row of $A'$ using the second up to the $(i-1)$-th rows of
$A'$. That is, $(A',B')$ is $\SL_n(k)$-equivalent to some
$(A'',B'')$ where the only nonzero entries in the $i$-th row and the
$i$-th column of $A''$ appear in the last entry. This allows us to
factor out an extra factor of $y^2$ in $\det(A''x -
B''y)=\pm f_{A,B}$, contradicting the assumption that
$(A,B)\notin Y(k)$ since we already had $x^2\mid f_{A,B}$.
\end{proof}

We are now ready to prove that $Q^2\mid \Delta$:
\begin{theorem}\label{thQdivDelta}
We have $Q^2\mid \Delta$ in $\Z[W_{0}]$.
\end{theorem}
\begin{proof}
Let $(A,B)\in W_{0}(R)$ be the generic element. We begin by proving
that $(A,B)\in Z(R/(Q))$, or equivalently that $Q\mid\Delta$ in
$R$. Let $(\bar{A},\bar{B})\in W_{0}(R/(Q))$ denote the
reduction of $(A,B)$ mod $Q$ and let $F$ denote the field of
fractions of $R/(Q)$. By Proposition~\ref{prop:notstrong}, we know
$(\bar{A},\bar{B})\notin Y(F)$. Since $Q(\bar{A},\bar{B}) = 0$, by
Proposition \ref{prop:normal}, there exists $\gamma\in \SL_2(F)\times
G_0(F)$ such that $\gamma(\bar{A},\bar{B})=(A',B')$ where
$(A'^{\,\top},B'^{\,\top})$ is of the form \eqref{eq:normal}. The invariant
binary form of $(A',B')$ has a factor of $x^2$, and
so $(A',B')\in Z(F)$.  Since $Z$ is $\SL_2\times
G_0$-invariant, we see that $(\bar{A},\bar{B})\in Z(F)$. 

Since $Q\mid\Delta$ in $R$, there exists
an element $\delta\in R$ such that $\Delta=Q\delta$. Let $Z_1$ denote
the closed subvariety of $W_{0}$ cut out by $\delta$. It now
suffices to prove that $Q\mid\delta$ or, equivalently, that the generic
element $(A,B)$ belongs to $Z_1(R/Q)$.  We claim that for any field
$k$, and every element $(A,B)\in W_{0}(k)$ such that
$(A^\top,B^\top)$ has the form \eqref{eq:normal}, we have
$\delta(A,B)=0$. Indeed, let $(A,B)$ be such an element. Let
$(A^{(\epsilon)},B^{(\epsilon)})\in W_{0}(k[\epsilon])$ be such that
$A^{(\epsilon)}=A$, the $(1,n-1)$-entry and the $(n-1,1)$-entry of
$B^{(\epsilon)}$ equal $\epsilon$, and the other coefficients of
$B^{(\epsilon)}$ are the same as those of $B$.  By Lemma
\ref{lemQcomp}, we have  
\begin{equation*}
  Q(A^{(\epsilon)},B^{(\epsilon)})=\pm
  \epsilon\, a_1^ga_2^{g-1}\cdots a_gb_2^2\cdots b_g^g.
\end{equation*}
Moreover,  $\epsilon^2$ divides the
$y^n$-coefficient of $f_{A^{(\epsilon)},B^{(\epsilon)}}$ and 
$\epsilon$ divides the $xy^{n-1}$-coefficient of
$f_{A^{(\epsilon)},B^{(\epsilon)}}$. Hence 
$\epsilon^2\mid\Delta(A^{(\epsilon)},B^{(\epsilon)})$, which implies
(since $\epsilon^2\nmid Q(A^{(\epsilon)},B^{(\epsilon)})$) that
$\epsilon\mid\delta(A^{(\epsilon)},B^{(\epsilon)})$. Since $(A,B)$ is
obtained from $(A^{(\epsilon)},B^{(\epsilon)})$ by setting
$\epsilon=0$, we have $\delta(A,B)=0$.  We have proven
that the generic element $(A,B)\in W_{0}(R)$ belongs to
$Z_1(R/(Q))$. Therefore, $Q\mid\delta$.
\end{proof}

We end this section with another divisibility result for $\Delta$,
which will be used in \S\ref{sec:moniceven}.

\begin{proposition}\label{propBsubdiv}
We have $\det(A^\top (A^\top)^t)\det(B^\top (B^\top)^t) \mid \Delta$ as elements in $\Z[W_{0}]$. 
\end{proposition}

\begin{proof}
It suffices to prove that $\det(B^\top (B^\top)^t)$ divides $\Delta$ in
$\Z[W_{0}]$. Suppose $(A,B)\in W_{0}(\C)$ with $\det(B^\top
(B^\top)^t) = 0$. Then $B^\top$ does not have full rank. Hence there
exists some nonzero $v\in\text{Span}_\C\{e_1,\ldots,e_g\}$ such that
$Bv = 0$. However, any such $v$ is isotropic with respect to $A$. As a
result, $\Delta(A,B) = 0$. Thus, by the  Nullstellensatz, 
$\det(B^\top (B^\top)^t) \mid c \Delta^{d}$ in $\Z[W_{0}]$ for some
nonzero integer $c$ and positive integer $d$.

Define $P_g\in \Z[M_{g\times(g+1)}]$ by $P_g(M) = \det(MM^t)$. For the
purpose of proving Proposition~\ref{propBsubdiv}, it suffices to prove
that $P_g$ is squarefree in $\Z[M_{g\times(g+1)}]$. We proceed by
induction on $g$. Denote the $(i,j)$-entry of any $M\in
M_{g\times(g+1)}$ by $u_{ij}$. When $g = 1$, we have $P_1 = u_{11}^2 +
u_{12}^2$, which is squarefree in $\Z[u_{11},u_{12}]$. For general
$g\geq2$, consider
$$M = \begin{pmatrix} 
u_{11} & \cdots & u_{1\,g-1} & u_{1\,g} & 0\\
\vdots & \ddots & \vdots & \vdots & \vdots \\
u_{g-1\, 1} & \cdots & u_{g-1\,g-1} & u_{g-1\,g} & 0\\
0&\cdots&0&\alpha&\beta
\end{pmatrix}.$$
Then 
$$\det(MM^t) = \beta^2 P_{g-1} + \alpha^2 D_{g-1}^2$$ where $D_{g-1}$
is the determinant of the top left $(g-1)\times(g-1)$ block of
$M$. Any square factor of $\det(MM^t)$ must be a common square
factor of $P_{g-1}$ and $D_{g-1}^2$, which can only be $\pm 1$ since
$P_{g-1}$ is squarefree by induction. We have shown that $P_g$ is
squarefree even after setting certain variables to $0$. Therefore,
$P_g$ is squarefree in $\Z[M_{g\times(g+1)}]$.
\end{proof}

\vspace{-.1in}

\subsection{Embedding $\W_{m,n}^\w$ into $W_n(\Z)$, for $n$ odd}\label{sembedodd}

Let $n=2g+1$ be an odd integer, and set $W=W_n$.
For an odd squarefree integer $m>0$, let $\W_m^\w=\W_{m,n}^\w$ denote the set of 
integer binary forms whose discriminants are weakly divisible by $p^2$
for every prime factor $p$ of $m$. Fix an element $f(x,y)\in\W_m^\w$. Then just as shown in
\cite[\S3.2]{BSWsq}, there exists an $\SL_2(\Z)$-change of variable
such that $f((x,y)\gamma)$ has the form
\begin{equation}\label{eqflift}
  f((x,y)\gamma) = m^2b_0x^n + mb_1x^{n-1}y + \cdots + b_ny^n
\end{equation}
for some integers $b_0,\ldots,b_n$ and where $m$ and $b_0$ are coprime.

Consider the following pair of matrices:
\begin{equation*}
 A =
 \left(\begin{array}{ccccccc}&&&&&&1\\&&&&&\iddots&\\ &&&&m&&\\ &&&c_0&&&
   \\ &&m&&c_{2}&&\\ &\iddots&&&&\ddots&\\ 1&&&&&&c_{n-1} \end{array}\right),\;\;
 B =
 \left(\begin{array}{ccccccc}&&&&&1&0\\&&&&\iddots&\iddots&\\ &&&1&r&&\\ &&1&c_1&&&
   \\ &\iddots&r&&c_3&&\\ 1&\iddots&&&&\ddots&\\ 0&&&&&&c_n \end{array}\right).
\end{equation*}
Here the dots on the antidiagonal of $A$ are all $1$. We claim that
$c_i,r$ can be chosen to be integers so that $(-1)^g\det(xA - yB) =
f((x,y)\gamma)$. It is clear that $c_0 = b_0$ and $2mrc_0 +
m^2c_1 = -mb_1$. Choose $r\in\Z$ such that $m\mid
2rc_0+b_{1}$; this then determines $c_{1}$. It is then not hard to check
that the coefficient of $x^{n-i}y^i$ in $(-1)^g\det(xA - yB)$ is of
the form $(-1)^ic_i + L(c_0,\ldots,c_{i-1})$ where $L$ is a linear
form with coefficients in $\Z[r]$. The existence of integers $c_2,\ldots,c_n$ now follows by induction.

Set $\sigma_m(f)=\sigma_{m,n}(f)$ to be the element $(A_f,B_f)$ such
that $$\begin{pmatrix} A_f \\ -B_f \end{pmatrix} =
\gamma^{-1} \begin{pmatrix} A \\ -B \end{pmatrix}.$$ Then
$f_{\sigma_{m}(f)} = f.$ Next, we note that $(A,B)$ and thus
$(A_f,B_f)$ are in $W_{0}(\Z)$ and from Lemma \ref{lemQcomp}, we
obtain that $|Q|(A,B) = m$. Since $Q$ is $\SL_2$-invariant, we
conclude that $$|Q|({\sigma_{m}(f)})=m.$$ We have proven
the following theorem.
\begin{theorem}\label{keymap}
Let $m>0$ be a  squarefree integer. There exists a map
$\sigma_{m}:\U_{m}^\w\to W_{0}(\Z)$ such~that $$f_{\sigma_{m}(f)}=f,\qquad |Q|(\sigma_{m}(f)) = m$$
for every $f\in \U_{m}^\w$.
\end{theorem}

We will later use the image of $\sigma_{1}$ as a fundamental set for the action of $\SL_n(\R)$ on the set of distinguished elements of $W(\R)$.
We now extend the function $|Q|$ to the set of distinguished elements
of $W(\Z)$ having irreducible invariant binary form. Suppose that
$(A,B)$ is a distinguished element of $W(\Z)$. Then there is a
$g$-dimensional subspace $X$ isotropic with respect to $A$ and
$B$. Let $\Lambda = X\cap \Z^n$ be the primitive lattice in $X$.
There exists an element $\gamma$ in $\SL_n(\Z)$, unique up to left
multiplication by an element in $G_0(\Z)$, such that $\Lambda =
\gamma^t\cdot\,\text{Span}_\Z\{e_1,\ldots,e_g\}$, where $e_1,\ldots,e_n$ is  the standard basis of $\Z^n$. Then 
$\gamma\cdot(A,B)\in W_{0}(\Z)$, and we can thus define the
$|Q|$-invariant on the triple $(A,B,\Lambda)$ by
$$|Q|(A,B,\Lambda):=|Q|(\gamma\cdot(A,B)).$$ That is, we
complete an integral basis of $\Lambda$ to an integral basis of
$\Z^n$ with respect to which the pair $(A',B')$ of Gram matrices for
the quadratic forms defined by $A$ and $B$ lies inside $W_0(\Z)$,
and we define $|Q|(A,B,\Lambda)$ to be $|Q|(A',B').$ 

We end with the
following result that will be crucial in Section 4.
\begin{proposition}\label{propuniqueQodd}
Let $n=2g+1$ be an odd integer with $n\geq 3$. Let $m$ be an
odd positive squarefree integer. Let $f(x,y)\in\U_{m}^{(2)}$ be an
irreducible integral binary $n$-ic form. Let $(A,B)$ be any element in
$\SL_n(\Z)\cdot \sigma_{m}(f)$. Then there is a unique primitive
$g$-dimensional lattice $\Lambda$ that is isotropic with respect to
both $A$ and $B$. Moreover, $|Q|(A,B):=|Q|(A,B,\Lambda)=m$. In
particular, if $f(x,y)\in \U_{m}^{(2)}\cap \U_{m'}^{(2)}$ is
irreducible where $m$ and $m'$ are distinct odd positive squarefree
integers, then $\sigma_{m}(f(x,y))$ and $\sigma_{m'}(f(x,y))$ are
not $\SL_{n}(\Z)$-equivalent.
\end{proposition}
\begin{proof}
Let $C_f$ denote the smooth hyperelliptic curve $z^2=f(x,y)y$ of genus
$g$ viewed as a curve in the weighted projective space $\P(1,1,g+1)$, 
and let $J_f$ denote its Jacobian. Since $(A,B)$ is
$\SL_n(\Z)$-equivalent to $\sigma_{m}(f)$, it follows that $(A,B)$
is distinguished. Thus the set of common isotropic $g$-dimensional
subspaces of $A$ and $B$ over $\Q$ is in bijection with
$J_f[2](\Q)$. Since $f$ is irreducible, we have
$J_f[2](\Q)=1$. Therefore, there is a unique primitive $g$-dimensional
lattice $\Lambda$ which is isotropic with respect to both $A$ and $B$.

Let $\gamma\in\SL_n(\Z)$ be an element such that $\gamma(A,B) = \sigma_{m}(f)
=: (A_f, B_f)\in W_{0}(\Z)$. Since we know that
$\text{Span}_\Z\{e_1,\ldots,e_g\}$ is a primitive $g$-dimensional
lattice isotropic with respect to $A_f$ and $B_f$, we see that
$\gamma^t\cdot\,\text{Span}_\Z\{e_1,\ldots,e_g\}$ is a primitive
$g$-dimensional lattice isotropic with respect to $A$ and $B$. By
uniqueness, it follows that $\Lambda =
\gamma^t\cdot\,\text{Span}_\Z\{e_1,\ldots,e_g\}$, and so by
definition,  $|Q|(A,B,\Lambda)=|Q|(\sigma_{m}(f))=m$, where
the final equality is Theorem \ref{keymap}.
\end{proof}

\vspace{-.1in}

\subsection{Embedding $\U_{m,n}^{(2),\,\gen}$ into $W_{n+1}(\Z)$, for $n$ even}\label{sembedeven}

Suppose now that $n = 2g+2$ is even with $g\geq1$. For an odd squarefree integer $m>0$, let $\W_{m,n}^\w$ denote the set of 
integer binary forms having discriminant weakly divisible by $p^2$
for every prime factor $p$ of $m$. Let
$\W_{m,n}^{(2),\,\gen}\subset \W_{m,n}^{(2)}$
consist of those  $f(x,y)$ with $f(0,1)$ coprime to $m$. Since
$\Delta(xf(x,y)) = \Delta(f(x,y))f(0,1)^2$, we see that if $f(x,y)\in
\W_{m,n}^{(2),\,\gen}$, then $xf(x,y)\in \W_{m,n+1}^{(2)}$. We 
define $\sigma_{m,n}:\W_{m,n}^{(2),\,\gen}\rightarrow W_{n+1}(\Z)$ via
$\sigma_{m,n}(f) = \sigma_{m,n+1}(xf)$. For the rest of this subsection, we drop subscripts and denote $\sigma_{m,n}$ by $\sigma_m$, $\W_{m,n}^{\w,\,\gen}$ by $\W_m^{\w\,\gen}$, and $W_{n+1}$ by $W$.

We now define the finer $q$-invariant. Let $f\in
\W_{m}^{(2),\,\gen}$ and suppose $(A,B) = \sigma_{m}(f)$. Then
$B$ is singular since $xf(x,y)$ has vanishing $y^n$-term. Moreover,
since $\Delta(xf)\neq 0$, the kernel of $B$ has dimension exactly $1$
and is not isotropic with respect to $A$ (see Lemma \ref{lem:disc}).
Fix an integral domain~$D$. Let $W_{1}(D)$ be the subset of
$W(D)$ consisting of pairs $(A,B)$ of symmetric $(n+1)\times
(n+1)$ matrices satisfying the following conditions:
\begin{itemize}
\item[{\rm (a)}] The top left $(g+1)\times (g+1)$ block of $A$ is $0$.
\item[{\rm (b)}] The top left $(g+2)\times (g+2)$ block of $B$ is $0$ (implying that $B$ is singular).
\item[{\rm (c)}] The kernel of $B$ has dimension exactly $1$ (over the fraction field of $D$) and is not isotropic with respect to $A$.
\end{itemize}

Take any $(A,B)\in W_{1}(D)$. Conditions (a) and (c) imply that
the first $g+1$ columns of $B$ are linearly independent over the
fraction field of $D$. Let $B'$ denote the top right $(g+1)\times
(g+1)$ block of $B$. Since $B$ is symmetric, we see that $B'$ is
nonsingular. It is now easy to see from the definition of the
$Q$-invariant that as polynomials in the coordinates of
$W_{1}(D)$, we have
$$\det(B') \mid Q(A^\top, B^\top).$$
We define the quotient to be the $q$-invariant of $(A,B)$:
\begin{equation}\label{eq:qinv}
q(A,B) := Q(A^\top, B^\top)/\det(B').
\end{equation}

Let $G_1(D)$ denote the subgroup of $\SL_{n+1}(D)$ preserving
$W_{1}(D)$. Then elements of $G_1(D)$ have the following block
matrix form
\begin{equation}\label{eq:G_1}
\gamma=\left(\begin{array}{ccc}\gamma_1 & 0 & 0\\ n_1 & \gamma_2 & 0\\ n_2 & n_3 & \gamma_3
\end{array}\right)\in\left(\begin{array}{ccc}\GL_{g+1} & &\\ M_{1\times (g+1)} & \GL_1 &\\ M_{(g+1)\times (g+1)} & M_{(g+1)\times 1} & \GL_{g+1}
\end{array}\right).
\end{equation}
It is easy to check that for any $(A,B)\in W_{1}(D)$,
\begin{equation}\label{eqqweight1}
q(\gamma(A,B)) = \det(\gamma_1)(\det(\gamma_1)\det(\gamma_3))^{-1}q(A,B) = \det(\gamma_1)\gamma_2\,q(A,B).
\end{equation}

We now consider the situation over $\Z$. Let $(A,B)\in W(\Z)$
be a distinguished element having nonzero discriminant such that $B$
is singular. Let $X$ denote a common isotropic $(g+1)$-dimiensional
subspace of $A$ and $B$. We know that the kernel $\langle v\rangle$ of
$B$ has trivial intersection with~$X$. Denote the span of $X$ and $v$
by $X'$, which is a $(g+2)$-dimensional subspace containing $X$ that is 
isotropic with respect to $B$. Let $\Lambda = X\cap\Z^{n+1}$ and
$\Lambda' = X'\cap\Z^{n+1}$ be the primitive lattices in $X$ and~$X'$, respectively.  There exists an element $\gamma$ in
$\SL_{n+1}(\Z)$, unique up to left multiplication by an element in
$G_1(\Z)$, such that $\Lambda =
\gamma^t.\,\text{Span}_\Z\{e_1,\ldots,e_{g+1}\}$ and $\Lambda' =
\gamma^t.\,\text{Span}_\Z\{e_1,\ldots,e_{g+2}\}$. Then 
$\gamma(A,B)\in W_{1}(\Z)$, and we can thus define the
$|q|$-invariant for the quadruple $(A,B,\Lambda,\Lambda')$~by $$|q|(A,B,\Lambda,\Lambda'):= |q(\gamma(A,B))|.$$ In other words,
we complete an integral basis $\{v_1,\ldots,v_{g+1}\}$ of $\Lambda$
to an integral basis $\{v_1,\ldots,v_{n+1}\}$ of $\Z^{n+1}$ such
that $\{v_1,\ldots,v_{g+2}\}$ forms an integral basis of
$\Lambda'$. When expressed in this basis, the pair $(A',B')$ of Gram
matrices for the quadratic forms defined by $A$ and $B$ lies in
$W_{1}(\Z)$ and we define $|q|(A,B,\Lambda,\Lambda') :=
|q|(A',B').$

Finally, we compute the $|Q|$- and $|q|$-invariants of
$\sigma_{m}(f(x,y))$, where $f(x,y)\in\W_{m}^{(2),\,\gen}$ is~irreducible.

\begin{proposition}\label{propuniqueQeven}
Let $n=2g+2$ with $g\geq 1$. Let $m$ be an odd positive squarefree
integer. Let $f(x,y)\in\U_{m}^{(2),\,\gen}$ be irreducible.  Let
$(A,B)$ be any element in $\SL_{n+1}(\Z)\cdot
\sigma_{m}(f(x,y))$. Let $\Lambda$ be a $(g+1)$-dimensional
primitive lattice contained in a $(g+2)$-dimensional primitive lattice
$\Lambda'$ such that $\Lambda$ is isotropic with respect to $A$ and
$\Lambda'$ is isotropic with respect to $B$. Then $|Q|(A,B,\Lambda)$
is either $m$ or $|f(0,1)|m$, and $|q|(A,B) :=
|q|(A,B,\Lambda,\Lambda') = m$, independent of
$(\Lambda,\Lambda')$. In particular, if $f(x,y)\in
\U_{m}^{(2),\,\gen}\cap \U_{m'}^{(2),\,\gen}$ is irreducible where
$m$ and $m'$ are distinct odd positive squarefree integers, then
$\sigma_{m}(f(x,y))$ and $\sigma_{m'}(f(x,y))$ are not
$\SL_{n+1}(\Z)$-equivalent.
\end{proposition}

\begin{proof}
The size of $J_f[2](\Q)$ is $2$ since $xf(x,y)$ has a unique even
degree factor (namely, $f(x,y)$) over $\Q$. Therefore, the pair 
$(A,B)$ has two $(g+1)$-dimensional common isotropic subspaces $X_1$
and $X_2$ over $\Q$. Let $\Lambda_1$ and $\Lambda_2$ denote the
corresponding primitive lattices contained in $X_1$ and $X_2$. The
unique $(g+2)$-dimensional subspace $X_1'$ (resp., $X_2'$) isotropic
with respect to $B$ and containing $X_1$ (resp., $X_2$) is the span of
$X_1$ (resp., $X_2$) with the kernel of $B$. Let $\Lambda'_1$ and
$\Lambda_2'$ denote the primitive lattices contained in $X_1'$ and
$X_2'$. We compute the $|Q|$- and $|q|$-invariants associated to these
lattices.

We may assume that  $(A,B)=\sigma_{m}(f)$ since the action
of $\SL_{n+1}(\Z)$ does not change the $|Q|$- or $|q|$
invariants. Since $|Q|$ is $\SL_2$-invariant, and $|q|$
remains unchanged when we add a multiple of $B$ to $A$, we may also assume
that
$$  xf(x,y) = m^2b_0x^{n+1} + mb_1x^{n}y + \cdots + b_nxy^n$$
and so
\begin{equation*}
 A =
 \left(\begin{array}{ccccccc}&&&&&&1\\&&&&&\iddots&\\ &&&&m&&\\ &&&c_0&&&
   \\ &&m&&c_{2}&&\\ &\iddots&&&&\ddots&\\ 1&&&&&&c_n \end{array}\right),\;\;
 B =
 \left(\begin{array}{ccccccc}&&&&&1&0\\&&&&\iddots&\iddots&\\ &&&1&r&&\\ &&1&c_1&&&
   \\ &\iddots&r&&c_3&&\\ 1&\iddots&&&&\ddots&\\ 0&&&&&&c_{n+1} \end{array}\right).
\end{equation*}
Comparing the $y^{n+1}$- and the $xy^{n}$-coefficients, we have $c_{n+1} = 0$ and $c_{n}=b_{n}$.

Let $\langle\,,\rangle_A$ and $\langle\,,\rangle_B$ denote the quadratic
forms corresponding to $A$ and $B$. Let $e_1,\ldots,e_{n+1}$ be the
standard basis on $\Z^{n+1}$. Since $c_{n+1}=0$, the vector $e_{n+1}$
spans the kernel of $B$. We may take the subspace spanned by
$e_1,\ldots,e_{g+1}$ as $X_1$. Then by construction 
$|Q|(A,B,\Lambda_1) = m$. When expressed in terms of the ordered integral basis
$\{e_1,\ldots,e_{g+1},e_n,e_{g+2},\ldots,e_{n-1}\}$, the top right
$(g+1)\times (g+1)$ block of $B$ has $1$'s on the antidiagonal and $0$'s
above the antidiagonal, and so has determinant $\pm 1$. Hence
$|q|(A,B,\Lambda_1,\Lambda_1') = m$.

The second common isotropic $(g+1)$-dimensional subspace $X_2$ is the
reflection of $X_1$ 
in the hyperplane perpendicular to $e_{n+1}$ with respect to $\langle\,,\rangle_A$. That is,
\begin{eqnarray*}
X_2 &=& \displaystyle {\rm Span}_\Q\Bigl\{e_1 - \frac{2\langle e_1,e_{n+1}\rangle_A}{\langle
    e_{n+1},e_{n+1}\rangle_A}e_{n+1},\ldots,e_{g+1} - \frac{2\langle
  e_{g+1},e_{n+1}\rangle_A}{\langle e_{n+1},e_{n+1}\rangle_A}e_{n+1}\Bigr\}\\
  &=&\displaystyle {\rm Span}_\Q\Bigl\{e_1 - \frac{2}{b_n}e_{n+1},e_2,e_3,\ldots,e_{g+1}\Bigr\}.
\end{eqnarray*}
Suppose first that $b_n$ is odd. Then we have the following integral
basis for $\Z^{n+1}$:
\begin{equation*}
  \Big\{b_{n}e_1-2e_{n+1},e_2,e_3,\ldots,e_{g+1},\frac{b_{n}+1}{2}e_1-e_{n+1},e_{g+2},\ldots,e_{n}\Big\}.
\end{equation*}
In terms of this basis, the top right $(g+1)\times(g+2)$ blocks
of $A$ and $B$ have the following form:
\begin{equation}\label{eq:generaloddlift2}
  A^\top = \left(\begin{array}{cccccc}
    -1&0&0&\cdots&\cdots&0\\0&0&0&&\iddots&1\\
    \vdots&\vdots&\vdots&\iddots&\iddots&\vdots\\
    \vdots&\vdots&0&1&&\vdots
   \\ 0&0&m&\cdots&\cdots&0\end{array}\right),\;\;
 B^\top =
 \left(\begin{array}{cccccc}0&&&&&b_{n}\\0&&&&1&\\
   \vdots&&&\iddots&&\\ \vdots&&\iddots&&&
   \\ 0&1&&&&\end{array}\right). 
\end{equation}
It is then easy to check that $|Q|(A,B,\Lambda_2) = |b_{n}|m$ and $|q|(A,B,\Lambda_2,\Lambda_2') = m$. 

When $b_{n}$ is even, we have the following integral basis
\begin{equation*}
  \Big\{\frac{b_{n}}{2}e_1-e_n,e_2,e_3,\ldots,e_{g+1},(b_{n}+1)e_1-2e_{n+1},e_{g+2},\ldots,e_{n}\Big\}.
\end{equation*}
In terms of this basis, the top right $(g+1)\times(g+2)$ blocks
of $A$ and $B$ have the same form as in~\eqref{eq:generaloddlift2}. Hence the $|Q|$- and $|q|$-invariants are
as stated in the proposition.
\end{proof}

\vspace{-.1in}

\section{A uniformity estimate for odd degree polynomials}\label{sec:monicodd}

Throughout this section, we fix an odd integer $n=2g+1$ with $g\geq
1$.
Our goal is to prove Theorem~\ref{thm:mainestimate}(b) by obtaining a
bound on the number of integral binary $n$-ic forms having bounded
height and discriminant  weakly divisible by the square of a
large squarefree integer.  

Let $m>0$ be an odd squarefree
integer. Recall that we defined a map \smash{$\sigma_m:\U^{(2)}_m\to W_0(\Z)$}
in Theorem~\ref{keymap} with the following two properties:
\smash{$f_{\sigma_m(f)}=f$} for every \smash{$f\in \U^{(2)}_m$},  and
\smash{$|Q|(\sigma_m(f))=m$}. Moreover, in Proposition
\ref{propuniqueQodd}, we proved that when \smash{$f\in \W_m^{(2)}$} is
irreducible, it is possible to naturally extend the definition of the
$|Q|$-invariant to the set \smash{$\SL_n(\Z)\cdot \sigma_m(f)$}.

Let
$W(\Z)^\dist$ denote the set of distinguished elements in $W(\Z)$, and
for any set $L\subset W(\Z)$, let $L^\irr$ denote the set of elements
$w\in L$ such that $f_w$ is irreducible. There is a natural
extension of the $|Q|$-invariant to the set $W(\Z)^{\dist,\irr}$. 
For a positive real number $M$ and any set $S\subset W(\Z)^{\dist,\irr}$, let $S_{|Q|>M}$ denote the set of elements $w\in S$ with $|Q(w)|>M$.
By~\cite{vdw}, the number of reducible
elements $f\in V(\Z)$ with~$H(f)<X$ is~$O(X^n)$. Hence we have the bound \medskip

\begin{equation}\label{eq:odddegreeprelim}
\#\bigcup_{\substack{m>M\\{\rm squarefree}}}\bigl\{f\in\W_m^{(2)}:H(f)<X\bigr\}
\ll \#\bigl(\SL_n(\Z)\backslash \bigl\{w\in W(\Z)^{\dist,\irr}_{|Q|>M}:H(w)<X\}\bigr)+O(X^{n}).
\end{equation}

In this section, we obtain an upper bound on the number of
$\SL_n(\Z)$-orbits on $W(\Z)^{\dist,\irr}_{|Q|>M}$ with height bounded by $X$. First, in
\S\ref{sec:redtheory}, we lay out the reduction theory necessary to
express the number of such orbits in terms of the
counts of lattice points in certain bounded regions. Then in \S4.2, we
partition these regions into three parts, the {\it main body}, the
{\it shallow cusp}, and the {\it deep cusp}. We prove the desired 
estimate for each of these parts, thereby obtaining Theorem
\ref{thm:mainestimate}(b).

\subsection{Reduction theory and averaging over fundamental domains}\label{sec:redtheory}

Recall that the Iwasawa decomposition of $\SL_n(\R)$ is given by 
$ \SL_n(\R)=NTK,
$ where $N$ is the group of unipotent lower triangular matrices in
$\SL_n(\R)$, $K=\SO(n)$ is a maximal compact subgroup of $\SL_n(\R)$,
and $T$ is the split torus of $\SL_n(\R)$ consisting of $n\times n$
diagonal matrices with positive diagonal entries and determinant $1$. We
denote elements in $T$ by
$s=\diag(t_1^{-1},t_2^{-1},\ldots,t_n^{-1})$, where $t_i>0$ for $1\leq
1\leq n$ and $t_1t_2\cdots t_n=1$.  It will be convenient to make the
following change of variables. For $1\leq i\leq n-1$, set $s_i$ to be
\begin{equation*}
s_i=(t_i/t_{i+1})^{1/n},\mbox{  which implies  }\;
t_i=\prod_{k=1}^{i-1}s_k^{-k}\prod_{k=i}^{n-1}s_k^{n-k}
\end{equation*}
for $1\leq i<n$. The Haar measure of $G(\R)$ in these coordinates is then 
given by
$$ dg=dn\delta(s)d^\times s dk, \;\;\;\;\;\mbox{where}\;\;\;\;\;
\delta(s)=
\prod_{1\leq i < j\leq n}\frac{t_j}{t_i} =
\prod_{i=1}^{n-1}s_k^{-nk(n-k)},$$ $dn$ and $dk$ are Haar measures on
$N$ and $K$, respectively, and $d^\times
s=\prod_{i=1}^{n-1}s_i^{-1}ds_i$.  

We denote the coordinates on $W$ by $a_{ij},b_{ij}$  for $1\leq i\leq
j\leq n$. These coordinates are eigenvectors for the action of $T$ on
the dual $W^*$ of $W$. Denote the $T$-weight of a coordinate $\alpha$
on $W$, or more generally a product $\alpha$ of powers of such
coordinates, by $w(\alpha)$. Then $w(a_{ij}) = w(b_{ij}) =
t_i^{-1}t_j^{-1}.$ It will be useful in what follows to compute the
weight of the $Q$-invariant, which is a homogeneous polynomial of
degree $g(g+1)$ in the coordinates of $W_0$. We view the torus~$T$~as~sitting inside $G_0$. Then by \eqref{eq:weightG_0}, we have
\begin{equation}\label{eqQweight}
w(Q)=\prod_{k=1}^gt_k^{-1}.
\end{equation}

Let $\FF$ be a fundamental set for the action of $\SL_n(\Z)$ on
$\SL_n(\R)$ that is contained in a {\it Siegel set}, i.e., contained
in $N'T'K$, where $N'$ is a set consisting of elements in $N$ whose
coefficients are absolutely bounded and $T'\subset T$ consists of
elements in $s\in T$ with $s_i\geq c$ for some positive constant
$c$. Let $\W(1)$ denote the subset of real binary $n$-ic forms of
height bounded by $1$ and let $R' = \sigma_1(\W(1))$, where 
$\sigma_1$ is as in \S3.4. Set $R:=\R_{>0}\cdot R'$,
and note that every distinguished element of $W(\R)$ is
$\SL_n(\R)$-equivalent to some element in $R$. 

Let $H_0$ be a nonempty open bounded left $K$-invariant set in
$\SL_n(\R)$. Denote the set $H_0\cdot R'$ by $\B_1$. Then $\B_1$ is
an absolutely bounded set in $W(\R)$. 
Let $\L$ be
any $\SL_n(\Z)$-invariant subset of $W(\Z)$ consisting of elements
that are distinguished over $\R$, and denote the set of elements in
$\L$ with height less than $X$ by $\L_X$. Throughout this section, let $Y=X^{1/n}$. Then the averaging method as described in
\cite[\S2.3]{BS2} yields the bound
\begin{equation}\label{eq:avg}
\displaystyle\#\bigl(\SL_n(\Z)\backslash\L_X\bigr)
\;\ll\; \displaystyle
\int_{\gamma\in\FF} \#\bigl(\gamma (Y\B_1)\cap
\L\bigr)d\gamma
\;\ll\;\displaystyle
\int_{\substack{s=(s_i)_i\\ s_i\geq c}}
\#\bigl(s (Y\B)\cap\L\bigr)\delta(s)d^\times s 
\end{equation}
for some absolutely bounded open set $\B$ containing $\B_1$.

We denote the second integral on the right hand side of
\eqref{eq:avg} by $\cI_X(\L)$, and break it up into an integral over the main body, the shallow cusp, and the deep cusp. 
We define the {\it main body} to be the range 
of the integral where
$|a_{11}|\geq 1$ for some element in $s(Y\B)$, and denote the main-body portion of $\cI_X(\L)$ by $\cI_X^\main(\L)$.
We define the {\it shallow cusp} to be the range 
of the integral where $|a_{11}|< 1$ for all elements in $s(Y\B)$ but $|a_{ij}|\geq 1$ for some $i,j\leq g$, and denote the shallow-cusp portion of $\cI_X(\L)$ by $\cI_X^\scusp(\L)$. We define the {\it deep cusp} to be the range 
of the integral where $|a_{ij}|<1$ for all $i, j\leq g$ and all elements in $s(Y\B)$, and denote the deep-cusp portion of $\cI_X(\L)$ by $\cI_X^\dcusp(\L)$.
Then 
\begin{equation}\label{eq:intpartition}
\cI_X(\L)=\cI_X^\main(\L)+\cI_X^\scusp(\L)+\cI_X^\dcusp(\L). 
\end{equation}

In the next subsection, we prove bounds for
the main body, the shallow cusp, and the deep cusp when
$\L=W(\Z)_{|Q|>M}^{\dist,\irr}$.

\subsection{The number of orbits of distinguished elements with large $Q$-invariant}\label{sgomodd}

In this subsection, we obtain the following upper bound on
{$\cI_X(W(\Z)^{\dist,\irr}_{|Q|>M})$}, thus yielding the same bound on
the quantity {$\#\bigl(\SL_n(\Z)\backslash
\{w\in W(\Z)^{\dist,\irr}_{|Q|>M}:H(w)<X\}\bigr)$} by \eqref{eq:avg}.

\begin{theorem}\label{thm:distirrodd}
We have
{$\cI_X(W(\Z)^{\dist,\irr}_{|Q|>M}) \ll_\epsilon X^{n+1-1/{5n}+\epsilon}
+ {X^{n+1+\epsilon}}/{M}.$}
\end{theorem}
Note that \eqref{eq:odddegreeprelim}, \eqref{eq:avg}, and Theorem
\ref{thm:distirrodd} immediately imply Part (b) of Theorem
\ref{thm:mainestimate}.

We bound $\cI_X(W(\Z)^{\dist,\irr}_{|Q|>M})$ by obtaining bounds
for the main body, the shallow cusp, and the deep cusp.  We consider first the main body.

\begin{proposition}\label{prop:oddMB}
We have
\smash{$
\cI_X^\main(W(\Z)^\dist)\ll_\epsilon X^{n+1-1/{(5n)}+\epsilon}.
$}
\end{proposition}
\begin{proof}
In \cite[Proposition 4.6]{HSV}, an upper bound of  $o(X^{n+1})$ is
obtained on $\cI_X^\main(W(\Z)^\dist)$. This is proved using the
following two ingredients: estimates with a power saving error tern on
$\cI_X^\main(\L)$ for lattices $\L\subset W(\Z)$, and a proof that the
density of \pagebreak  elements in $W(\F_p)$ that are not $\F_p$-distinguished is
bounded away from $1$, independent of $p$. Proposition
\ref{prop:oddMB} follows from these two ingredients along with an
application of the Selberg sieve used exactly as in \cite{ST}.
\end{proof}

Next, a bound on the shallow cusp follows directly from the proof of \cite[Proposition 4.3]{HSV}:

\begin{proposition}\label{prop:oddC}
We have
\smash{$
\cI_X^\scusp(W(\Z))\ll X^{n+1-1/{n}}.
$}
\end{proposition}
In \cite[Proposition 4.3]{HSV}, the shallow and deep cusps were treated simultaneously, but the points in the deep cusp
were ruled out since only nondistinguished elements were counted there. Hence the proof of \cite[Proposition~4.3]{HSV}
yields the claimed bound in Proposition~\ref{prop:oddC}.

Finally, to treat the deep cusp, let $U=\{a_{ij},b_{ij}:1\leq i\leq j\leq n\}$
denote the set of coordinates on $W$, and let $U_0=\{a_{ij},b_{ij}\mid
i\leq j, \:j\geq g+1\}$ denote the set of coordinates on~$W_0$. We define
a partial order $\lesssim$ on $U$ by setting $\alpha\lesssim\beta$ if
all the powers of $s_i$ in $w(\alpha)^{-1}w(\beta)$ are
nonnegative. Explicitly, $a_{ij}\leq a_{i'j'}$ if and only if
$i\leq i'$ and $j\leq j'$ (and similarly for $b_{ij}$,  as $a_{ij}$ and
$b_{ij}$ have the same weight). A subset $\cZ$ of 
$U_0$
is 
{\it 
saturated} 
if for any
$\beta\in \cZ$ and any $\alpha\in U_0$ with $\alpha\lesssim\beta$, 
the coordinate $\alpha$ also lies in
$\cZ$. We pick positive constants $c_{ij}$ for $1\leq i\leq j\leq n$ such that:
\begin{enumerate}[(a)]
\item If $|Yw(a_{ij})|<c_{ij}$, then $|a_{ij}|<1$ and $|b_{ij}|<1$ for every $(A,B)\in s(Y\B)$.
\item For all $s\in T'$ and $a_{ij}\lesssim a_{i'j'}$, we have
$w(a_{ij})/c_{ij}\leq w(a_{i'j'})/c_{i'j'}$.
\end{enumerate}
More explicitly, we may choose $c_{nn}$ to be sufficiently small and take
$$ c_{ij} = \Bigl(\sup_{s\in T'} \frac{w(a_{ij})}{w(a_{nn})}\Bigr) c_{nn},\quad\mbox{for }i\leq j\leq n.
$$
The significance of these constants $c_{ij}$ is the following: for every $Y>1$, first, if $Yw(a_{ij})<c_{ij}$, then every integral element in $s(Y\B)$ has $a_{ij}$- and $b_{ij}$-coordinates equal to $0$; and second, if $a_{ij}\lesssim a_{i'j'}$, then $Yw(a_{i'j'})<c_{i'j'}$ implies $Yw(a_{ij})<c_{ij}$.

We have the following result due to Davenport which we
use to estimate the number of lattice points in bounded regions.

\begin{proposition}[\cite{DavenportLemma}] \label{prop-davenport}
Let $\RR$ be a bounded, semi-algebraic multiset in $\R^n$ having
maximum multiplicity $m$ that is defined by at most $k$ polynomial
inequalities, each having degree at most $\ell$. Let $\RR'$ denote the
image of $\RR$ under any $($upper or lower$)$ triangular, unipotent
transformation of $\R^n$. Then the number of lattice points $($counted
with multiplicity$)$ contained in the region $\RR'$ is given by
\begin{equation*}
  \Vol(\RR) + O (\max\{\overline{\Vol}(\overline{\RR}),1\}),
\end{equation*}
where
$\overline\Vol(\overline{\RR})$ denotes the greatest $d$-dimensional
volume of any projection of $\RR$ onto a coordinate subspace obtained
by equating $n-d$ coordinates to zero, as $d$ ranges over all values
in $\{1, \dots, n-1\}$. The implied constant in the second summand
depends only on $n$, $m$, $k$, and $\ell$.
\end{proposition}

The following lemma gives conditions that ensure an element in $W(\R)$
has discriminant $0$. 
\begin{lemma}\label{lemD0}
Suppose that $(A,B)\in W(\R)$ satisfies $a_{ij}=b_{ij}=0$ for all
$i\leq k$ and $j\leq n-k$ for some  $k\in\{1,\ldots,g\}$. Then
the discriminant of $(A,B)$ is $0$.
\end{lemma}

\begin{proof}
One checks that $f_{A,B}$ has a
square factor of degree $k$ and so has discriminant $0$.
\end{proof}

The next lemma states that when $\L\subset W(\Z)$ consists of elements
with nonzero discriminant, the integral defining $\cI_X(\L)$ can be
cut off by  conditions of the form $s_i\ll X^{\Theta}$ for some absolute
constant $\Theta$ depending only on $n$.

\begin{lemma}\label{lemsbound}
There exists an absolute constant $\Theta$ depending only on $n$ such
that if $s\in T'$ with $s_i\gg X^\Theta$ for some $i$, then
$s(Y\B)\cap W(\Z)$ contains only points with discriminant $0$.
\end{lemma}

\begin{proof}
Let $s=\diag(t_1^{-1},\cdots,t_n^{-1})\in T'$; then 
$t_1\gg t_2\gg \cdots \gg t_n$ and $t_1t_2\cdots t_n=1$. Because of the
relation between the $t_j$'s and the $s_i$'s, it suffices to prove that if
$s(Y\B)$ contains an integral element with nonzero discriminant,
then $t_1$ is bounded from above by some power of $X$ or, equivalently, $t_n$ is
bounded from below by some power of $X$. By Lemma \ref{lemD0}, for $s(Y\B)$ to contain an integral element with
nonzero discriminant, we must have $Yw(a_{k,n-k})\gg 1$ for every
$k\in\{1,\ldots,g\}$. That is, $t_kt_{n-k}\ll Y$ for every
$k\in\{1,\ldots,g\}$. Multiplying these conditions together, we obtain
$t_n\gg Y^{-g}$. 
The lemma follows.
\end{proof}

We now estimate the contribution to
{$\cI_X(W(\Z)_{|Q|>M}^{\dist,\irr})$} coming from the deep cusp.

\begin{proposition}\label{proplargeQbound}
We have
\smash{$
\cI_X^\dcusp(W(\Z)_{|Q|>M}^\irr)\ll_\epsilon {X^{n+1+\epsilon}}/{M}.
$}
\end{proposition}

\begin{proof}
For a subset $\cZ$ of $U_0$, let $T'_\cZ$ denote the subset of $s\in
T'$ with $Y^{g(g+1)}w(Q)\gg M$, and $|Yw(a_{ij})|< c_{ij}$ precisely for those $(i,j)$ where $a_{ij}\in\cZ$ or $b_{ij}\in\cZ$. Note that $T'_\cZ$ is empty if $\cZ$ is not saturated. Define
\begin{equation*}
\begin{array}{rcl}
N(\cZ,X) &:=& 
\displaystyle\int_{s\in T'_\cZ}\#\bigl(s(Y\B)\cap W_0(\Z)\bigr)\,\delta(s)d^\times s
\\[.1in]&\ll &
\displaystyle\int_{s\in T'_\cZ}
\big(\prod_{\alpha\in U_0\backslash \cZ}Yw(\alpha)\big) \,\delta(s) d^\times s
\\[.1in]&= &
\displaystyle \int_{s\in T'_\cZ}
\big(\prod_{\alpha\in U_0}Yw(\alpha)\big)
\big(\prod_{\alpha\in \cZ}Y^{-1}w(\alpha)^{-1}\big) \,\delta(s) d^\times s,
\end{array}
\end{equation*}
where the bound on the second line follows from Proposition \ref{prop-davenport}.
Let $U' := \{a_{ij},b_{ij}\mid i+j<n\}$. If $\cZ$ is saturated and not
contained in $U'$, then $\cZ$ contains $a_{k,n-k}$ for some
$k=1,\ldots,g$. Hence, for any $s\in T'_\cZ$, every
integral element in $s(Y \B)\cap W_0(\Z)$ satisfies  $a_{ij}=b_{ij}=0$,
for $i\leq k$ and $j\leq n-k$, and so has zero discriminant by
Lemma \ref{lemD0}. Therefore, 
$$\cI_X^\dcusp(W(\Z)^\irr_{|Q|>M})\ll \sum_\cZ N(\cZ,X)$$ where the sum
is over saturated subsets $\cZ$ of $U_0$ contained in $U'$.

Now
\begin{eqnarray}
\nonumber\prod_{\alpha\in U_0} Yw(\alpha) &=& Y^{n(n+1)-g(g+1)} (t_1\cdots t_{g+1})^{2g+4}\\[-.1in]
\nonumber&=& \frac{Y^{n(n+1)}}{Y^{g(g+1)}w(Q)} (t_1\cdots t_{g+1})^{g+2}(t_{g+2}\cdots t_n)^{-(g+1)}\\
\label{eq:oddwow}&=& \frac{Y^{n(n+1)}}{Y^{g(g+1)}w(Q)} \prod_{i=1}^{g+1}\prod_{j=g+2}^n \frac{t_i}{t_j}.
\end{eqnarray}
Fix a saturated subset $\cZ$ of $U_0$ contained in $U'$. We define a
map $\pi:\cZ\rightarrow U_0\backslash U'$ by
\begin{equation*}
  \pi(a_{ij}) = a_{i,n-i},\qquad \pi(b_{ij}) = b_{n-j,j}.
\end{equation*}
Note that for any $\alpha\in \cZ$, we have $\pi(\alpha)\notin U'$ and
so $Yw(\pi(\alpha))\gg 1$. Furthermore, for every $\alpha\in U'$, we
have $\alpha\lesssim\pi(\alpha)$ and so $w(\pi(\alpha))/w(\alpha)\gg
1$.  Hence for any $s\in T'_\cZ$,
\begin{equation}
\prod_{\alpha\in \cZ} (Yw(\alpha))^{-1} \ll \prod_{\alpha\in \cZ}\frac{Yw(\pi(\alpha))}{Yw(\alpha)}
\ll \prod_{\alpha\in U'}\frac{Yw(\pi(\alpha))}{Yw(\alpha)}
\label{eq:oddwow2}= \Big(\prod_{g+2\leq i< j\leq n} \frac{t_i}{t_j}\Big) \Big(\prod_{1\leq i< j\leq g+1} \frac{t_i}{t_j}\Big).
\end{equation}
Here the first product on the right hand side is the
contribution from all $a_{ij}\in U_1$, and the second product is the contribution from all $b_{ij}\in
U_1$. Combining \eqref{eq:oddwow} and \eqref{eq:oddwow2} gives, for
any $s\in T'_\cZ$,
$$\prod_{\alpha\in U_0\backslash \cZ}Yw(\alpha) \ll \frac{Y^{n(n+1)}}{Y^{g(g+1)}w(Q)} \prod_{1\leq i < j\leq n} \frac{t_i}{t_j}
= \frac{Y^{n(n+1)}}{Y^{g(g+1)}w(Q)}\,\delta(s)^{-1}
\ll \frac{Y^{n(n+1)}}{M}\,\delta(s)^{-1}.
$$
Since each $s_i$ is bounded below by an absolute constant and bounded
above by a power of $X$ by Lemma~\ref{lemsbound}, we obtain 
$$N(\cZ,X) = O_\epsilon\Big(\frac{X^{n+1+\epsilon}}{M}\Big).$$ The proof is completed by summing over all saturated subsets $\cZ$
contained in $U_1$.
\end{proof}

Theorem \ref{thm:distirrodd} now follows from Propositions
\ref{prop:oddMB}, \ref{prop:oddC} and \ref{proplargeQbound}.

\section{A bound on the number of singular symmetric matrices in skewed boxes}\label{secSEK}

Let $n\geq 2$ be a positive integer and let $S=\Sym_2(n)$ denote the
space of symmetric $n\times n$ matrices.  Let $|\cdot|$ denote
Euclidean length on $S(\R)$ obtained by identifying $S(\R)$ with
$\R^{\dim S}=\R^{n(n+1)/2}$. Let $\D\subset S(\R)$ be a bounded open
set. For an integer $r$ with $1\leq r<n$, let $S(\Z)_{(r)}$ denote the
set of elements in $S(\Z)$ having rank $r$. In \cite{EK}, Eskin and
Katznelson obtained  asymptotics for the number of elements in $Y\D\cap
S(\Z)_{(r)}$ for $r\in\{1,\ldots,n-1\}$. 

In this paper, we will not
need exact asymptotics; upper bounds will suffice. 
In this section, our goal
is to obtain upper bounds on the number of elements of $S(\Z)_{(r)}$
in {\it skew} balls.

The group $\SL_n$ acts on $S$ via $\gamma(A)=\gamma A\gamma^t$ for
$\gamma\in\SL_n$ and $A\in S$. Let $T\subset \SL_n(\R)$ denote the
subgroup of diagonal matrices with positive coefficients. We denote
elements in $T$ by $s=\diag(t_1^{-1},\ldots,t_n^{-1})$. We are
interested in studying the skew ball $s(Y\D)$.  By symmetry, we may
assume that $s\in T'$, i.e., we have $t_1\gg t_2\gg \ldots \gg
t_n$. Moreover, in light of Lemma \ref{lemsbound}, we will assume that
$t_1\ll Y^\Theta$ and $t_n\gg Y^{-\Theta}$ for some absolute constant
$\Theta$ depending only on $n$.

For $s\in T$ and $r\in\{1,\ldots,n-1\}$, we define the constants 
$C(r,s)$ by
\begin{equation}\label{eq:Crs}
  C(r,s) = \prod_{i=1}^r\prod_{j=1}^{n-i}\frac{t_j}{t_{n-i+1}}= \prod_{\substack{1\leq i < j \leq n\\ j> n-r}}\frac{t_i}{t_j}.
  \end{equation}
When $s\in T'$, these constants satisfy 
$$C(r,s) \,\ll\,
C(n-1,s) = \prod_{1 \leq i < j \leq n} \frac{t_i}{t_j} =
\delta(s)^{-1},$$ where as before $\delta(s)$ is the character of the torus appearing in the Haar
measure of $\SL_n(\R)$.

Finally, for $1\leq r<n$, a positive real
number $Y$, and $s\in T$, let $N_r(Y,s)$ denote the number of elements
in $s(Y\D)\cap S(\Z)_{(r)}$.  We prove the following result.

\begin{theorem}\label{thm:mainB}
  Let $n\geq 2$ and $1\leq r<n$ be positive integers. Let $\Theta>0$
  be a real number. Let $Y>1$ be a real number, and let $s\in T'$ with
  $t_1\ll Y^\Theta$ and $t_n\gg Y^{-\Theta}$. Then
  $$N_r(Y,s) = O\bigl(C(r,s)Y^{nr/2}\log^r Y\bigr)$$
  where the implied constants depend only on $n$ and $\Theta$.
\end{theorem}
The case $s = 1$ of Theorem \ref{thm:mainB} follows from the work of 
Eskin-Katznelson \cite{EK}. Their strategy is to express the set of
singular symmetric matrices of rank $r$ as a union of lattices, each
of which consists of elements having a fixed row span. 
They count the number of elements in each
such lattice having bounded norm, and then sum over all possible
row spans.  We follow this strategy, explaining the
modifications necessary to bound integer points in skew balls.

Fix positive integers $k$ and $m$ with $k\leq m$, and a lattice
$\Lambda$ in $\R^n$ of rank $r$. A basis
$\{\ell_1,\ldots,\ell_k\}$ of $\Lambda$ is {\it reduced} if
the product $|\ell_1||\ell_2|\cdots |\ell_k|$ is minimal among all
integral bases of $\Lambda$. It is {\it almost reduced}
if $$|\ell_1||\ell_2|\cdots |\ell_k|\ll d(\Lambda),$$ where
$d(\Lambda)$ denotes the covolume of $\Lambda$ in $\Lambda\otimes\R$,
and the implied constant in the inequality depends only on $n$. If we
order an almost reduced basis $\{\ell_1,\ldots,\ell_k\}$ by length, then the $i$-th successive minimum of $\Lambda$
is within a constant multiple (depending only on $n$) of $|\ell_i|$
for every $i = 1,\ldots,k$.  To bound the number elements of $\Lambda$
in a ball, we use the following result of Schmidt \cite{Schmidt}.

\begin{proposition}\label{prop:sch}
Let $\Lambda$ be a rank $k$ lattice in $\R^m$ and $\D$ a
bounded open domain in $\R^m$. Let $\mu_1,\ldots,\mu_k$ be the
successive minima of $\Lambda$. Then for $Y>0$, we have
\begin{equation}\label{eq:sch}
\# (Y\D\cap \Lambda) = O\Big(\underset{1\leq j\leq
  k}{\max}\frac{Y^j}{\mu_1\cdots\mu_j}\Big).
\end{equation}
\end{proposition}

We use the notation of Theorem \ref{thm:mainB}.  Given a lattice
$\Lambda\subset\Z^n$ of rank $r$, let $S(\Lambda)$ denote the lattice
of symmetric matrices $B\in S(\Z)$ such that the row space
(equivalently, the column space) of $B$ is contained in
$\Lambda\otimes\R$. For two vectors $v_1$ and $v_2$ in $\R^n$, we
define
\begin{equation}\label{stardef}
  v_1\ast v_2:=\left\{
  \begin{array}{ll}
    v_1\cdot v_2^t+v_2\cdot v_1^t &\;{\rm if}\; v_1\;{\rm and}\;v_2
    \mbox{ are linearly independent};\\
    v_1\cdot v_2^t &\;{\rm otherwise}.
  \end{array}\right.
\end{equation}
Then $v_1\ast v_2=v_2\ast v_1\in S(\text{Span}\{v_1,v_2\})$, and 
\begin{equation}\label{eq:v*w}
|v||w|\leq |v\ast w|\leq 2|v||w|.
\end{equation}

Fix $\gamma\in\SL_n(\R)$. (For our applications, we
will take $\gamma\in T$.) Let $\Lambda\in\Z^n$ be a primitive lattice
of rank $r$. We bound the number of elements in $\gamma^{-1}(Y\D)\cap
S(\Lambda)$ using the bijection
\begin{equation*}
\begin{array}{rcl}
\gamma^{-1}(Y\D)\cap S(\Lambda)&\to&Y\D\cap\gamma(S(\Lambda))
\\[.05in]
A&\mapsto&\gamma A\gamma^{-1}
\end{array}
\end{equation*}
and instead bounding the number of elements in
$Y\D\cap\gamma(S(\Lambda))$.  We thus 
study the lattice $\gamma(S(\Lambda))\subset S(\R)$. The next result, which gives an almost reduced basis for
$\gamma(S(\Lambda))$ in terms of an almost reduced basis of
$\gamma\Lambda$, follows from the proofs of \cite[Proposition
  3.3]{EK} and \cite[Lemma 3.5]{EK}.

\begin{theorem}\label{thm:EK1}
Fix $\gamma\in\SL_n(\R)$. Let $\Lambda\subset\Z^n$ be a
primitive lattice of rank $r$, and let $\{\ell_1,\ldots,\ell_r\}$ be a
basis for $\gamma\Lambda$. Then $\{\ell_i\ast\ell_j\colon
1\leq i\leq j\leq r\}$ is a basis for
$\gamma(S(\Lambda))$. Furthermore,
$$d(\gamma(S(\Lambda)))=2^{r(r-1)/4}d(\gamma\Lambda)^{r+1}=2^{r(r-1)/4}d(\Lambda)^{r+1}.$$
In particular, if $\{\ell_1,\ldots,\ell_r\}$ is almost reduced, then
so is $\{\ell_i\ast\ell_j\colon 1\leq i\leq j\leq r\}$.
\end{theorem}

Next, by the proof of
\cite[Lemma 4.1]{EK}, we have the following result giving a necessary condition for the set
$Y\D\cap \gamma(S(\Lambda))$ to be nonempty.
\begin{proposition}\label{prop:lilj}
Let $\gamma\in\SL_n(\R)$ and let $\Lambda\subset\Z^n$ be a primitive
lattice of rank $r$ such that the successive minima of $\gamma\Lambda$
are $\mu_1\leq\ldots\leq\mu_r$. If $\#(Y\D\cap \gamma(S(\Lambda))>0$,
then $\mu_i\mu_j\leq c_1 Y$ for every pair~$(i,j)$ with $i+j\leq r+1$, for some constant $c_1$ depending only on $n$.
\end{proposition}

We now prove an upper bound on $\#(Y\D\cap
\gamma(S(\Lambda))$.
\begin{proposition}\label{prop:Brefined}
Let $\gamma\in\SL_n(\R)$ and let $\Lambda\subset\Z^n$ be a primitive
lattice of rank $r$ such that the successive minima of $\gamma\Lambda$
are $\mu_1\leq\ldots\leq\mu_r$. Then 
\begin{equation}\label{eq:Bcountrefine}
  \#(Y\D\cap \gamma(S(\Lambda))=
  O\Big(\frac{Y^{r(r+1)/2}}{d(\Lambda)^{r+1}}
  \prod_{\substack{1\leq i< j\leq r\\ i+j\leq r+1}}\frac{\mu_j}{\mu_i}\Big).
\end{equation}
\end{proposition}
\begin{proof}
Let $U(r)$ denote the set of pairs $(i,j)$ of positive integers such that $i\leq j\leq r$ and $i+j>r+1$. In other words,
elements in $U(r)$ correspond to the successive minima of the lattice
$\gamma(S(\Lambda))$ that are $\gg Y$. By Proposition \ref{prop:sch},
Theorem \ref{thm:EK1}, and \eqref{eq:v*w}, we have
\begin{equation*}
\#(Y\D\cap \gamma(S(\Lambda))\ll 
\frac{Y^{r(r+1)/2}}{d(\Lambda)^{r+1}}
\prod_{(i,j)\in U(r)}\Big(\frac{Y}{\mu_i\mu_j}\Big)^{-1}.
\end{equation*}
Assume that $\#(Y\D\cap \gamma(S(\Lambda))>0$. Then  $\mu_{r+1-j}\mu_j\ll Y$ for
all~$1\leq j\leq r$
by
Proposition \ref{prop:lilj}.   Thus
\begin{equation}\label{eq:Bcount}
\begin{array}{rcl}
\#(Y\D\cap \gamma(S(\Lambda))
&\ll&\displaystyle 
\frac{Y^{r(r+1)/2}}{d(\Lambda)^{r+1}}\prod_{(i,j)\in U(r)}
\Big(\frac{Y}{\mu_i\mu_j}\Big)^{-1}\frac{Y}{\mu_{r+1-j}\mu_j}
\\[.2in]&\ll&\displaystyle
\frac{Y^{r(r+1)/2}}{d(\Lambda)^{r+1}}\prod_{(i,j)\in U(r)}
\frac{\mu_i}{\mu_{r+1-j}}.
\end{array}
\end{equation}
Since $i\leq j$ and
$i+j>r+1$ for $(i,j)\in U(r)$, we have the following injection:
\begin{equation}\label{eq:inj}
\begin{array}{rcl}
U(r)&\to& \{(k,\ell):1\leq k< \ell\leq r:k+\ell\leq r+1\}
\\[.05in]
(i,j)&\mapsto& (r+1-j,i).
\end{array}
\end{equation}
Since $\frac{\mu_j}{\mu_i}\geq 1$ for $j>i$, the injection (\ref{eq:inj}) implies that the product of the ratios $\mu_j/\mu_i$ in \eqref{eq:Bcountrefine} is at least as large as the product of the ratios $\mu_i/\mu_{r+1-j}$ in \eqref{eq:Bcount}. The result follows.
\end{proof}

We now sum over the
appropriate lattices $\Lambda\subset\Z^n$ having rank $r$. To this
end, we fix an element  $s=\diag(t_1^{-1},t_2^{-1},\dots,t_n^{-1})\in
T'$. We will apply the previous results with $\gamma=s^{-1}$.
Set $L=(L_1,\ldots,L_r)$ with
$0<L_1\leq L_2\leq\cdots\leq L_r$. Let $\Sigma(L,s)$ denote the set of
primitive lattices $\Lambda\subset \Z^n$ of rank $r$ whose successive minima $\mu_1,\ldots,\mu_r$ of $s^{-1}\Lambda$ satisfy
$L_i\leq \mu_i<2L_i$ for each~$i$. 

\begin{lemma}\label{lem:litj}
Let $L=(L_1,\ldots,L_r)$ and $s=\diag(t_1^{-1},\ldots,t_n^{-1})\in
T'$. Then there is a constant $c'>0$ depending only on $n$ such that if $\#\Sigma(L,s) > 0$, then  $L_it_j^{-1} > c'$ for all
$(i,j)$ with $i+j\geq n+1$.
\end{lemma}
\begin{proof} 
Since $\#\Sigma(L,s) > 0$, there exists an integral lattice
$\Lambda\subset\Z^n$ of rank $r$ with basis
$\{\ell_1,\ldots,\ell_r\}$ such that $|s^{-1}\ell_i|<2L_i$
for $i\in\{1,\ldots,r\}$. For $1\leq j\leq n$, let $u_{ij}$ denote the
(integral) $j$-th entry of $\ell_i$. Then $|u_{ij}|\leq
2L_it_j^{-1}$ for every $1\leq i\leq r$ and $1\leq j\leq n$. The assumption that $s\in T'$ implies that $L_it_j^{-1}\ll L_{i'}t_{j'}^{-1}$ whenever $i\leq i'$ and $j\leq j'$. 

Suppose
that there is an  integer $k$ with $1\leq k\leq r$ such that
$L_kt_{n+1-k}^{-1}<c''$ for some sufficiently small constant
$c''>0$. Then $|u_{ij}|<1$, and thus $u_{ij}=0$ for all 
$(i,j)$ with $1\leq i\leq k$ and $1\leq j\leq n+1-k$. However, this
implies that the vectors $\ell_1,\ldots,\ell_k$ are not
linearly independent, a contradiction. Hence such a $k$ does not
exist and $L_kt_{n+1-k}^{-1}\gg 1$ for all~$k$, implying the
result.
\end{proof}

We now determine an upper bound for $\#\Sigma(L,s)$.
\begin{proposition}\label{prop:numL}
Let $L=(L_1,\ldots,L_r)$ and $s=\diag(t_1^{-1},\ldots,t_n^{-1})\in T'$.
Then 
\begin{equation}\label{eq:lsbound2}
  \#\Sigma(L,s) = O\Bigl((L_1\cdots L_r)^n
  \Bigl(\prod_{1\leq i < j\leq r}\frac{L_i}{L_j}\Bigr) C(r,s)\Bigr)
\end{equation}
where $C(r,s)$ is defined as in \eqref{eq:Crs}.
\end{proposition}

\begin{proof}
We count lattices $\Lambda$ by counting $r$-tuples of vectors
$(\ell_1,\ldots,\ell_r)$ such that each $\ell_i\in s^{-1}\Z^n$ satisfies
$L_i\leq |\ell_i| < 2L_i$ and such that 
$\{\ell_1,\ldots,\ell_r\}$ is a reduced basis of the lattice
it generates. For each $i = 1,\ldots,r$, let $\alpha(i)$ be the
largest integer such that $L_it_{\alpha(i)}^{-1}\leq c'$, where $c'$
is as in Lemma \ref{lem:litj}, or let $\alpha(i)=0$ if no such integer exists. By Proposition \ref{prop:sch}, the number of possibilities for $\ell_i$ is
\begin{equation*}
\ll\; \prod_{j=1}^n \max\bigl(L_it_j^{-1},1\bigr)\;\ll\; 
L_i^n\prod_{j=1}^{\alpha(i)}\bigl(L_i^{-1}t_j\bigr).
\end{equation*}
However, once $\ell_1$ is fixed, and given a vector $\ell_2$, at
most two of $\ell_2-k\ell_1$ can be part of a reduced basis for
$k\in\Z$. Since $\gg L_2/L_1$ vectors $\ell_2-k\ell_1$ satisfy the
same size bound as $\ell_2$ (namely, those with $k\ll L_2/L_1$), the number of choices
for the pair $(\ell_1,\ell_2)$ that are part of a reduced
basis is
\begin{equation*}
\frac{L_1}{L_2}L_1^nL_2^n
\prod_{j=1}^{\alpha(1)}\bigl(L_1^{-1}t_j\bigr)
\prod_{j=1}^{\alpha(2)}\bigl(L_2^{-1}t_j\bigr).
\end{equation*}
Continuing in this way, we obtain the bound 
\begin{equation}\label{eq:lbound}
\#\Sigma(L,s)\ll (L_1L_2\cdots L_r)^n
\Bigl(\prod_{1\leq i < j\leq r}\frac{L_i}{L_j}\Bigr)
\Bigl(\prod_{i=1}^r\prod_{j=1}^{\alpha(i)}L_i^{-1}t_j\Bigr).
\end{equation}
By Lemma \ref{lem:litj}, we have $\alpha(i)\leq n-i$ for
$i\in\{1,\ldots,r\}$. Therefore, 
\begin{equation}\label{eq:lsbound}
\prod_{i=1}^r\prod_{j=1}^{\alpha(i)}L_i^{-1}t_j
\ll\displaystyle 
\prod_{i=1}^r\prod_{j=1}^{\alpha(i)}L_i^{-1}t_j L_it_{\alpha(i)+1}^{-1}
\ll\displaystyle 
\prod_{i=1}^r\prod_{j=1}^{\alpha(i)}\frac{t_j}{t_{n-i+1}}
\ll\displaystyle 
\prod_{i=1}^r\prod_{j=1}^{n-i}\frac{t_j}{t_{n-i+1}}
=\displaystyle 
C(r,s).
\end{equation}
Equations \eqref{eq:lbound} and \eqref{eq:lsbound} yield the desired result.
\end{proof}

We are now ready to prove the main result of this section.

\bigskip

\noindent {\bf Proof of Theorem \ref{thm:mainB}:} Let $L=(L_1,\ldots,L_r)$ be a tuple such that $0<L_1\leq L_2\leq \cdots\leq L_r$. Then, by Lemma \ref{lem:litj}, Proposition \ref{prop:lilj} and the definition of $T'$, we see that for there to exist
a lattice $\Lambda\in\Sigma(L,s)$ such that
$\#(Y\D\cap s^{-1}(S(\Lambda))>0$, 
we must have
$$Y^{-\Theta_1} \ll L_1\leq\cdots\leq L_r\ll Y^{\Theta_2}\qquad\mbox{and}\qquad L_1\cdots L_r\ll Y^{r/2}$$
for some absolute constants $\Theta_1,\Theta_2>0$.
For any such $\Lambda$, Proposition \ref{prop:Brefined} states that
\begin{equation*}
\#(Y\D\cap s^{-1}S(\Lambda))\ll
\frac{Y^{r(r+1)/2}}{(L_1\ldots L_r)^{r+1}}
\prod_{\substack{1\leq i<j\leq r\\i+j\leq r+1}}\frac{L_j}{L_i}.
\end{equation*}
Thus
\begin{equation*}
N_r(Y,s)\ll \sum_L \#\Sigma(L,s)
\frac{Y^{r(r+1)/2}}{(L_1\ldots L_r)^{r+1}}
\prod_{\substack{1\leq i<j\leq r\\i+j\leq r+1}}\frac{L_j}{L_i},
\end{equation*}
where the sum is over $r$-tuples $L=(L_1,\ldots,L_r)$ with $L_1\leq
L_2\leq \cdots\leq L_n$ that partition the region
$\{(\mu_1,\ldots,\mu_r)\in
[Y^{-\Theta},Y^{\Theta'}]^r:\mu_1\leq\ldots\leq\mu_r\}$ into dyadic
ranges.  The sum over $L$ has length $O(\log^r
Y)$. Using the upper bound on $\#\Sigma(L,s)$ in Proposition
\ref{prop:numL}, we obtain 
\begin{equation*}
\begin{array}{rcl}
N_r(Y,s)
&\ll&\displaystyle
\sum_L C(r,s)(L_1\ldots L_r)^{n}\frac{Y^{r(r+1)/2}}{(L_1\ldots L_r)^{r+1}}
\bigg(\prod_{\substack{1\leq i<j\leq r\\i+j\leq r+1}}\frac{L_j}{L_i}\bigg)
\bigg(\prod_{1\leq i < j\leq r}\frac{L_i}{L_j}\bigg)
\\[.025in]&\ll&\displaystyle
\sum_LC(r,s)Y^{(n-r-1)r/2}Y^{r(r+1)/2}
\\[.2in]&\ll&\displaystyle
C(r,s)Y^{nr/2}\log^r Y.
\end{array}
\end{equation*}
This concludes the proof of Theorem \ref{thm:mainB}. $\Box$

\section{A uniformity estimate for even degree polynomials}\label{sec:moniceven}

We fix an even integer $n=2g+2$ with $g\geq
1$. Our goal is to prove Theorem~\ref{thm:mainestimate}(c) by
obtaining a bound on the number of integral binary $n$-ic forms having
bounded height having  discriminant weakly divisible by the square
of a large squarefree integer.

Throughout this section, we write $V:=V_n$ and $W:=W_{n+1}$. Let $m>0$ be an odd squarefree integer, and let $\W_m^\w:=\W_{m,n}^\w$. We also define the following
auxiliary sets:
\begin{eqnarray}
V(\Z)^\red&:=&
\{f\in V(\Z):\Gal(f(x,1))\neq S_n\}, \\
V(\Z)^{\Delta\text{\,small}}&:=&
\{f\in V(\Z):\Delta(f)\leq H(f)^{2n-2-\kappa}\},\\
\W_{m}^{(1\#)}&:=&
\{f\in \W_{m}^{(2)}:m\mid f(0,1)\},\\
\W_{m}^\gen&:=&
\{f\in \W_{m}^{(2)}:\gcd(m,f(0,1))=1\text{ and } f\notin V(\Z)^\red\cup  V(\Z)^{\Delta\text{\,small}}\},
\end{eqnarray}
where $\kappa>0$ is a small constant (whose exact value will be optimized later)
and $\Gal$ denotes the Galois group. Then, for any $M>0$, we have the following containment:
\begin{equation}\label{eq:evencontainment}
  \bigcup_{\substack{m>M\\{\rm squarefree}}}
  \W_{m}^{(2)}\;\subset\; V(\Z)^\red\cup V(\Z)^{\Delta\text{\,small}}
  \cup\bigcup_{\substack{m>\sqrt{M}\\{\rm squarefree}}}
  \W_{m}^{(1\#)}\cup\bigcup_{\substack{m>\sqrt{M}\\{\rm squarefree}}}
  \W_{m}^{\gen}.
\end{equation}
The number of
elements in $V(\Z)^\red$ having height less than $X$ was bounded by
$O(X^n$) 
in~\cite{vdw}.
We next prove a bound on
the number of elements in $V(\Z)^{\Delta\text{\,small}}$ of bounded height.

\begin{lemma}\label{lem:discsmall}
The number of integral binary
$n$-ic forms with height less than $X$ and absolute discriminant less than
$X^{2n-2-\kappa}$ is $O(X^{n+1 - \frac{\kappa}{2n-2}})$.
\end{lemma}
\begin{proof}
Set $\eta := \kappa/(2n-2)$. The number of integral binary $n$-ic
forms $a_0x^n + \cdots + a_ny^n$ with height less than $X$ such that
$|a_0|\leq X^{1-\eta}$ is $O(X^{n+1-\eta})$. Hence we assume  $|a_0|>X^{1-\eta}.$

Now fix integers $a_0,\ldots,a_{n-1}$ with $|a_i|\leq X$ and
$|a_0|>X^{1-\eta}$. The discriminant of $a_0x^n + \cdots + a_ny^n$ is
a polynomial $F(a_n)$ in $a_n$ of degree $n-1$ with leading
coefficient $C_na_0^{n-1}$ for some nonzero constant $C_n$. Let
$r_1,\ldots,r_{n-1}\in\C$ be the $n-1$ roots of $F(x)$. Then $$F(a_n) = C_na_0^{n-1}(a_n-r_1)\cdots(a_n-r_{n-1}).$$ Since
$|F(a_n)| < X^{2n-2-\kappa}$, we have $(a_n-r_1)\cdots(a_n-r_{n-1})\ll
X^{n-1-(n-1)\eta}$. Hence $|a_n - r_i| \ll X^{1-\eta}$ for some
$i=1,\ldots,n-1$. The number of such integers $a_n$ is
$O(X^{1-\eta})$. Since there are $O(X^n)$ choices for
$a_0,\ldots,a_{n-1}$, we obtain the desired bound.
\end{proof}

\noindent 
A direct application of a quantitative version of the Ekedahl sieve as in 
\cite[Theorem~3.3]{geosieve} implies the following bound on the number
of elements of bounded height belonging to \smash{$\W^{(1\#)}_{m}$} for
large $m$.
\begin{lemma}\label{lem:geoeven}
We have \, 
$\displaystyle
  \#\!\!\bigcup_{\substack{m>\sqrt{M}\\ m\;\rm{squarefree}}}
  \!\!\{f\in \W_{m}^{(1\#)}:H(f)<X\} = O\Big(\frac{X^{n+1}}{\sqrt{M}} + X^n\Big).
$
\end{lemma}

To prove Theorem \ref{thm:mainestimate}(c), it thus remains
to obtain an upper bound for
\begin{equation}\label{eq:WgenM}
\#\bigcup_{\substack{m>\sqrt{M}\\ m\;\rm{squarefree}}} \{f\in
\W_{m}^{\gen}:H(f)<X\}.
\end{equation}
In \S\ref{sembedeven},
we defined a map $\sigma_m$ from the set of elements \smash{$f\in\W_m^\w$} with $\gcd(m,f(0,1))=1$ to $W(\Z)$
such that $f_{\sigma_{m}(f)} = xf$ and $ |q|(\sigma_{m}(f)) =
m.$ 
For 
any $M>0$, define the
set \smash{$\LLM$} by 
\begin{equation*}
\LLM :=\bigcup_{\substack{m>M\\m\;{\rm squarefree}}}
\SL_{n+1}(\Z)\cdot \sigma_{m}(\W_{m}^{\gen}).
\end{equation*}
Then \eqref{eq:WgenM} is $\ll$
\begin{equation}\label{eq:evenbreakup}
\begin{array}{rcl}
\displaystyle
\#\Bigl(\SL_{n+1}(\Z)\backslash \{f\in \LLM \colon H(f) < X\}\Bigr)
&\ll& \displaystyle
\cI_X(\LLM )
\end{array}
\end{equation}
where
\begin{equation*}
  \cI_X(\LLM )=\int_{s\in T'} \#\bigl(s(Y\B)
  \cap\LLM \bigr)\,\delta(s)\,d^\times s
\end{equation*}
is as defined
immediately after \eqref{eq:avg}, where $Y$ is now taken to be $X^{1/(n+1)}$ throughout this section. Moreover, exactly as in the paragraph leading up to (\ref{eq:intpartition}), we break up \smash{$\cI_X(\LLM)$} into three parts---corresponding to the main body, the shallow cusp, and the deep cusp---and again write  $$\cI_X(\LLM)=\cI^\main_X(\LLM)+\cI^\scusp_X(\LLM)+\cI^\dcusp_X(\LLM).$$

The rest of this section is dedicated to obtaining an upper bound on
\smash{$\cI_X(\LLM )$}. Every element $(A,B)\in\LLM $ satisfies $\det(B)=0$ since
$f_{A,B}$ is divisible by $x$. In \S4, we used vanishing conditions
on the coefficients $\{a_{ij},b_{ij}\}$ of $W$ to estimate the
number of integral pairs $(A,B)$ in skewed domains of
$W(\R)$. Now, since we also need to impose the
condition that $B$ has determinant~$0$, we use the setup of \S\ref{secSEK} to
count the number of such $B$'s in skewed bounded domains by fibering
over the row space of $B$.

In \S\ref{sec:setupsplit}, we thus  further break up the three parts of $\cI_X(\LLM)$ into sums over row spaces of the singular matrix $B$. We also obtain some
preliminary bounds on $\cI_X(\LLM)$, and give some conditions that ensure that a pair $(A,B)$ has discriminant $0$.  In \S\ref{sec:evenmainbody},
\S\ref{sec:evencusp}, and \S\ref{sec:evendistcup}, we then prove the desired upper
bounds on 
\,$\cI^\main_X(\LLM)$, \,$\cI^\scusp_X(\LLM)$, \,and $\cI^\dcusp_X(\LLM)$, 
respectively. In conjunction with
\eqref{eq:evencontainment}, \eqref{eq:evenbreakup}, and Lemmas
\ref{lem:discsmall}--\ref{lem:geoeven} this will yield
Theorem \ref{thm:mainestimate}(c).

\subsection{Setup and preliminary bounds}\label{sec:setupsplit}

\subsubsection*{Coordinate systems, weight functions, and summing over row spaces}

Let $S(\Z)$ denote the set of $(n+1)\times(n+1)$ integral symmetric matrices. For
any primitive lattice $\Lambda$ of $\Z^{n+1}$, let $S(\Lambda)$ denote
the sublattice of $S(\Z)$ consisting of elements $B\in S(\Z)$ with row
space 
contained in $\Lambda$. For
$L=(L_1,\ldots,L_n)$ with $L_i\in\R$ and $L_1\leq L_2\leq \cdots\leq
L_n$ and $s\in T'$, let $\Sigma(L,s)$ denote the set of primitive
lattices $\Lambda\subset\Z^{n+1}$ of rank $n$ such that the successive
minima $\mu_1,\ldots,\mu_n$ of $s^{-1}\Lambda$ satisfy
$L_1\leq\mu_i\leq 2L_i$ for each $i$. We define $\cS(L,s)\subset
S(\Z)$ by
\begin{equation*}
\cS(L,s):=\bigcup_{\Lambda\in\Sigma(L,s)}S(\Lambda).
\end{equation*}

We next introduce coordinate systems and weight functions. Let
\begin{equation*}
\cM:=\{\ell_{ij}:1\leq i\leq n,\;1\leq j\leq n+1\}
\end{equation*}
denote the set of coordinates of $n$-tuples of vectors in $\R^{n+1}$.
We define
\begin{equation*}
w_L(\ell_{ij}):=L_it_j^{-1}.
\end{equation*}
The significance of $w_L$ is the following. Let
$\Lambda\in\Sigma(L,s)$ be a lattice with an integral basis
$\{\ell_1,\ldots,\ell_n\}$ such that
$\{s^{-1}\ell_1,\ldots,s^{-1}\ell_n\}$ is a Minkowski-reduced basis
for $s^{-1}\Lambda$. Then the $j$th coefficient of $\ell_i$ is $\ll
L_it_j^{-1}=w_L(\ell_{ij})$. In particular, for the absolute value of
the $j$th coefficient of $\ell_i$ to be nonzero, we must have
$w_L(\ell_{ij})\gg 1$. When $L$ is implicit, we will write $w$ in
place of $w_L$.  

Let $\cK$ denote the set of coefficients
$\{a_{ij}:1\leq i\leq j\leq n+1\}$, and recall the weight function
\begin{equation*}
w(a_{ij})=t_i^{-1}t_j^{-1}.
\end{equation*}
Define a partial order on $\cK$ by setting $a_{ij}\lesssim a_{i'j'}$
if $i\leq i'$ and $j\leq j'$, and on $\cM$ by setting
$\ell_{ij}\lesssim \ell_{i'j'}$ if $i\leq i'$ and $j\leq j'$. The
significance of this partial order is that if $\alpha,\beta\in\cK$
with $\alpha\lesssim\beta$ and $s\in T'$, then $w(\alpha)\ll w(\beta)$
and similarly $w_L(\alpha)\ll w_L(\beta)$ if $\alpha,\beta\in\cM.$ 

We say that a subset
$\cZ$ of $\cK\cup\cM$ is {\it saturated} if for any $\alpha\in \cZ$,
all the $\alpha'\in\cK\cup\cM$ with $\alpha'\lesssim\alpha$ are also
contained in $\cZ$. 

 Let $\D\subset
S(\R)$ be a bounded domain such that $\B_1\subset\D\times\D$. We pick positive constants $c_{ij}$ for $1\leq i\leq j\leq n+1$ and $c_i'$ for $1\leq i\leq n$ such that:
\begin{enumerate}[(a)]
\item if $|Yw(a_{ij})|<c_{ij}$, then the $a_{ij}$--coordinate of any integral element in $s(Y\D)$ is $0$;
\item if $|w_L(\ell_{ij})|<c_j'$, then the $j$th coefficient of $\ell_i$ for any lattice $\Lambda\in\Sigma(L,s)$ is $0$;
\item $c_i'<c'$ for all $i=1,\ldots,n$, where $c'$ is the constant in Lemma \ref{lem:litj};
\item $c_1c_{g+1,g+1}\leq c_{g+1}'^{\,\,2}$, where $c_1$ is the constant in Proposition \ref{prop:lilj};
\item for any $i\leq i'$ and $j\leq j'$, we have $w(a_{ij})/c_{ij} \leq w(a_{i'j'})/c_{i'j'}$ and $w(\ell_{ij})/c_j' \leq w(\ell_{i'j'})/c_{j'}'.$
\end{enumerate} 
More explicitly, we choose $c_{n+1,n+1}$ and $c_n'$ to be sufficiently small and take
\begin{eqnarray*}
\nonumber c_{ij} &=& \Bigl(\sup_{s\in T'} \frac{w(a_{ij})}{w(a_{n+1,n+1})}\Bigr) c_{n+1,n+1}\quad\mbox{for }i\leq j\leq n+1;\\
\nonumber c_i' &=&  \Bigl(\sup_{s\in T'} \frac{t_i^{-1}}{t_n^{-1}}\Bigr) c_n'\quad\mbox{for }i\leq n.
\end{eqnarray*}

For any nondecreasing $n$-tuple $L$ of positive real numbers, and a saturated
subset $\cZ$ of $\cK\cup\cM$, we define the following subset
$T_\cZ(L,Y)$ of $T'$:
\begin{equation}\label{eqTZ}
T_\cZ(L,Y) := \left\{s\in T'\left| 
\begin{array}{l} 
s_i\ll X^\Theta\;\; \forall i\in\{1,\ldots,n\}\\[.1in]
{\rm for}\;a_{ij}\in\cK,\;|Yw(a_{ij})|<c_{ij} \;{\rm{iff}}\;a_{ij}\in \cZ\cap\cK\\[.1in]
{\rm for}\;\ell_{ij}\in\cM,\;|w_L(\ell_{ij})|<c_j' \;{\rm{iff}}\;\ell_{ij}\in \cZ\cap\cM
\end{array}
\right.\right\},
\end{equation}
where $\Theta$ is the absolute constant from Lemma \ref{lemsbound}.

For $X$, $Y$, $L$, $\cZ$ as above and any subset $\L$ of
$W(\Z)$, we define the quantity
\begin{equation}\label{eq:defNLLZX}
N(\L,L,\cZ,X) := \int_{T_\cZ(L,Y)} \#\{(A,B)\in
(s(Y\D)\times s(Y\D))\cap \L\mid B\in \cS(L,s)\}
\,\delta(s)\,d^\times s.
\end{equation}
In the proof of Theorem \ref{thm:mainB}, we showed that unless
$Y^{-\Theta_1}<L_1$ and $Y^{\Theta_2}>L_n$ for some absolute positive constants
$\Theta_1$ and $\Theta_2$, we have $\cS(L,s)=\emptyset$, which implies
that $N(\LLM ,L,\cZ,X)=0$.  Therefore, 
\begin{equation*}
\cI_X(\LLM )\ll\displaystyle
\sum_L\sum_{\cZ}N(\LLM ,L,\cZ,X),
\end{equation*}
where the inner sum is over saturated subsets $\cZ$ of $\cK\cup\cM$,
and the outer sum is over $n$-tuples $L=(L_1,\ldots,L_n)$ with
$L_1\leq L_2\leq \cdots\leq L_n$ that partition the region
$\{(\mu_1,\ldots,\mu_n)\in
[Y^{-\Theta_1},Y^{\Theta_2}]^n:\mu_1\leq\ldots\leq\mu_n\}$ into dyadic
ranges.  

We may therefore bound the main-body, the shallow-cusp, and the
deep-cusp parts of $\cI_X(\LLM )$ in terms of
sums over $N(\LLM ,L,\cZ,X)$. We have 
\begin{equation}\label{eq:evenMCDC}
\begin{array}{rcl}
\cI_X^\main(\LLM )&\ll&\displaystyle
\sum_L\sum_{\cZ:a_{11}\not\in\cZ}N(\LLM ,L,\cZ,X),
\\[.225in]
\cI_X^\scusp(\LLM )&\ll&\displaystyle
\sum_L\sum_{\substack{\cZ:a_{11}\in\cZ\\a_{g+1,g+1}\not\in\cZ}}N(\LLM ,L,\cZ,X),
\\[.325in]
\cI_X^\dcusp(\LLM )&\ll&\displaystyle
\sum_L\sum_{\cZ:a_{g+1,g+1}\in\cZ}N(\LLM ,L,\cZ,X).
\end{array}
\end{equation}

\subsubsection*{A preliminary upper bound}

We now prove some preliminary results on $N(\LL1 ,L,\cZ,X))$.  We start with an upper bound on
$N(\LL1 ,L,\cZ,X)$, which also bounds $N(\LLM ,L,\cZ,X)$ by directly counting the number of possible $A$'s and then using the results of \S\ref{secSEK} to count
$B$'s. For a saturated subset $\cZ$ of $\cK\cup\cM$,  define
\begin{equation*}
w_L(\cZ):=\Bigl(\prod_{\alpha\in\cZ\cap\cK}w(\alpha)\Bigr)
\Bigl(\prod_{\alpha\in \cZ\cap\cM}w_L(\alpha)\Bigr).
\end{equation*}
In what follows, the $n$-tuple $L$ will be clear from the context, and
we simply write $w$ in place of~$w_L$.

\begin{proposition}\label{prop:NUX}
Suppose that $\cZ$ is a saturated subset of $\cK\cup\cM$.
Then  
\begin{equation}\label{eq:NUX}
N(\LL1 ,L,\cZ,X) \ll
X^{n+1}\,\int_{T_\cZ(L,Y)}
Y^{-\#(\cZ\cap\cK)}w(\cZ)^{-1}
\Big(\prod_{\substack{1\leq i< j\leq n\\ i+j> n+1}}\frac{L_i}{L_j}\Big)
\,\delta(s)d^\times s.
\end{equation}
\end{proposition}

\begin{proof} 
By Proposition \ref{prop-davenport}, the number of elements $A\in
s(Y\D)\cap S(\Z)$ is
\begin{equation}\label{eq:Abound}
\ll Y^{(n+1)(n+2)/2}\prod_{a_{ij}\in \cZ\cap\cK}(Yw(a_{ij}))^{-1}
\ll Y^{(n+1)(n+2)/2-\#(\cZ\cap\cK)} w(\cZ\cap\cK)^{-1}.
\end{equation}
By the definition of $T_\cZ(L,Y)$, it follows from \eqref{eq:lbound}
that for every $s\in T_\cZ(L,Y)$, we have 
\begin{equation}\label{eq:Lambda}
\#\Sigma(L,s)
\ll
(L_1\cdots L_n)^{n+1} w(\cZ\cap\cM)^{-1}
\prod_{1\leq i<j\leq n}\frac{L_i}{L_j}.
\end{equation}
For each $\Lambda\in\Sigma(L,s)$, Proposition \ref{prop:Brefined}
implies that the number of integral symmetric matrices $B\in s(Y\D)$ 
whose row space is contained in $\Lambda$ is
\begin{equation}\label{eq:refine}
   \ll \frac{Y^{n(n+1)/2}}{(L_1\cdots L_n)^{n+1}}
    \prod_{\substack{1\leq i< j\leq n\\i+j\leq n+1}}
    \frac{L_j}{L_i}.
\end{equation}
Combining \eqref{eq:Abound}, \eqref{eq:Lambda}, and \eqref{eq:refine},
and recalling that $X = Y^{n+1}$,  gives \eqref{eq:NUX}.
\end{proof}

\vspace{-.1in}

\subsubsection*{Conditions for vanishing discriminant}

Next, we give some conditions on $\cZ$ that ensure
$N(\LL1 ,L,\cZ,X) = 0$. We start with the following
algebraic result that gives sufficient conditions on a pair $(A,B)\in
W(\C)$ that ensure it has discriminant $0$.
\begin{lemma}\label{lem:disc} 
Suppose that $(A,B)$ is an element of $W(\C)$ such that one of
the following three conditions are satisfied:
\begin{itemize}
\item[{\rm (a)}] The kernel of $B$ has dimension at least $2$.
\item[{\rm (b)}] There is a nonzero vector $v\in \C^{n+1}$ that is in
  the kernel of $B$ and isotropic with respect to $A$.
\item[{\rm (c)}] There exists $k\in\{ 1,\ldots,g+1\}$ such that
  $a_{ij} = b_{ij} = 0$ for all $1\leq i\leq k$ and all $1\leq j\leq
  n+1-k$.
\end{itemize}
Then $\Delta(A,B) = 0$.
\end{lemma}
\begin{proof}
This is a standard result in the algebraic geometric theory of pencils
of quadrics. We give another proof using the explicit formula for
$f(x,y) = f_{A,B}(x,y).$ The claim regarding Condition (c) is
Lemma \ref{lemD0}. If the kernel of $B$ has dimension at least $2$,
then the quadratic form defined by~$A$ restricted to the kernel of $B$ admits a
nonzero isotropic vector in $\C^{n+1}$. Thus  Condition (a) implies
Condition~(b). Suppose now that Condition (b) is satisfied. Then the
$y^{n+1}$-coefficient of $f(x,y)$ is~$0$ since $B$ is singular. The
$xy^n$-coefficient of $f(x,y)$ equals, up to sign, the alternating sum
of the determinants of the matrices obtained by replacing the $i$-th
column of $B$ by the $i$-th column of $A$.  By translating the vector
$v$ to $(1,0,0,\ldots,0)$ using an element of $\SL_{n+1}(\C)$, we may assume that the first column (and row) of $B$ is~$0$
and the $(1,1)$-entry of $A$ is $0$. It is then easy to see that the
determinant of the matrix obtained by replacing the $i$-th column of
$B$ by the $i$-th column of $A$ is $0$ for any $i$. Hence $\Delta(A,B)
= \Delta(f) = 0$.
\end{proof}

We now translate these conditions into the vanishing of
$N(\LL1 ,L,\cZ,X)$ for certain sets~$\cZ$.  To this end,
define the set $\cZ_1\subset \cK\cup\cM$ by
\begin{equation*}
\cZ_1 := \{a_{ij}\mid i\leq j,\,i+j\leq n\} \cup 
\{\ell_{ij}\mid i+j\leq n+1\}. 
\end{equation*}
\vspace{-.25in}
\begin{lemma}\label{lem:NUX}
Let $\cZ$ be a saturated subset of $\cK\cup\cM$ satisfying one of 
the following two conditions:
\begin{itemize}
\item[{\rm (a)}] The set $\cZ$ is not contained in $\cZ_1$.
\item[{\rm (b)}] There exists $k\in\{1,\ldots,g+1\}$ such that
  $a_{kk}\in \cZ$ and $\ell_{n+1-k,k}\in \cZ$.
\end{itemize}
Then $N(\LL1 ,L,\cZ,X) = 0$.
\end{lemma}
\begin{proof}
If $\cZ$ contains some $\ell_{ij}\notin \cZ_1$, then for every $s\in
T_\cZ(L,Y)$, the set $\Sigma(L,s)$ (and hence $\cS(L,s)$) is empty by
Lemma \ref{lem:litj}. This implies that
$N(\LL1 ,L,\cZ,X)=0$. If $\cZ$ contains some $a_{ij}\notin
\cZ_1$, then every integral $(A,B)\in s(Y\D\times s(Y\D)$ has
discriminant $0$ by Condition (c) of Lemma~\ref{lem:disc}. Once again,
this implies that $N(\LL1 ,L,\cZ,X)=0$.

Let $k$ be an integer satisfying Condition (b) of the lemma, and let
$s\in T_\cZ(L,Y)$. Let $(A,B)$ be such that $A\in s(Y\D)$ and $B\in
\cS(L,s)$. Since $\ell_{n+1-k,k}\in \cZ$, it follows that there exists
a nonzero vector $v\in\C^{n+1}$ of the form
$(v_1,\ldots,v_k,0,\ldots,0)$ that is in the kernel of $B$. Since
$a_{kk}\in \cZ$, it follows that $v$ is isotropic with respect to
$A$. By Condition (b) of Lemma \ref{lem:disc}, it follows that
$\Delta(A,B) = 0$, implying that $N(\LL1 ,L,\cZ,X)=0$, as
desired.
\end{proof}

\vspace{-.1in}

\subsection{Bounding the number of distinguished elements
  in the main body}\label{sec:evenmainbody}

In this subsection, we  bound the number of
distinguished elements in the main body:
\begin{theorem}\label{th:evenmainbody}
We have $\cI_X^{\main}(\LL1 )=O\bigl(X^{n+1-1/(10n)}\bigr)$.
\end{theorem}
As $\LLM \subset \LL1 $ for $M\geq 1$, it follows that $\cI_X^\main(\LLM )$
satisfies the same bound.

We will use the Selberg sieve to show that distinsuished elements are
negligible in number in the main body. However, applying the Selberg
sieve requires asymptotics along with a power saving error term. Our
methods in \S\ref{secSEK} do not yield such results. 

Hence we will
instead fiber over $B\in s(Y\D)\cap S(\Z)$ having determinant $0$,
apply the Selberg sieve to prove that there are negligibly many $A\in
s(Y\D)\cap S(\Z)$ such that $(A,B)$ is distinguished, and then bound
the number of possible $B$'s using the results of Section 5. To carry
out the middle step, we require the following lower bound on the
number of nondistinguished elements modulo primes~$p$ that is  independent of $p$ and $B$.

\begin{lemma}\label{lemndist}
Let $B_0$ be an element in $S(\F_p)$ with $\F_p$-rank $n$. Let  $S_{B_0}^{\ndist}(\F_p)$ denote the
set of elements $A\in S(\F_p)$ such that $(A,B)$ has nonzero
discriminant and $A$ and $B$ do not have a common isotropic
$(g+1)$-dimensional subspace. Then 
\begin{equation*}
\frac{\#S_{B_0}^{\ndist}(\F_p)}{\#S(\F_p)}\gg_n 1.
\end{equation*}
\end{lemma}
\begin{proof}
For an element $B\in S(\F_p)$ with $\F_p$-rank $n$ and kernel spanned by $v$,
let $d(B)$ denote the discriminant of the corresponding quadratic form
on $\F_p^{n+1}/(\F_p v)$. If $B_1,B_2\in S(\F_p)$
have $\F_p$-rank $n$ and $d(B_1)/d(B_2)\in\F_p^{\times 2}$, then $B_1$
and $B_2$ are $\SL_{n+1}(\F_p)$-equivalent. Indeed, by using
$\SL_{n+1}(\F_p)$ transformations, we may assume the last row and
columns of $B_1$ and $B_2$ are all $0$. The nondegenerate forms
defined by the top left $n\times n$ blocks of $B_1$ and $B_2$ have
discriminants $d(B_1)$ and $d(B_2)$, which are in the same quadratic
residue class. Hence they are equivalent via an element
$\gamma\in\GL_n(\F_p)$. Expanding $\gamma$ to an element in
$\SL_{n+1}(\F_p)$ by appending an additional row and column whose
entries are all $0$, except for the $(n+1,n+1)$-entry which is
$\det\gamma^{-1}$,  gives an element in $\SL_{n+1}(\F_p)$ that
takes $B_1$ to $B_2$.

Let $B_0\in S(\F_p)$ have $\F_p$-rank $n$. For each binary
$n$-ic form $f(x,y) = a_0x^n + \cdots + a_ny^n$ over $\F_p$ that
splits completely over $\F_p$ such that $\Delta(xf(x,y))\neq 0$ and
$a_0\neq 0$, we construct a nondistinguished element $(A_0,B_0)$ with
$f_{A_0,B_0}=xf(x,y)$.  Let $f$ be such a form. Then $a_n\neq0$. Let  $\alpha =
d(B_0)/a_n$. As noted in  \S\ref{sec:2.1}, there exist at
least two (in fact $2^{n-1}$) $\SL_n(\F_p)$-orbits of $(A,B)\in
W_n(\F_p)$ such that $f_{A,B} = \alpha f(x,y)$. Pick two
inequivalent representatives $(A_1,B_1)$ and $(A_2,B_2)$. Let $A_1'$
and $A_2'$ be the $(n+1)$-ary quadratic forms obtained from $A_1$ and
$A_2$, respectively, by appending an additional row and column whose
entries are all $0$ except for the $(n+1,n+1)$-entry which is
$\alpha^{-1}$. Let $B_1'$ and $B_2'$ be the $(n+1)$-ary quadratic
forms obtained from $B_1$ and $B_2$, respectively, by appending an
additional row and column whose entries are all $0$. Then
$f_{A_1',B_1'} = f_{A_2',B_2'} = xf(x,y).$ Since
$(A_1,B_1)$ and $(A_2,B_2)$ are $\SL_n(\F_p)$-inequivalent, it follows
that $(A_1',B_1')$ and $(A_2',B_2')$ are
$\SL_{n+1}(\F_p)$-inequivalent. Hence, without loss of generality, we
may assume that $(A_1',B_1')$ is nondistinguished. Now
$d(B_1')=\alpha a_n = d(B_0)$, and so there exists
$\gamma\in\SL_{n+1}(\F_p)$ such that $\gamma B_1'\gamma^t = B_0$. Then
$A_0 = \gamma A_1'\gamma^t$ does the job.

We complete the proof of the lemma via the orbit-stabilizer
theorem. By the above construction, there are $\gg_n
p^{n+1}$ binary $(n+1)$-ic forms $xf(x,y)$, with $\Delta(xf(x,y))\neq
0$ and $a_0\neq 0$, such that there exists an element $A\in S(\F_p)$
with $f_{A,B_0}=xf(x,y)$ and $(A,B_0)$ nondistinguished. The
group $G_{B_0}(\F_p) = \{\gamma\in \SL_{n+1}(\F_p)\colon \gamma
B_0\gamma^t = B_0\}$ acts on the set of such $A$ with stabilizer of size
$\# J_{xf}[2](\F_p)$, where $J_{xf}$ is the Jacobian of the hyperelliptic
curve defined by $z^2 = xf(x,y)y$. Any element of $\gamma\in
G_{B_0}(\F_p)$ preserves the kernel $\F_p v$ of $B_0$ and
stabilizes the nondegenerate form $b_0$ on $\F_p^{n+1}/(\F_p v)$
induced by $B_0$. The determinant $1$ condition then gives
\begin{equation*}
  \# G_{B_0}(\F_p) = \# \text{O}(b_0)(\F_p) =
  2p^{\frac{n^2+n}{2}}\Bigl(1 - \frac{O(1)}{p^2}\Bigr).
\end{equation*}
Finally, since $\#J_{xf}[2](\F_p)\ll_n 1$, we have
\begin{equation*}
\#S_{B_0}^{\ndist}(\F_p)\gg_n p^{n+1}p^{n(n+1)/2}=p^{(n+1)(n+2)/2} = \#S(\F_p),
\end{equation*}
as desired.
\end{proof}

\vspace{-.1in}

\begin{corollary}\label{corndist}
  Fix $a\in\F_p^\times$ and $B_0\in S(\F_p)$ with rank $n$. Let
  $S_{B_0}^{\ndist}(\F_p)_{a_{11}=a}$ denote the set of all elements
  $A\in S_{B}^{\ndist}(\F_p)$ with $a_{11}=a$. Then
\begin{equation*}
\frac{\#S_{B}^{\ndist}(\F_p)_{a_{11}=a}}{\#S(\F_p)/p}\gg_n 1.
\end{equation*}
\end{corollary}
\begin{proof}
Since the property of $(A,B)$ being nondistinguished is preserved
when $A$ is multiplied by an element of $\F_p^\times$, the claim
follows immediately from Lemma \ref{lemndist}.
\end{proof}

We now bound the number of pairs $(A,B)$ in the main body where
the first row and column of $B$ are zero.
\begin{proposition}\label{propMBSel1}
We have
\begin{equation}\label{eqMBSel1}
\int_{\substack{s\in T'\\Yw(a_{11})\gg 1}}\#\bigl\{(A,B)\in
(s(Y\D)\times s(Y\D))\cap \LL1 : b_{1i}=0\,\forall i\}
\,\delta(s)d^\times s
\ll X^{n+1-1/(10n)}.
\end{equation}
\end{proposition}

\begin{proof}
Let $s\in T'$ be an element with $Yw(a_{11})\gg 1$. Then 
\begin{equation*}
\begin{array}{rcl}
  \displaystyle\#\bigl\{A\in s(Y\D)\cap S(\Z)\bigr\}&\ll&
  \displaystyle Y^{(n+1)(n+2)/2};
\\[.025in]\displaystyle
\#\bigl\{B\in s(Y\D)\cap S(\Z):b_{1i}=0\,\forall i\bigr\} &\ll&
\displaystyle Y^{n(n+1)/2}\prod_{i=1}^{n+1}w(b_{1i})^{-1}.
\end{array}
\end{equation*}
For each $B\in s(Y\D)\cap S(\Z)$ having rank $n$, we bound the number
of $A\in s(Y\D)\cap S(\Z)$ such that $(A,B)$ is
distinguished. Indeed, after additionally fibering over the
coefficient $a_{11}$, Corollary \ref{corndist} in conjunction with an
application of the Selberg sieve, used as in \cite{ST}, saves a
$1/5$th power of the smallest coefficient range among the remaining
$a_{ij}$. That is, we obtain a saving of $(Yw(a_{12}))^{-1/5}$.

Therefore, the left hand side of \eqref{eqMBSel1} is
\begin{equation}\label{eqSelsieve1temp}
\begin{array}{rcl}
&\ll & \displaystyle
Y^{(n+1)^2-1/5}\int_{\substack{s\in T'\\Yw(a_{11})\gg 1}}
w(a_{12})^{-1/5}s_1^{n(n+1)}s_2^{(n-1)(n+1)}\cdots s_n^{n+1}
\delta(s)d^\times s
\\[.175in]&\ll&\displaystyle
Y^{(n+1)^2-1/5}\int_{\substack{s\in T'\\Yw(a_{11})\gg 1}}
s_1^{(n-1)/5}\prod_{j=2}^{n}s_j^{(2n+2-2j)/5-(n+1)(n+1-j)(j-1)}d^\times s.
\end{array}
\end{equation}
In particular, the power of $s_i$ above is negative for all
$j\in\{2,\ldots,n\}$, and hence the integral over $s_2,\ldots,s_n$ is
absolutely bounded. The condition that $Yw(a_{11})\gg 1$ on the integrand
implies that we have $s_1\ll Y^{1/2n}$. Therefore the terms in
\eqref{eqSelsieve1temp} are 
\begin{equation*}
  \ll Y^{(n+1)^2-1/5}\int_{1\ll s_1\ll Y^{1/(2n)}}s_1^{(n-1)/5}d^\times s_1
  \ll Y^{(n+1)^2-1/5+(n-1)/(10n)} = Y^{(n+1)^2 - (n+1)/(10n)}.
\end{equation*}
Since $Y=X^{1/(n+1)}$, we obtain the result.
\end{proof}

\vspace{-.1in}

\begin{remark}
{\em Our use of the Selberg sieve saves a power of the smallest range of any coordinate. In the above proof, we fiber over $a_{11}$ because in the region of the main body close to the cusp, just before we enter the shallow cusp, the range of $a_{11}$ has size $\ll 1$. In this case, the Selberg sieve gives no saving at all. Once we fiber over $a_{11}$, the next smallest range is that of $a_{12}$. Implicit in our proof is an argument that either the range of $a_{12}$ is large, in which case the Selberg sieve gives the desired saving, or the number of pairs $(A,B)$ is automatically small.}
\end{remark}

\smallskip

\noindent {\bf Proof of Theorem \ref{th:evenmainbody}:} Recall from
\eqref{eq:evenMCDC} that we have
\begin{equation*}
  \cI_X^\main(\LL1 )\ll
  \sum_L\sum_{\cZ:a_{11}\notin\cZ}N(\LL1 ,L,\cZ,X),
\end{equation*}
where the second sum is over all saturated $\cZ$. Since $\cZ$ is
saturated and $a_{11}\not\in\cZ$, we have $\cZ\subset \cM$. If
$\ell_{k,1}=0$ for every $k=1,\ldots,n$, then $(1,0,\ldots,0)$ is in
the kernel of $B$ implying that the top row of $B$ is zero. The number
of such pairs $(A,B)$ has already been bounded in Proposition
\ref{propMBSel1}, and hence we may assume that $\ell_{n,1}\not\in
\cZ$. Fix a nondecreasing $n$-tuple $L$ of positive real numbers, and
a saturated $\cZ\subset\cM$ with $\ell_{n,1}\not\in\cZ$ such that
$N(\LL1 ,L,\cZ,X)\neq 0$.
We partition the integrand $T_\cZ(L,Y)$ into two parts: let $T_1$
denote the subset of $T_\cZ(L,Y)$ consisting of elements $s$ for
$s=(s_i)_i$ with $s_n\geq Y^\delta$, and let $T_2$ denote the subset
of elements $s$ with $1\ll s_n< Y^\delta$, where $\delta$
is a positive constant to be optimized later.

We first bound the contribution to $N(\LL1 ,L,\cZ,X)$ from $T_1$.
Since $Yw(a_{11})\gg 1$, we have
\begin{equation*}
  \#\bigl((s(Y\D)\times s(Y\D))\cap
  \LL1 \bigr)\leq \#\bigl((s(Y\D)\times s(Y\D))
  \cap W_{n+1}(\Z)\bigr)\ll Y^{(n+1)(n+2)}
\end{equation*}
for $s\in T_1$. Integrating over $T_1$ gives the bound
\begin{equation}\label{eq:intT1}
\begin{array}{rcl}
\displaystyle \int_{s\in T_1} \#\bigl((s(Y\D)\times s(Y\D))\cap
  \LL1 \bigr)\,\delta(s)d^\times s&\ll&
\displaystyle Y^{(n+1)(n+2)}\int_{s_1,\ldots,s_{n-1}\gg 1}
\int_{s_n\geq Y^{\delta}}\delta(s)d^\times s
\\[.1in]&\ll&\displaystyle
Y^{(n+1)(n+2)}\int_{s_n\geq Y^{\delta}}s_n^{-n(n+1)}d^\times s_n
\\[.1in] &\ll&\displaystyle
Y^{n+1-n(n+1)\delta}X^{n+1}.
\end{array}
\end{equation}

Next, we consider the contribution from $T_2$. Define the map
$\pi:\cZ_1\cap\cM\to\cM$ by 
\begin{equation*}
\pi(\ell_{ij}) = \left\{
  \begin{array}{rl}
    \ell_{n1}&\;{\rm if}\; j=1\mbox{ and }i\geq 2,\\
    \ell_{i,n+2-i} &\;{\rm otherwise}.
  \end{array}\right.
\end{equation*}
Since we have assumed that $N(\LL1 ,L,\cZ,X)\neq 0$, Lemma
\ref{lem:NUX} implies that $\cZ\subset\cZ_1$ and so the image of $\pi$
lies in $\cM\backslash\cZ$. Then for any $\alpha\in\cZ_1\cap \cM$ and
any $s\in T_\cZ(L,Y)$, we have $w_L(\pi(\alpha))\gg w_L(\alpha)$ and
$w_L(\pi(\alpha))\gg 1$. These inequalities along with
\eqref{eq:Lambda} and \eqref{eq:refine} imply that for any $s\in
T_\cZ(L,Y)$, the number $\# (s(Y\D) \cap \cS(L,s))$ of possible $B$'s
is
\begin{equation*}
\ll Y^{n(n+1)/2}w(\cZ\cap\cM)^{-1}\Bigl(\prod_{\substack{1\leq i< j\leq
    n\\i+j>n+1}} \frac{L_i}{L_j}\Bigr)\\
    \ll
    Y^{n(n+1)/2}\Bigl(\prod_{\substack{\ell\in \cZ_1\cap\cM\\\ell\neq \ell_{n1}}}
    \frac{w(\pi(\ell))}{w(\ell)}\Bigr)
\Bigl(\prod_{\substack{1\leq i< j\leq
    n\\i+j>n+1}} \frac{L_i}{L_j}\Bigr).
\end{equation*}
For each possible $B$, applying the Selberg sieve using Lemma
\ref{lemndist} gives us a bound of
\begin{equation*}
\ll Y^{(n+1)(n+2)/2}Y^{-1/5}w(a_{11})^{-1/5}
\end{equation*}
for the number of possible
choices for $A$.
Therefore, 
\begin{equation*}
  \#\bigl((s(Y\D)\times s(Y\D))\cap \LL1 \bigr)
  \ll Y^{-1/5}X^{n+1}w(a_{11})^{-1/5}
  \Bigl(\prod_{\substack{\ell\in \cZ_1\cap\cM\\\ell\neq \ell_{n1}}}
  \frac{w(\pi(\ell))}{w(\ell)}\Bigr)
\Bigl(\prod_{\substack{1\leq i < j\leq n\\i+j>n+1}}
\frac{L_i}{L_j}\Bigr),
\end{equation*}
for $s\in T_2$.  We compute the ratio of these weights: For any $i\geq
2$ and $j=1$, we have
\begin{equation*}
  \frac{w(\pi(\ell_{i1}))}{w(\ell_{i1})}=\frac{w(\ell_{n1})}{w(\ell_{i1})}
  =\frac{L_n}{L_i}.
\end{equation*}
For any other $i,j$, we have
\begin{equation*}
  \frac{w(\pi(\ell_{ij}))}{w(\ell_{ij})}=\frac{w(\ell_{i,n+2-i})}{w(\ell_{ij})}
  =\frac{t_j}{t_{n+2-i}}.
\end{equation*}
As the $L_i$ are nondecreasing and positive, we multiply by
the Haar measure character $\delta(s)$ to~obtain
\begin{equation*}
\begin{array}{rcl}
&&\displaystyle
w(a_{11})^{-1/5}\Bigl(\prod_{\substack{\ell\in \cZ_1\cap\cM\\\ell\neq \ell_{n1}}}
\frac{w(\pi(\ell))}{w(\ell)}\Bigr)
\Bigl(\prod_{\substack{1\leq i < j\leq n\\i+j>n+1}}
\frac{L_i}{L_j}\Bigr)\,\delta(s)\\[.2in]
&\leq&\displaystyle w(a_{11})^{-1/5}
\Bigl(\prod_{\substack{i,j\geq 1\\i+j\leq n+1}}\frac{t_j}{t_{n+2-i}}\Bigr)\Bigl(
\prod_{i=2}^{n}\frac{t_{n+2-i}}{t_1}\Bigr)
\Bigl(\prod_{1\leq j < i \leq n+1}\frac{t_i}{t_j}\Bigr)
\\[.2in]&=&\displaystyle w(a_{11})^{-1/5}
\prod_{i=2}^{n}\frac{t_{n+2-i}}{t_1}
\\[.2in]&=&\displaystyle
w(a_{11})^{-1/5}s_1^{-(n+1)(n-1)}s_2^{-(n+1)(n-2)}\cdots s_{n-1}^{-(n+1)}.
\end{array}
\end{equation*}
The powers of $s_i$ in the above expression are negative for $1\leq
i\leq n-1$, while the power of $s_n$ is~$2/5$.
Integrating over $T_2$ now gives the bound
\begin{equation}\label{eq:intT2}
\begin{array}{rcl}
\displaystyle \int_{s\in T_1} \#\bigl((s(Y\D)\times s(Y\D))\cap
  \LL1 \bigr)\,\delta(s)d^\times s&\ll&
\displaystyle Y^{-1/5}X^{n+1}\int_{1\ll s_n\ll Y^\delta}s_n^{2/5}d^\times s_n
\\[.175in]&\ll&\displaystyle
Y^{-1/5+(2\delta)/5}X^{n+1}. 
\end{array}
\end{equation}

Combining \eqref{eq:intT1} and \eqref{eq:intT2} and choosing 
$\delta=\frac{5n+6}{5n^2 + 5n+2}$ yields
\begin{equation*}
N(\LL1 ,L,\cZ,X)\ll X^{n+1- \frac{n-2}{5n^2 + 5n+2}}.
\end{equation*}
The summation of this bound over the $O(1)$ different possible $\cZ$'s and the
$O(Y^\epsilon)$ different possible~$L$'s, in conjunction with the bound in
Proposition \ref{propMBSel1}, implies Theorem
\ref{th:evenmainbody}. $\Box$

\subsection{Bounding the number of distinguished elements
  in the shallow cusp}\label{sec:evencusp}

In this subsection, we bound the number of distinguished elements
having large $q$-invariant that lie in the shallow cusp of the
fundamental domain.
\begin{theorem}\label{th:evenmaincusp}
Let $\eta>0$ be any real number. Assume that $M>X^{\eta}$. Then 
\begin{equation*}
{\cI_X^{\scusp}(\LLM )=O\bigl(X^{n+1-\min(\eta,1)/(22n^6)}\bigr)}.
\end{equation*}
\end{theorem}
We will take $\eta = 1/4$ when we prove Theorem \ref{thm:mainestimate} in \S\ref{sec:6.5}.

\subsubsection{A preliminary bound of
  $O_\epsilon(X^{n+1+\epsilon})$}\label{sec:6.3.1}

We again use \eqref{eq:evenMCDC} to write
\begin{equation*}
\cI_X^\scusp(\LLM )\ll\sum_{L,\cZ}N(\LLM ,L,\cZ,X),
\end{equation*}
where the sum is over nondecreasing $n$-tuples $L=(L_1,\ldots,L_n)$
of positive real numbers that partition the region
\smash{$\{(\mu_1,\ldots,\mu_n)\in[Y^{-\Theta_1},Y^{\Theta_2}]^n:\mu_1\leq \mu_2\leq
\ldots\leq \mu_n\}$} into dyadic ranges, and over saturated
$\cZ\subset\cK\cup\cM$ such that $a_{11}\in\cZ$ and
$a_{g+1,g+1}\not\in\cZ$. By Lemma \ref{lem:NUX}, we have 
\smash{$N(\LLM ,L,\cZ,X)>0$ only when $\cZ\subset\cZ_1$}, which we
henceforth assume. 

For $k\in\{0,\ldots,g\}$, define the map
$\pi_k:\cZ_1\to\cK\cup\cM$ by
\begin{equation*}
\pi_k(a_{ij}) = a_{n+1-j,j},\qquad
\pi_k(\ell_{ij}) = \left\{
  \begin{array}{rl}
    \ell_{n+1-j,j}&\;{\rm if}\; i>j\mbox{ and }j\leq k,\\
    \ell_{i,n+2-i} &\;{\rm otherwise}.
  \end{array}\right.
\end{equation*}
We define the auxiliary set $\cZ^*$ by
\begin{equation*}
\cZ^*\,=\,
\{a_{ij}\mid i\leq j, \,i+j\leq n\}\cup 
\{\ell_{ij}\mid i\leq j,\,i+j\leq n+1\} \,=\,
\cZ_1\backslash
\{\ell_{ij}\mid i> j,\,i+j\leq n+1\}.
\end{equation*}
Then, when restricted to $\cZ^*\subset\cZ_1$, the functions $\pi_k$
are equal for every $k$.

\begin{lemma}\label{lem:pitrick1}
For any $k\in\{0,\ldots,g\}$, we have
\begin{equation}\label{eq:pitrick1}
\Big(\prod_{\alpha\in \cZ^*}
\frac{w(\pi_k(\alpha))}{w(\alpha)} \Big)\,\delta(s)=1.
\end{equation}
\end{lemma}

\begin{proof}
We directly compute
\begin{equation*}
\begin{array}{rcl}
\displaystyle \prod_{\alpha\in \cZ^*}
\frac{w(\pi_k(\alpha))}{w(\alpha)}
&=&\displaystyle
\Bigl(\prod_{\substack{i\leq j\\i+j<n+1}}\frac{w(a_{n+1-j,j})}{w(a_{ij})}\Bigr)
\Bigl(\prod_{\substack{i\leq j\\i+j<n+2}}\frac{w(\ell_{i,n+2-i})}{w(\ell_{ij})}\Bigr)
\\[.1in]&=&\displaystyle
\Bigl(\prod_{\substack{i\leq j\\i+j<n+1}}\frac{t_i}{t_{n+1-j}}\Bigr)
\Bigl(\prod_{\substack{i\leq j\\i+j<n+2}}\frac{t_j}{t_{n+2-i}}\Bigr)
\\[.1in]&=&\displaystyle
\Bigl(\prod_{\substack{i<r\\i+r\leq n+1}}\frac{t_i}{t_{r}}\Bigr)
\Bigl(\prod_{\substack{j< r\\j+r\geq n+2}}\frac{t_j}{t_{r}}\Bigr),
\end{array}
\end{equation*}
which is $\delta(s)^{-1}$.
\end{proof}

Fix a saturated set $\cZ\subset\cZ_1$ such that $a_{11}\in\cZ$, $a_{g+1,g+1}\notin\cZ$ and $N(\LLM ,L,\cZ,X)>0$. Let $k\in\{1,\ldots,g\}$ be the largest integer such
that $a_{kk}\in\cZ$. 
Then we have the following results.
\begin{lemma}\label{lem:piknotinZ}
Let $\cZ$ and $k$ be as above. Then for every $\alpha\in \cZ$, we have
$\pi_k(\alpha)\notin \cZ$. In particular, for any $s\in T_\cZ(L,Y)$, we have $Yw(\pi_k(\alpha))\gg 1.$
\end{lemma}
\begin{proof}
Since $a_{n+1-j,j}\not\in\cZ_1$ for any $j$ and $\cZ\subset\cZ_1$, we have
$\pi_k(a_{ij})\not\in\cZ$ for any
$a_{ij}\in\cZ$. Moreover, since $a_{jj}\in\cZ$ for every $j\leq k$, it
follows from Lemma \ref{lem:NUX} that $\ell_{n+1-j,j}\not\in\cZ$. Furthermore,
$\ell_{i,n+2-i}\not\in\cZ_1$. Hence $\pi_k(\ell_{ij})\not\in\cZ$ for any
$\ell_{ij}\in\cZ$.
\end{proof}

\vspace{-.1in}
\begin{lemma}\label{lem:pitrick}
Let $\cZ$ and $k$ be as above. Then, uniformly for $s\in T_\cZ(L,Y)$,
 we have
\begin{equation}\label{eq:pitrick}
\Bigl(\prod_{\alpha\in \cZ} \frac{w(\pi_k(\alpha))}{w(\alpha)}\Bigr)
  \Big(\prod_{\substack{1\leq i < j\leq n\\ i+j> n+1}}\frac{L_i}{L_j}\Big)
  \,\delta(s) \ll 1.
\end{equation}
\end{lemma}

\begin{proof}
Since we have
\begin{equation*}
\frac{w(a_{n+1-j,j})}{w(a_{ij})}=\frac{t_i}{t_{n+1-j}},\quad
\frac{w(\ell_{i,n+2-i})}{w(\ell_{ij})}=\frac{t_j}{t_{n+2-i}},\quad
\frac{w(\ell_{n+1-j,j})}{w(\ell_{ij})}=\frac{L_{n+1-j}}{L_i},
\end{equation*}
it follows that $w(\pi_k(\alpha))/w(\alpha)\gg 1$ for every $k$,
$\alpha\in\cZ_1$, and $s\in T_\cZ(L,Y)$. Thus, by adding elements in $\cZ_1$
to $\cZ$, if necessary, we can assume that $\cZ$ is equal to
\begin{equation*}
  \begin{array}{rcl}
    \cZ_0&=&\bigl\{a_{ij}:i\leq j,\,i\leq k,\,i+j\leq n\bigr\}\cup
    \bigl\{\ell_{ij}:i>j> k,\,i+j\leq n+1\bigr\}
    \\[.05in]&&
    \cup\,\bigl\{\ell_{ij}:i>j,\,j\leq k,\,i+j\leq n+1\bigr\}
    \cup\bigl\{\ell_{ij}:i\leq j,\,i+j\leq n+1\bigr\}.
  \end{array}
\end{equation*}
Denote the four sets on the right hand side of the above equation as
$S_1$, $S_2$, $S_3$, and $S_4$, respectively. For an element
$\ell_{ij}\in S_2$, we have
\begin{equation*}
  \frac{w(\pi_k(\ell_{ij}))}{w(\ell_{ij})}=\frac{w(\ell_{i,n+2-i})}{w(\ell_{ij})}
  =\frac{t_j}{t_{n+2-i}}=\frac{w(\pi_k(a_{j,i-1}))}{w(a_{j,i-1})}.
\end{equation*}
Therefore, 
\begin{equation*}
\begin{array}{rcl}
\displaystyle
\Bigl(\prod_{\alpha\in \cZ_0} \frac{w(\pi_k(\alpha))}{w(\alpha)}\Bigr)
\Big(\prod_{\substack{1\leq i < j\leq n\\ i+j> n+1}}\frac{L_i}{L_j}\Big)
\,\delta(s)  
&=&\displaystyle
\Bigl(\prod_{\alpha\in \cZ^*} \frac{w(\pi_k(\alpha))}{w(\alpha)}\Bigr)
\Bigl(\prod_{\alpha\in S_3} \frac{w(\pi_k(\alpha))}{w(\alpha)}\Bigr)
\Big(\prod_{\substack{1\leq i < j\leq n\\ i+j> n+1}}\frac{L_i}{L_j}\Big)
\,\delta(s)
\\[.025in]&=&\displaystyle
\Bigl(\prod_{\substack{i>j,\,j\leq k\\i+j\leq n+1}}\frac{L_{n+1-j}}{L_i}\Bigr)
\Big(\prod_{\substack{1\leq i < j\leq n\\ i+j> n+1}}\frac{L_i}{L_j}\Big)
\\[.25in]&=&\displaystyle
\prod_{\substack{1\leq i < j\leq n\\ i+j> n+1\\ j< n +1 - k}}\frac{L_i}{L_j}
\\[.25in]&\leq &1,
\end{array}
\end{equation*}
where the second equality follows from Lemma \ref{lem:pitrick1}, and the last inequality follows because the $L_i$'s are nondecreasing. 
\end{proof}

Proposition \ref{prop:NUX} and Lemmas \ref{lem:piknotinZ}
and \ref{lem:pitrick} thus yield the bound $$N(\LLM ,L,\cZ,X) \ll X^{n+1}\int_{1\ll s_1,\ldots,s_n\ll X^\Theta }d^\times s \ll_\epsilon X^{n+1+\epsilon}.$$  
We now work towards obtaining a power saving. 

\subsubsection{Strategy towards a power saving} 
In light of Proposition \ref{prop:NUX}, it is enough to have a bound of the form
\begin{equation}\label{eq:pitrickgood}
Y^{-\#\cZ}w(\cZ)^{-1}
\Bigl(\prod_{\substack{1\leq i < j\leq n\\ i+j> n+1}} \frac{L_i}{L_j}\Bigr)
\,\delta(s) \ll X^{-\delta}
\end{equation}
for some $\delta > 0$, for all $s\in T_\cZ(L,Y).$ By modifying $\pi_k$ on a certain subset of $\cZ$, we are able to obtain \eqref{eq:pitrickgood} except for some $s\in T_\cZ(L,Y)$ satisfying some special conditions. We then consider the contribution from these special $s$ using a different count for $\#\bigl((s(Y\D)\times s(Y\D))\cap \LLM \bigr)$.

More precisely, let $\cK_1 := \{a_{1j}: 1\leq j \leq g+2\}$. Then $\cK_1$ consists
exactly of those $\alpha\in \cK$ such that the exponent of every $s_i$
is negative in $w(\alpha)$. As such, one expects that the hardest case is when $\cZ = \cK_1$. We show first in Lemma \ref{lem:reducetoK1} how to reduce to considering only $\cZ\cap \cK_1$.

\begin{lemma}\label{lem:reducetoK1}
Let $\cZ\subset\cZ_1$ be saturated with $a_{11}\in\cZ$
, $a_{g+1,g+1}\not\in\cZ$ and $N(\LLM ,L,\cZ,X) > 0$. For any $\cZ'\subset\cK_1$ and any $s\in T'$, we write $$I(\cZ',s) = Y^{-\#\cZ'}
w(\cZ')^{-1}\prod_{i=1}^{g} s_i^{-(n+1)(g+2)}
\prod_{i=g+1}^{n-1} s_i^{-(n+1)(n-i)}.$$
Then for any $s\in T_\cZ(L,Y)$, we have 
$$Y^{-\#\cZ}w(\cZ)^{-1}
\Bigl(\prod_{\substack{1\leq i < j\leq n\\ i+j> n+1}} \frac{L_i}{L_j}\Bigr)
\,\delta(s)\ll  I(\cZ\cap\cK_1,s).$$
\end{lemma}
We then prove in Lemma \ref{lem:ZK1bound} the following bound for $I(\cZ\cap \cK_1)$ when $\cZ\cap\cK_1$ is a proper subset of $\cK_1$, which gives a bound of the form \eqref{eq:pitrickgood} when $s_n\ll Y^{1/2-\delta}$.

\begin{lemma}\label{lem:ZK1bound}
Let $\cZ\subset\cZ_1$ be saturated with $a_{11}\in\cZ$
, $a_{g+1,g+1}\not\in\cZ$ and $N(\LLM ,L,\cZ,X) > 0$. Suppose $\cZ\cap\cK_1\neq \cK_1$. For any $s\in T_\cZ(L,Y)$, if $I(\cZ\cap\cK_1,s) \gg Y^{-2\delta}$, then
$s_n\gg Y^{1/2 - \delta}.$
\end{lemma}
In the case where $s_n\gg Y^{1/2-\delta}$, the Haar measure turns out to be very small so we may simply ignore the singularity condition of $B$ and prove the following bound.

\begin{lemma}\label{lem:ZK1boundnosing}
Let $\cZ\subset\cZ_1$ be saturated with $a_{11}\in\cZ$
, $a_{g+1,g+1}\not\in\cZ$ and $N(\LLM ,L,\cZ,X) > 0$. Suppose $\cZ\cap\cK_1\neq \cK_1$. 
Then for any $s\in T_\cZ(L,Y)$ with $s_n\gg Y^{1/2-\delta}$,
$$\#\bigl((s(Y\D)\times s(Y\D))\cap W(\Z)\bigr)\,\delta(s) \ll Y^{(\frac{n}{2}+2)(n+1)+n(2n^2+9n+9)\delta}.$$
\end{lemma}
Therefore, by taking $\delta = (n-2)/(4n^2 + 14n + 4)$, we obtain the following result from Proposition~\ref{prop:NUX} and Lemmas \ref{lem:reducetoK1}, \ref{lem:ZK1bound}, and \ref{lem:ZK1boundnosing}:

\begin{proposition}\label{prop:cZnotcK}
Let $\cZ\subset\cZ_1$ be saturated with $a_{11}\in\cZ$
, $a_{g+1,g+1}\not\in\cZ$ and $N(\LLM ,L,\cZ,X) > 0$.  Suppose $\cZ\cap\cK_1\neq \cK_1$. Then 
\begin{equation*}
N(\LLM ,L,\cZ,X)\ll X^{n+1-\frac{n-2}{2(n+1)(n^2 + 7n + 7)}}.
\end{equation*}
\end{proposition}

We next handle the case $\cK_1\subset\cZ$. We give necessary conditions in Lemma \ref{lem:sisn-iclose} on $s$ so that a bound of the form \eqref{eq:pitrickgood} does not hold.
\begin{lemma}\label{lem:sisn-iclose}
Let $\cZ\subset\cZ_1$ be saturated with $a_{11}\in\cZ$
, $a_{g+1,g+1}\not\in\cZ$ and $N(\LLM ,L,\cZ,X) > 0$. Suppose $\cK_1\subset\cZ$. 
For any $s\in T_\cZ(L,Y)$, if $I(\cK_1,s)\gg X^{-\delta}$, then
\begin{equation}\label{eq:siconditions}
\begin{array}{rcccl}
\displaystyle Y^{-\delta} &\ll&\displaystyle   \frac{s_i}{s_{n-i}} &\ll& Y^{\delta},\quad\mbox{for }\quad i=1,\ldots,g-1
\\[.05in]\displaystyle Y^{1/2-(g/2)         \delta}\RR^{-1}&\ll&s_g&\ll& Y^{1/2+3g\delta}\RR^{-1}
\\[.075in]\displaystyle
1&\ll&s_{g+1}&\ll&Y^{\delta}
\\[.05in]\displaystyle Y^{1/2-\delta}\RR^{-1}&\ll&s_{g+2}&\ll& Y^{1/2+g\delta}\RR^{-1}
\end{array}
\end{equation}
where $$\RR = \prod_{i=g+3}^n s_i \ll Y^{1/2+3g\delta}.$$
\end{lemma}
Note that the coefficients of $\delta$ in the exponents in the above bounds are not optimal and are simply chosen to make the formula look nice. The optimal coefficients can be obtained from the proof.

When $s$ satisfies \eqref{eq:siconditions}, we give further conditions in Lemma \ref{lem:crazyg} on $s$ so that simply using the Haar measure and ignoring the singularity condition by counting all symmetric matrices is not enough for a power saving. 

\begin{lemma}\label{lem:crazyg}
Let $\cZ\subset\cZ_1$ be saturated with $a_{11}\in\cZ$
, $a_{g+1,g+1}\not\in\cZ$ and $N(\LLM ,L,\cZ,X) > 0$. Suppose $\cK_1\subset\cZ$. For any $s\in T_\cZ(L,Y)$, if $$I(\cK_1,s)\gg X^{-\delta},\qquad \mbox{and} \qquad \#\bigl((s(Y\D)\times s(Y\D))\cap W(\Z)\bigr)\,\delta(s)\gg X^{n+1-\delta},$$
then 
\begin{equation}\label{eq:s_iwtf}
    s_i \ll Y^{258g^3\delta}\qquad\mbox{ for }\quad i = g+3,\ldots,n.
\end{equation}
\end{lemma}

To obtain a further saving, we need to use the $|q|$-invariant! 
\begin{lemma}\label{lem:crazyq}
Suppose $M>X^\eta$ where $\eta>0$ is some fixed constant. Let $\cZ\subset\cZ_1$ be saturated with $a_{11}\in\cZ$
, $a_{g+1,g+1}\not\in\cZ$ and $N(\LLM ,L,\cZ,X) > 0$. Suppose $\cK_1\subset\cZ$. Then for $\delta < \min(\eta,1)/(1355g^6)$ and any $s\in T_\cZ(L,Y)$ such that \eqref{eq:siconditions} and \eqref{eq:s_iwtf} hold, we have
$$\#\bigl((s(Y\D)\times s(Y\D))\cap \LLM \bigr)\,\delta(s) \ll X^{n+1+514g^3\delta - 1/2}.$$
\end{lemma}
Therefore, by taking $\delta = 64\min(\eta,1)/(1355n^6)$, we obtain the following result from Proposition~\ref{prop:NUX} and Lemmas \ref{lem:reducetoK1}, \ref{lem:sisn-iclose}, \ref{lem:crazyg},  and \ref{lem:crazyq}.

\begin{proposition}\label{prop:cZcK}
Suppose $M>X^\eta$ where $\eta>0$ is some fixed constant. Let $\cZ\subset\cZ_1$ be saturated with $a_{11}\in\cZ$, $a_{g+1,g+1}\not\in\cZ$ and $N(\LLM ,L,\cZ,X) > 0$. Suppose $\cK_1\subset\cZ$. Then 
\begin{equation*}
N(\LLM ,L,\cZ,X)\ll X^{n+1-64\min(\eta,1)/(1355n^6)}.
\end{equation*}
\end{proposition}

Theorem \ref{th:evenmaincusp} then follows immediately from \eqref{eq:evenMCDC}, Proposition \ref{prop:cZnotcK},  Proposition \ref{prop:cZcK} and summing over the $O(1)$ different possible $\cZ$'s and the
$O(Y^\epsilon)$ different possible $L$'s.

\subsubsection{Proofs of Lemmas \ref{lem:reducetoK1}, \ref{lem:ZK1bound}, \ref{lem:ZK1boundnosing}, \ref{lem:sisn-iclose}, \ref{lem:crazyg} and \ref{lem:crazyq}.}

We fix a saturated \smash{$\cZ\subset\cZ_1$} with $a_{11}\in\cZ$, $a_{g+1,g+1}\not\in\cZ$ and \smash{$N(\LLM ,L,\cZ,X) > 0$}.

\bigskip

\noindent\textbf{Proof of Lemma \ref{lem:reducetoK1}}: Recall that $\cK_1 := \{a_{1j}: 1\leq j \leq g+2\}$. Let $k\in\{1,\ldots,g\}$ be the largest integer such that
$a_{kk}\notin \cZ$. Then, applying Lemma \ref{lem:pitrick} to the
saturated set $\cZ\cup\cK_1$, we have
\begin{equation*}
\Bigl(\prod_{\alpha\in \cZ\backslash \cK_1} \frac{w(\pi_k(\alpha))}{w(\alpha)}\Bigr)
\Bigl(\prod_{\substack{1\leq i < j\leq n\\ i+j> n+1}} \frac{L_i}{L_j}\Bigr)
\,\delta(s)
\ll \prod_{\alpha \in \cK_1} \frac{w(\alpha)}{w(\pi_k(\alpha))}
=\prod_{i=1}^{g} s_i^{-(n+1)(g+2)}\prod_{i=g+1}^{n-1} s_i^{-(n+1)(n-i)}.
\end{equation*}
Hence, by Lemma \ref{lem:piknotinZ}, we obtain for any $s\in T_\cZ(L,Y)$, 
\begin{equation*}\label{eqNLprelimbound}
\begin{array}{rcl}
\displaystyle Y^{-\#\cZ}w(\cZ)^{-1}
\Bigl(\prod_{\substack{1\leq i < j\leq n\\ i+j> n+1}} \frac{L_i}{L_j}\Bigr)
\,\delta(s)&\ll&\displaystyle
Y^{-\#\cZ}w(\cZ)^{-1}\Bigl(\prod_{\alpha\in\cZ\backslash\cK_1}
Yw(\pi_k(\alpha))\Bigr)
\Bigl(\prod_{\substack{1\leq i < j\leq n\\ i+j> n+1}} \frac{L_i}{L_j}\Bigr)
\,\delta(s)
\\[.2in] &\ll&\displaystyle
Y^{-\#(\cZ\cap\cK_1)}w(\cZ\cap \cK_1)^{-1}\Bigl(\prod_{\alpha\in\cZ\backslash\cK_1}
\frac{w(\pi_k(\alpha))}{w(\alpha)}\Bigr)
\Bigl(\prod_{\substack{1\leq i < j\leq n\\ i+j> n+1}} \frac{L_i}{L_j}\Bigr)
\,\delta(s)
\\[.1in]&=&\displaystyle
Y^{-\#(\cZ\cap\cK_1)}
w(\cZ\cap\cK_1)^{-1}\prod_{i=1}^{g} s_i^{-(n+1)(g+2)}
\prod_{i=g+1}^{n-1} s_i^{-(n+1)(n-i)}\\[.2in]&=&\displaystyle I(\cZ\cap\cK_1,s),
\end{array}
\end{equation*} as desired. $\Box$

\vspace{.1in}

Note that a direct computation yields
\begin{equation}\label{eqIcK1}
I(\cK_1,s) =\displaystyle 
  Y^{-(g+2)}\prod_{j=1}^{g+2}t_{n+1-j}t_j=
Y^{-(g+2)}\frac{t_{g+1}t_{g+2}}{t_{n+1}}.
\end{equation}

\medskip

\noindent\textbf{Proof of Lemma \ref{lem:ZK1bound}}: Since $\cZ$ is saturated and $\cZ\cap\cK_1\neq\cK_1$, we have $\cZ\cap \cK_1 =
\{a_{11},\ldots,a_{1j}\}$ for  some $j=1,\ldots,g+1$. Since
$a_{g+1,g+1}$ and $a_{1,g+2}$ do not belong to $\cZ$, we have for
$s\in T_\cZ(L,Y)$:
\begin{equation*}
\begin{array}{rcl}
I(\{a_{11},\ldots,a_{1,g+1}\},s)&=&YI(\cK_1,s)w(a_{1,g+2})
\\[.05in]&\ll&
Y^{g+1}I(\cK_1,s)w(a_{g+1,g+1})w(a_{1,g+2})^{g}
\\[.025in]&\ll&\displaystyle
Y^{-1}\frac{1}{t_1^gt_{g+1}t_{g+2}^{g-1}t_{n+1}}
\\[.15in]&\ll&\displaystyle
Y^{-1}s_n^2,
\end{array}
\end{equation*}
since the powers of the $s_i$'s in the third line are negative for
$i<n$.

Similarly, for any $j=1,\ldots,g$, we compute
\begin{equation*}
I(\{a_{11},\ldots,a_{1j}\},s)(Yw(a_{1,j+1}))^{j-1}\ll Y^{-1}s_n^2,
\end{equation*}
as desired. $\Box$

\bigskip

\noindent\textbf{Proof of Lemma \ref{lem:ZK1boundnosing}}: Suppose now $s_n\gg Y^{1/2-\delta}$. 
First note, that the inequality
\begin{equation}\label{eq:tn+1}
1\ll Y^{g+1}w(a_{1,n})w(a_{2,n-1})\cdots w(a_{g+1,g+2}) = Y^{g+1}\prod_{i=1}^n s_i^{-i}
\end{equation}
implies that we have
\begin{equation}\label{eq:prodsii}
\prod_{j=1}^{n-1} s_j^{j}\ll Y^{g+1}s_n^{-n}\ll Y^{n\delta}.
\end{equation}
Since each $s_i\gg 1$, we also have $s_n\ll Y^{1/2}$ by \eqref{eq:prodsii}. Hence 
\begin{equation*}
t_1^{-1}\ll t_2^{-1}\ll\ldots\ll t_{n}^{-1} = s_n^{-1}\prod_{j=1}^{n-1} s_j^{j}\ll Y^{-1/2+(n+1)\delta};\quad
t_{n+1}^{-1}=s_n^n\prod_{j=1}^{n-1} s_j^{j}\ll Y^{n/2 + n\delta}.
\end{equation*}
Thus
\begin{equation*}
Yw(a_{ij})=Yw(b_{ij})=\frac{Y}{t_it_j}\,\ll\,
\begin{cases}
  Y^{(2n+2)\delta}&\mbox{ if }i\leq j\leq n,\\
  Y^{(n+1)/2+(2n+1)\delta}&\mbox{ if }i\leq n,\,j= n+1,\\
  Y^{n+1+2n\delta}&\mbox{ if }i=j=n+1.
\end{cases}
\end{equation*}
Multiplying these weights together and applying Proposition
\ref{prop-davenport} gives the estimate
\begin{equation}\label{eq:sYDWnZ}
  \#\bigl(s(Y\D)\times s(Y\D)\cap W_{n+1}(\Z)\bigr)
  \ll Y^{(n+2)(n+1) + 2n(n+2)^2\delta}.
\end{equation}
Meanwhile, in this region where $s_n\geq X^{1/2-\delta}$, the quantity
$\delta(s)$ satisfies
\begin{equation}\label{eq:deltassss}
  \delta(s)=\prod_{k=1}^{n}s_k^{-(n+1)k(n+1-k)} \ll s_n^{-n(n+1)}\ll
  Y^{-n(n+1)/2 + n(n+1)\delta}.
\end{equation}
Multiplying the bounds in \eqref{eq:sYDWnZ} and \eqref{eq:deltassss} together yields
$$\#\bigl((s(Y\D)\times s(Y\D))\cap W_{n+1}(\Z)\bigr)\,\delta(s) \ll Y^{(\frac{n}{2}+2)(n+1)+n(2n^2+9n+9)\delta},$$
as desired. $\Box$

\bigskip

\noindent\textbf{Proof of Lemma \ref{lem:sisn-iclose}}: Suppose now $\cK_1\subset \cZ\subset \cZ_1$ and $I(\cK_1,s) \gg X^{-\delta}$ for some $s\in T_\cZ(L,Y)$. We prove first that for any $i = 1,\ldots,g-1$, we have 
\begin{equation}\label{eq:sisn-i}
Y^{-\delta}\ll\frac{s_i}{s_{n-i}}\ll Y^{\delta}.
\end{equation}
Indeed, since $a_{j,n+1-j}\notin\cZ$ for all $j$, we have from \eqref{eqIcK1} that, for any $k = 1,\ldots,g$,
\begin{equation*}
\begin{array}{rcl}
I(\cK_1,s)&\ll& I(\cK_1,s) Y^{g+2}w(a_{k,n+1-k})^gw(a_{g+1,g+2})^2
\\[.1in]&\ll&\displaystyle
\frac{t_{g+1}t_{g+2}}{t_{n+1}}\frac{1}{t_j^{g}t_{g+1}^2t_{g+2}^2t_{n-k+1}^g}
\\[.175in]&=&\displaystyle
\frac{t_1\cdots t_g}{t_k^g}\frac{t_n\cdots t_{g+3}}{t_{n-k+1}^g}
\\[.175in]&=&\displaystyle
\frac{t_1}{t_g}\cdots \frac{t_{g-1}}{t_g}\left(\frac{t_g}{t_k}\right)^g\frac{t_{g+4}}{t_{g+3}}\cdots \frac{t_n}{t_{g+3}}\left(\frac{t_{g+3}}{t_{n-k+1}}\right)^g.
\end{array}
\end{equation*}
Hence 
\begin{equation*}
\begin{array}{rcl}
I(\cK_1,s)&\ll& \left(s_1s_2^2\cdots s_{g-1}^{g-1}(s_ks_{k+1}\cdots s_{g-1})^{-g}s_{n-1}^{-1}s_{n-2}^{-2}\cdots s_{g+3}^{-(g-1)}(s_{g+3}s_{g+4}\cdots s_{n-k})^g)\right)^{n+1}
\\[.15in]&=&\displaystyle
\left(\prod_{i=1}^{g-1}\left(\frac{s_i}{s_{n-i}}\right)^{i(n+1)}\right)\left(\prod_{i=k}^{g-1}\left(\frac{s_i}{s_{n-i}}\right)^{-g(n+1)}\right).
\end{array}
\end{equation*}
Denote the product of the two factors in the final line by $J_k$. Then
$$
  \prod_{k=1}^g J_k= 1\qquad\mbox{and}\qquad
  \frac{J_{i+1}}{J_{i}}= \left(\frac{s_i}{s_{n-i}}\right)^{g(n+1)}\,\mbox{ for }i=1,\ldots,g-1.
$$
Since, by assumption,  $I(\cK_1,s)\gg X^{-\delta}$, we have $J_k\gg Y^{-(n+1)\delta}$ for every $k=1,\ldots,g$.
Therefore, for every $i=1,\ldots,g-1$, we have 
\begin{equation*}
\begin{array}{rcccccl}\displaystyle
  \frac{s_i}{s_{n-i}}&=&\displaystyle \left(\frac{J_{i+1}}{J_i}\right)^{\frac{1}{g(n+1)}}&=& \left(J_1\cdots J_{i-1}\cdot J_{i+1}^2\cdot J_{i+2}\cdots J_g\right)^{\frac{1}{g(n+1)}}&\gg&Y^{-\delta};
  \\[.1in]
  \displaystyle
  \frac{s_i}{s_{n-i}}&=&\displaystyle \left(\frac{J_{i}}{J_{i+1}}\right)^{-\frac{1}{g(n+1)}}&=& \left(J_1\cdots J_{i-1}\cdot J_{i}^2\cdot J_{i+2}\cdots J_g\right)^{-\frac{1}{g(n+1)}}&\ll&Y^{\delta}.
\end{array}
\end{equation*}
The claimed bound \eqref{eq:sisn-i} follows.

By \eqref{eq:tn+1} and 
\eqref{eq:sisn-i}, we have
\begin{equation}\label{eq:sgsggsggg}
s_g^gs_{g+1}^{g+1}s_{g+2}^{g+2} \ll Y^{g+1}\prod_{i=1}^{g-1}\left(\frac{s_i}{s_{n-i}}\right)^{-i}\prod_{i=g+3}^ns_i^{-n} \ll Y^{g+1+\frac{g(g-1)}{2}\delta}\cdot \RR^{-n}
\end{equation}
where $$\RR=\prod_{i=g+3}^ns_i.$$
We next prove the desired lower bounds:
\begin{equation}\label{eq:desiredlower}
s_g\gg Y^{1/2-(g/2)\delta}\RR^{-1};\quad s_{g+1}\gg 1;\quad
s_{g+2}\gg Y^{1/2-\delta}\RR^{-1}.
\end{equation}
The bound $s_{g+1}\gg1$ follows from the definition of $T'$. For the bounds on $s_g$ and $s_{g+2}$, we use the assumption that $a_{g+1,g+1}\notin\cZ$ and the computation of $I(\cK_1,s)$ in \eqref{eqIcK1} to obtain
\begin{equation*}
\begin{array}{rcl}
I(\cK_1,s) &\ll& I(\cK_1,s)Y^{1/2}w(a_{g+1,g+1})^{1/2}
\\[.05in]&=&\displaystyle Y^{-(n+1)/2}\frac{t_{g+2}}{t_{n+1}}
\\[.15in]&=&\displaystyle
Y^{-(n+1)/2}s_{g+2}^{n+1}\,\RR^{n+1},
\end{array}
\end{equation*}
which along with $I(\cK_1,s)\gg Y^{-(n+1)\delta}$ implies the desired lower bounds on $s_{g+2}$; 
and
\begin{equation*}
\begin{array}{rcl}
I(\cK_1,s)&\ll& \displaystyle I(\cK_1,s)Y^2w(a_{g+1,g+1})^2
\\[.025in]&=&
\displaystyle Y^{-g}\frac{t_{g+2}}{t_{g+1}^3t_{n+1}}
\\[.1in]&=&\displaystyle
Y^{-g}\Big(\prod_{i=1}^g s_i^{3i}\Big)\, s_{g+1}^{-3(g+2)}\Big(\prod_{i=g+2}^{n}s_i^{-2n-2+3i}\Big)
\\[.05in]&\ll&\displaystyle
Y^{-g+\frac{3g(g-1)}{2}\delta}\Big(\prod_{i=1}^{g-1} s_{n-i}^{3i}\Big) s_g^{3g}s_{g+1}^{-3(g+2)}s_{g+2}^{-g}
\Big(\prod_{i=g+3}^{n-1}s_i^{-2n-2+3i}\Big) s_n^{2g}
\\[.1in]&\ll&\displaystyle
Y^{-g+\frac{3g(g-1)}{2}\delta}s_g^{3g}s_{g+1}^{-3(g+2)}s_{g+2}^{-g}\,\RR^{2g}
\\[.1in]&\ll&\displaystyle
Y^{-\frac{3g}{2} + \frac{g(3g-1)}{2}\delta}s_g^{3g}\,\RR^{3g},
\end{array}
\end{equation*}
implying the desired lower bound on $s_g$, where in the last inequality we used the already-established lower bounds on  $s_{g+1}$ and $s_{g+2}$.

The desired lower bounds for $s_g$, $s_{g+1}$, $s_{g+2}$ then follow by combining the upper bound on \smash{$s_g^gs_{g+1}^{g+1}s_{g+2}^{g+2}$} in \eqref{eq:sgsggsggg} and the individual lower bounds on  $s_g$, $s_{g+1}$, $s_{g+2}$ in \eqref{eq:desiredlower}. The desired upper bound on $\RR$ follows by  comparing the upper bound on $s_g$ and the trivial lower bound $s_g\gg 1$. 
$\Box$

\bigskip

\noindent\textbf{Proof of Lemma \ref{lem:crazyg}}: Suppose $\cK_1\subset\cZ$ and $s\in T_\cZ(L,Y)$ satisfies \eqref{eq:siconditions}. 
Then 
\begin{equation}\label{eq:tjboundddd}
t_j^{-1} \ll\begin{cases} \displaystyle Y^{-(g+2)+5g^2\delta}\prod_{i=n-j+1}^n s_i^{n+1}&\mbox{ for }j=1,\ldots,g,\\
 Y^{-1/2+ 20g^2\delta}&\mbox{ for }j=g+1,g+2,\\
  \displaystyle Y^{g+1+23g^2\delta}\prod_{i=j}^n s_i^{-(n+1)}&\mbox{ for }j = g+3,\ldots,n+1,
 \end{cases}
\end{equation}
where the upper bound on $\RR$ also gives $t_j^{-1} \ll Y^{-1/2+20g^2\delta}$ for $j = 1,\ldots,g$.
For $i,j\leq g+2$, we have \smash{$Yt_i^{-1}t_j^{-1}\ll Y^{40g^2\delta}$}. For $j\geq g+3$ and $i\leq n-j+1$, we have \smash{$Yt_i^{-1}t_j^{-1}\ll Y^{28g^2\delta}$}. Using \eqref{eq:tjboundddd} for the rest of the coordinates gives
$$\#\bigl((s(Y\D)\times s(Y\D))\cap W(\Z)\bigr) \ll Y^{(n+1)(g+1)(g+4) + 972g^4\delta}\left( \prod_{i=1}^{g} s_{g+2+i}^{-2i(i+3)(n+1)}\right).$$
The Haar measure satisfies the following bound:
$$\delta(s)=\prod_{k=1}^{n}s_k^{-(n+1)k(n+1-k)} \ll Y^{-(g^2 + 3g + 1)(n+1) + 55g^4\delta}\left(\prod_{i=1}^{g}  s_{g+2+i}^{2i(i+2)(n+1)}\right).$$
Hence
\begin{equation}\label{eq:sYDgsi}
\#\bigl((s(Y\D)\times s(Y\D))\cap W(\Z)\bigr)\,\delta(s) \ll Y^{(n+1)^2 + 1027g^4\delta} \left(\prod_{i=1}^{g}s_{g+2+i}^{-2i(n+1)}\right).
\end{equation}
Suppose now $\#\bigl((s(Y\D)\times s(Y\D))\cap W(\Z)\bigr)\,\delta(s)\gg X^{n+1-\delta}$. Then, for any $i = g+3,\ldots,n$,
$$s_i \ll Y^{(1027g^4 + 2g+3)\delta/(2(i-g-2)(2g+3))} \ll Y^{(1032g^4/(4g))\delta} = Y^{258g^3\delta},$$
as desired. $\Box$

\bigskip

\noindent\textbf{Proof of Lemma \ref{lem:crazyq}}: 
Suppose $M>X^\eta$ where $\eta>0$ is some fixed constant. Suppose $\delta < \max(\eta,1)/1355g^6$. Suppose $\cK_1\subset\cZ$ and $s\in T_\cZ(L,Y)$ satisfies \eqref{eq:siconditions} and \eqref{eq:s_iwtf}. 
We now impose the conditions $\det(B) = 0$ and $|q|(A,B)>M$ for any $(A,B)\in\LLM $ to obtain a further saving for $\#\bigl((s(Y\D)\times s(Y\D))\cap \LLM \bigr)\,\delta(s)$.

The bound \eqref{eq:s_iwtf} on $s_{g+3},\ldots,s_n$ gives $\RR\ll Y^{258g^4\delta}.$ Hence
\begin{equation}\label{eq:tjboundddd2}
t_j^{-1} \ll\begin{cases} \displaystyle Y^{-(g+2)+1295g^5\delta},&\mbox{ for }j=1,\ldots,g,\\
 Y^{-1/2+ 20g^2\delta},&\mbox{ for }j=g+1,g+2,\\
  \displaystyle Y^{g+1+23g^2\delta},&\mbox{ for }j = g+3,\ldots,n+1,
 \end{cases}
\end{equation}
thus improving \eqref{eq:tjboundddd}. 
In this case, 
$$
Yt_g^{-1}t_{g+2}^{-1}\ll Y^{-(n+1)/2+1315g^5\delta},\qquad
Yt_g^{-1}t_{n+1}^{-1} \ll Y^{1318g^5\delta},
\vspace{-.05in}$$
$$Yt_{g+2}^{-1}t_{n+1}^{-1}\ll Y^{(n+1)/2 + 43g^2\delta},\qquad Yt_{n+1}^{-2}\ll
Y^{n+1+46g^2\delta}.$$

 \pagebreak 
\noindent
Since $\delta < 1/(1315g^4)$, we may assume that every $(A,B)\in
(s(Y\D)\times s(Y\D))\cap W(\Z)$ satisfies:
\begin{itemize}
\item[{\rm (a)}] The top left $g\times (g+2)$-blocks of $A$ and
$B$ are $0$.
\item[{\rm (b)}] The entries of the top right $g\times (g+1)$
  blocks of $A$ and $B$ are $O(Y^{1315g^5\delta})$.
\item[{\rm (c)}] The entries $a_{g+1,g+1}$, $a_{g+1,g+2}$,
  $a_{g+2,g+2}$, $b_{g+1,g+1}$, $b_{g+1,g+2}$, and $b_{g+2,g+2}$ are
   $O(Y^{40g^2\delta})$.
\item[{\rm (d)}] The entries $a_{g+1,j}$, $a_{g+2,j}$,
  $b_{g+1,j}$ and $b_{g+2,j}$ are $O(Y^{(n+1)/2+43g^2\delta})$
  for $g+3\leq j\leq n+1$.
\item[{\rm (e)}] The entries $a_{ij}$ and $b_{ij}$ are  $O(Y^{n+1+46g^2\delta})$ for $g+3\leq i,j\leq n+1$.
\end{itemize}

Suppose now that  $(A,B)$ is an element of $(s(Y\D)\times s(Y\D))\cap\LLM $. Then $f_{A,B} = xg(x,y)$, where $g(x,1)$ is a degree $n$ polynomial with
Galois group $S_n$. 
\begin{lemma}\label{lem:nond}
Let $(A,B)$ be as above. If $b_{g+1,g+1}=b_{g+1,g+2}=b_{g+2,g+2}=0$, then $$|q|(A,B)\ll X^{1355g^6\delta}.$$
\end{lemma}

\begin{proof}
Let $v = x_0e_{g+1} + y_0e_{g+2}$ where $x_0,y_0\in K$ are not both $0$, where $K$ is some quadratic extension of $\Q$, and where $(x_0,y_0)$ is a solution to \begin{equation}\label{eq:qAmid}
a_{g+1,g+1}x^2+a_{g+1,g+2}xy+a_{g+2,g+2}y^2 = 0.
\end{equation}
We claim first that $v$ can be chosen in $\Q^{n+1}$. Indeed, $\text{Span}_K\{e_1,\ldots,e_g,v\}$
is a $(g+1)$-dimensional subspace isotropic with respect to $A$ and
$B$. Since $(A,B)$ is distinguished over $\Q$, the set of $(g+1)$-dimensional common isotropic subspaces defined
over any number field $L$ is in bijection with $J[2](L)$, where $J$ is
the Jacobian of the hyperelliptic curve $y^2 = xg(x,1)$ (which has a
rational Weierstrass point at infinity), and $J[2](L)$ is in bijection with the factorizations of $xg(x,1)$ over~$L$.  Since $g(x,1)$ has
Galois group $S_n$, it does not admit any factorization over any
quadratic extension of $\Q$. Therefore,  $J[2](K) = J[2](\Q)$ with the only nontrivial element corresponding to the factorization $x\cdot g(x,1)$. In other words, $\text{Span}_K\{e_1,\ldots,e_g,v\}$ is defined over $\Q$ and we may choose $v$ to be in~$\Q^{n+1}$. In particular, the discriminant $a_{g+1,g+2}^2 - 4a_{g+1,g+1}a_{g+2,g+2}\in\Z$ is a square.

If $a_{g+1,g+1}\neq 0$, let
$$x_1 = -a_{g+1,g+2} + \sqrt{a_{g+1,g+2}^2 - 4a_{g+1,g+1}a_{g+2,g+2}},\qquad y_1 = 2a_{g+1,g+1}.$$ If $a_{g+1,g+1} = 0$, let $x_1 = 1,\,y_1 = 0$. Then
$x_1,y_1$ are integers $\ll Y^{40g^2\delta}$, not both zero, and are solutions to \eqref{eq:qAmid}. Let $x_0 = x_1/\gcd(x_1,y_1)$ and $y_0 = y_1/\gcd(x_1,y_1)$. There then exist integers $x_2,y_2\ll Y^{40g^2\delta}$ such that
$$\{e_1,\ldots,e_g,x_0e_{g+1}+y_0e_{g+2}, x_2e_{g+1}+y_2e_{g+2},e_{g+3},\ldots,e_n\}$$ forms an integral basis for $\Z^{n+1}$ such that the first $g+1$ vectors generate a primitive lattice isotropic with respect to $A$ and $B$, and the first $g+2$ vectors generate a primitive lattice isotropic with respect to $B$. That is, we compute the $|q|$-invariant of $(A,B)$ using this basis. When so expressed, the top
right $(g+1)\times (g+2)$ blocks of the Gram matrices of $A$ and $B$ have the 
form
\begin{equation*}
  A^\top = \begin{pmatrix}0 & \flat & \cdots & \flat\\
    \vdots &\vdots &\ddots &\vdots\\ 0 & \flat & \cdots & \flat\\
    \flat & * & \cdots & *  \end{pmatrix},\qquad B^\top = \begin{pmatrix}
    0 & \flat & \cdots & \flat\\ \vdots &\vdots &\ddots &\vdots\\
    0 & \flat & \cdots & \flat\\ 0 & * & \cdots & *  \end{pmatrix},
\end{equation*}
where entries labeled ``$0$'' are $0$, entries labeled 
``$\flat$'' are $O(Y^{1355g^5\delta})$, and entries labeled 
``$*$'' are $O(Y^{(n+1)/2 + 83g^2\delta})$. Let $M_1$ denote the
$(g+2)\times(g+2)$ matrix whose $i$th row consists of the coefficients
of $\det(A_ix - B_iy)$, where $A_i$ and $B_i$ are the $(g+1)\times(g+1)$
matrices formed by removing the $i$-th columns from $A^\top$ and
$B^\top$, respectively. Then $M_1$ is of the form
\begin{equation*}M_1=\begin{pmatrix}
    * & \cdots & * & *\\ \sharp & \cdots & \sharp & 0\\ \vdots &
    \ddots &\vdots & \vdots \\ \sharp & \cdots & \sharp & 0\end{pmatrix},
\end{equation*}
where entries labeled ``$0$'' are $0$, entries labeled 
``$\sharp$'' are $O(Y^{2710g^6\delta})$, and entries labeled
``$*$'' are $O(Y^{(n+1)/2+1438g^6\delta})$, where the top right coefficient $m'$ of $M_1$ is the determinant of the top right $(g+1)\times(g+1)$ block $B'$ of $B^\top$, up to sign. Thus
\begin{equation*}
|q|(A,B)=\frac{|Q|(A,B)}{|\det(B')|}=\frac{|\det(M_1)|}{|m'|}=|\det(M_1')|,
\end{equation*}
where $M_1'$ is the bottom left $(g+1)\times(g+1)$ block of
$M_1$. Since the coefficients of $M_1'$ are $\ll
Y^{2710g^6\delta}$, it follows that $|q|(A,B)\ll X^{2710g^6(g+1)\delta/(n+1)} \ll X^{1355g^6\delta}$.
\end{proof}

We now return to the proof of Lemma \ref{lem:crazyq}. For any $(A,B)\in (s(Y\D)\times s(Y\D))\cap\LLM $, since $|q|(A,B) > M > X^\eta$, we may assume that $b_{g+1,g+1}$, $b_{g+1,g+2}$, and $b_{g+2,g+2}$ are not all $0$ since $\delta < \eta/(1355g^6)$.

We now fix $b_{ij}$ for $1\leq i\leq g,$ $g+3\leq j\leq n+1$, and $i = g+1,g+2$, $j=g+1,g+2$. We consider the number of pairs $(A,B)\in (s(Y\D)\times s(Y\D))\cap\LLM $ with these prescribed coefficients by viewing $\det(B)$ as a polynomial $F$ in $b_{ij}$ for $g+1\leq i\leq n+1$ and $g+3\leq j\leq n+1$. Note that all of these remaining coefficients have range at least $$Y^{(n+1)/2 + 43g^2\delta}\prod_{i=g+3}^n s_i^{-(n+1)}.$$ Hence, to complete the proof of Lemma \ref{lem:crazyq}, it remains to prove that $F$ is a nonzero polynomial, for then we would have, using \eqref{eq:sYDgsi}, that
\begin{eqnarray*}
\#\bigl((s(Y\D)\times s(Y\D))\cap \LLM \bigr)\,\delta(s) &\ll& Y^{(n+1)^2 + 1027g^4\delta - (n+1)/2 - 43g^2\delta}\, \prod_{i=1}^{g}s_{g+2+i}^{-(2i-1)(n+1)}\\
&\ll& X^{n+1+514g^3\delta - 1/2}.
\end{eqnarray*}

We may assume that the top right $g\times (g+1)$ block of $B$ has full rank, for  otherwise the kernel of $B$ would be isotropic with respect to $A$ forcing  $\Delta(A,B) = 0$ by Lemma \ref{lem:disc}. Hence we may also assume that the top right $g\times (g+1)$ block of $B$ equals $(I_g\,\, 0)$, where $I_g$ denotes the $g\times g$ identity matrix. Then 
$$\det(B) = \det\begin{pmatrix} b_{g+1,g+1}& b_{g+1,g+2} & b_{g+1,n+1}\\ b_{g+2,g+1} & b_{g+2,g+2} & b_{g+2,n+2} \\ b_{n+1,g+1} & b_{n+1,g+2} & b_{n+1,n+1}\end{pmatrix}.$$
Since $$\begin{pmatrix} b_{g+1,g+1}& b_{g+1,g+2}\\b_{g+2,g+1} & b_{g+2,g+2}\end{pmatrix}\neq 0,$$
we see that $\det(B)$ is a nonzero polynomial in $b_{g+1,n+1}$, $b_{g+2,n+1}$, and $b_{n+1,n+1}$. $\Box$

\subsection{Bounding the number of distinguished elements in the deep cusp}\label{sec:evendistcup}

In this subsection, we bound the number of elements with large
$q$-invariant that lie in the deep~cusp.
\begin{theorem}\label{thm:distcusp}
We have\,
$\displaystyle
\cI_X^\dcusp(\LLM )=O\Bigl(\frac{X^{n+1+\frac12\kappa}}{M}\log^{2n}X\Bigr).
$
\end{theorem}

Recall from \eqref{eq:evenMCDC} that
\begin{equation*}
\cI_X^\dcusp(\LLM )\ll\displaystyle
\sum_L\sum_{\cZ:a_{g+1,g+1}\in\cZ}N(\LLM ,L,\cZ,X);
\end{equation*}
here the first sum is over $r$-tuples $L=(L_1,\ldots,L_r)$ with
$L_1\leq L_2\leq \cdots\leq L_n$ that partition the region
$\{(\mu_1,\ldots,\mu_r)\in
[Y^{-\Theta_1},Y^{\Theta_2}]^r:\mu_1\leq\ldots\leq\mu_r\}$ into dyadic
ranges, and the second sum is over saturated subsets $\cZ$ of
$\cK\cup\cM$, where
\begin{equation*}
N(\LLM ,L,\cZ,X)=\int_{T_\cZ(L,Y)}\#\bigl\{
(A,B)\in (s(Y\D)\times s(Y\D))\cap\LLM :B\in\cS(L,s)\bigr\}
\,\delta(s)d^\times s.
\end{equation*}
The set $\cS(L,s)$ is the union over $\Lambda\in
\Sigma(L,s)$ of $S(\Lambda)$, where $S(\Lambda)$ denotes the lattice
of integral symmetric matrices whose row space is contained in
$\Lambda\otimes\R$, and $\Sigma(L,s)$ denotes the set of primitive
lattices $\Lambda\in\Z^{n+1}$ of rank $n$ such that the successive
minima $\mu_1,\ldots,\mu_n$ of $s^{-1}(\Lambda)$ satisfy
$L_i\leq\mu_i<2L_i$ for each $i\in\{1,\ldots,n\}$. Finally recall from \S\ref{sec:setupsplit} and Proposition \ref{prop:lilj} that
$$w(\ell_{g+1,g+1})^2=L_{g+1}^2t_{g+1}^{-2}\leq c_1Yt_{g+1}^{-2}=c_1Yw(a_{g+1,g+1})<c_1c_{g+1,g+1}<c_{g+1}'^{\,\,2}$$
for every $s\in T_\cZ(Y,L)$. Hence we may assume that
$\ell_{g+1,g+1}\in\cZ$. 

The deep cusp contains $\asymp X^{n+1}$ elements, and we
obtain a saving because the elements we are counting have
$q$-invariant greater than $M$. To make use of this condition, we require 
an upper bound on the size of the $|q|$-invariant of elements in
$(s(Y\D)\times s(Y\D)) \cap \LL1 $. To accomplish this, we
have the following preliminary result.

\begin{lemma}\label{lem:usingkappa}
Let $(A,B)\in (Y\D\times Y\D)\cap W_{0}(\R)$ be such that
$\Delta(A,B)> X^{2n-2-\kappa}$. Denote the top right $(g+1)\times
(g+2)$ block of $B$ by $B^\top$. Then 
\begin{equation*}
\det(B^\top (B^\top)^t)\gg Y^{2(g+1)-(n+1)\kappa}.
\end{equation*}
\end{lemma}
\begin{proof}
Let $(A',B')=Y^{-1}(A,B)\in (\D\times\D)\cap
W_{0}(\R)$. Then it suffices to prove that $$ \det(B'^{\,\top} (B'^{\,\top})^t)\gg Y^{-(n+1)\kappa}.$$ Since
$|\Delta(A,B)|>X^{2n-2-\kappa}$, we have
$|\Delta(A',B')|>X^{-\kappa}$.  By Proposition \ref{propBsubdiv},
there is a polynomial $P\in\Z[W_{0}]$ such that $$\Delta(A'',B'')
= P(A'',B'')\det(B''^{\,\top} (B''^{\,\top})^t)$$ for any $(A'',B'')\in
W_{0}(\R)$. Since $(A',B')\in\D\times\D$, which is an
absolutely bounded region, we have $|P(A',B')|\ll 1$. Hence $$\det(B'^{\,\top} (B'^{\,\top})^t) = \frac{\Delta(A',B')}{P(A',B')}\gg
X^{-\kappa}$$ as desired.
\end{proof}

Next, we have the following upper bound on the $|q|$-invariant.

\begin{proposition}\label{prop:qsize}
Let $\cZ\subset\cZ_1$ be a saturated set containing $a_{g+1,g+1}$ and
$\ell_{g+1,g+1}$. Let $L=(L_1,\ldots,L_n)$ be a sequence of
nondecreasing positive real numbers. Then for any $s\in T_\cZ(L,Y)$
and $(A,B)\in (s(Y\D)\times s(Y\D))\cap \LL1 $, we have 
\begin{equation}\label{eq:qsize}
|q|(A,B) \ll Y^{(g+1)^2+\frac{n+1}{2}\kappa}\prod_{i=1}^{g+1} L_i.
\end{equation}
\end{proposition}

\begin{proof}
Suppose $(A,B)\in (s(Y\D)\times s(Y\D))\cap \LL1 $. Since
$a_{g+1,g+1}\in\cZ$, we have $(A,B)\in W_{0}(\Z)$. By Lemma
\ref{lem:disc}, $\ker(B)$ is $1$-dimensional and does not lie inside
$\text{Span}\{e_1,\ldots,e_{g+1}\}$ as this $(g+1)$-plane is isotropic
with respect to $A$. Let
$w_1\in\text{Span}_\Z\{e_{g+2},\ldots,e_{n+1}\}$ be a primitive vector
so that $\{e_1,\ldots,e_{g+1},w_1\}$ forms a basis for the primitive
lattice in $\text{Span}_\R\{e_1,\ldots,e_{g+1}\} + \ker(B).$ Complete
$w_1$ to an integral basis $\{w_1,\ldots,w_{g+2}\}$ for
$\text{Span}_\Z\{e_{g+2},\ldots,e_{n+1}\}$. We can now use the integral~basis
\begin{equation}\label{eq:newbasisforq}
  \{e_1,\ldots,e_{g+1},w_1,\ldots,w_{g+2}\}
\end{equation}
of $\Z^{n+1}$ to compute the $|q|$-invariant of $(A,B)$, as the first
$g+1$ vectors generate a primitive lattice isotropic with respect to
$A$ and $B$, and the first $g+2$ vectors generate a primitive lattice
isotropic with respect to $B$. Note also that with respect to the
standard inner product on $\R^{n+1}$, since
$w_1\in\text{Span}_\R\{e_1,\ldots,e_{g+1}\} + \ker(B)$, we have
\begin{equation}\label{eq:w1ortho}
w_1 \perp \bigl(\text{Span}_\R\{e_{g+2},\ldots,e_{n+1}\}\cap C(B)\bigr)
\end{equation}
where $C(B)$ denotes the column space of $B$.

Let $A'$ and $B'$ be the Gram matrices of the quadratic
forms defined by $A$ and $B$ with respect to this new basis
\eqref{eq:newbasisforq}. Since the first $g+1$ vectors of this basis are
part of the standard basis, we see that $(A,B)$ and $(A',B')$ are
$G_0(\Z)$-equivalent, where $G_0$ is defined in \S\ref{sec:2.1}. Hence
\begin{equation*}
  |Q|(A',B') = |Q|(A,B) \ll Y^{(g+1)(g+2)}\prod_{k=1}^{g+1} t_k^{-1}.
\end{equation*}
Let $B''$ denote the top right $(g+1)\times(g+1)$ block of $B'$. Then,
by the definition of $q$, we have
\begin{equation}\label{eq:qupbd1}
|q|(A,B) = |q|(A',B') = 
\frac{|Q|(A',B')}{|\det(B'')|} \ll
\frac{1}{|\det(B'')|}Y^{(g+1)(g+2)}\prod_{k=1}^{g+1} t_k^{-1}.
\end{equation}

We now work towards proving a lower bound on $|\det(B'')|$.  Let $p_1$
(resp., $p_2$) denote the projection of $\R^{n+1}$ onto the first
$g+1$ coefficients (resp., the last $g+2$ coefficients).  Let
$B^{\top}$ (resp., $B'^\top$) denote the top right $(g+1)\times (g+2)$
block of $B$ (resp., $B'$). Then by \eqref{eq:w1ortho}, we have
$B^\top p_2(w_1) = 0$. Consider the following two $(g+2)\times(g+2)$
matrices in block form:
\begin{equation*}
  B^* = \begin{pmatrix}B^\top\\ p_2(w_1)^t\end{pmatrix},\qquad
    \gamma = \begin{pmatrix}p_2(w_1)&\cdots&p_2(w_{g+2})\end{pmatrix}.
\end{equation*}
Then
\begin{equation*}
  B^*\gamma = \begin{pmatrix}0&B''\\ |w_1|^2&*\end{pmatrix}.
\end{equation*}
Let $\Lambda_2$ denote the rank $g+1$ lattice in $\Z^{g+2}$ spanned by
the rows of $B^{\top}$. Then $|\det(B^*)| = d(\Lambda_2)|w_1|.$ Since
$\{p_2(w_1),\ldots,p_2(w_{g+2})\}$ is an integral basis for
$\Z^{g+2}$, we have $\det\gamma = \pm1$ and so
\begin{equation*}
  |\det(B'')| = \frac{|\det(B^*)\det\gamma|}{|w_1|^2} =
  \frac{d(\Lambda_2)|w_1|\cdot 1}{|w_1|^2} =
  \frac{d(\Lambda_2)}{|w_1|}.
\end{equation*}

We now use the fact that $B\in\cS(L,s)$. This means that the
row span of $B$ lies in an  $n$-dimensional
primitive lattice $\Lambda\subset\Z^{n+1}$ with basis of the
form $\{s\ell_1,\ldots,s\ell_n\}$ where $L_i\leq |\ell_i|<2L_i$ and
$\{\ell_1,\ldots,\ell_n\}$ are reduced. By assumption, 
$\ell_{g+1,g+1}\in\cZ$, and hence $\ell_{i,j}\in\cZ$ for all $i\leq
g+1$ and $j\leq g+1$. Thus the first $g+1$ coefficients
of $s\ell_1,\ldots,s\ell_{g+1}$ are all $0$, and  
$\{s\ell_1,\ldots,s\ell_{g+1}\}$ forms an integral basis of a
primitive lattice $\Lambda_1$ of rank $g+1$ in
$\text{Span}_\R\{e_{g+2},\ldots,e_{n+1}\}$. By \eqref{eq:w1ortho}, $w_1$ is a primitive vector in
$\text{Span}_\R\{e_{g+2},\ldots,e_{n+1}\}$ orthogonal to~$\Lambda_1$. Hence $$|w_1| = d(\Lambda_1).$$ By
\eqref{eq:w1ortho}, we have
$\text{Span}_\R\{e_{g+2},\ldots,e_{n+1}\}\cap
C(B)\neq\text{Span}_\R\{e_{g+2},\ldots,e_{n+1}\}$, and
so
\begin{equation*}
\text{Span}_\R\{e_{g+2},\ldots,e_{n+1}\}\cap
C(B)=\Lambda_1\otimes\R.
\end{equation*}
In particular, since $\Lambda_1$ is primitive, the first $g+1$ columns
of $B$ belong to $\Lambda_1$. That is, there is a
$(g+1)\times(g+1)$ matrix $C$ (with integer coefficients) such that
\begin{equation*}
  B^\top = C\begin{pmatrix}p_2(s\ell_1)^t\\
  \vdots\\ p_2(s\ell_{g+1})\end{pmatrix}
\end{equation*}
and so 
\begin{equation}\label{eq:detCdetB}
  |\det(C)| = \frac{d(\Lambda_2)}{d(p_2(\Lambda_1))} =
  \frac{d(\Lambda_2)}{|w_1|} = |\det(B'')|.
\end{equation}

To obtain a lower bound on $|\det(C)|$, we write
\begin{equation*}
  s = \begin{pmatrix} A_1&0\\0&A_2\end{pmatrix}
    \qquad \mbox{with} \qquad A_1 =
    \begin{pmatrix}t_1^{-1}& &\\ &\ddots&\\&&t_{g+1}^{-1}\end{pmatrix},
    \quad A_2 = \begin{pmatrix}
     t_{g+2}^{-1}& &\\ &\ddots&\\&&t_{n+1}^{-1}\end{pmatrix}.
\end{equation*}
Let $M^\top$ denote the $(g+1)\times (g+2)$ matrix with rows
$p_2(\ell_1)^t,\ldots,p_2(\ell_{g+1})^t$. Then  $$CM^\top
A_2=B^{\top}.$$ Consider the pair $(A_0,B_0):=s^{-1}(A,B)\in
(Y\D\times Y\D)\cap W_{0,n+1}(\R)$
satisfying 
$$|\Delta(A_0,B_0)| =
|\Delta(A,B)|>X^{2n-2-\kappa}$$ since $(A,B)\in\LL1 $. The
top right $(g+1)\times(g+2)$ block $B_{0}^\top$ of $B_0$
satisfies $$A_1B_{0}^\top A_2=B^\top,$$ and so  $$CM^\top=A_1B_{0}^\top.$$ The rows of $M^\top$ form a reduced
basis for a lattice $\Lambda_3\subset \Z^{g+2}$ with $L_i\leq
|p_2(\ell_i)|<2L_i$. Thus 
\begin{equation}\label{eq:detClower}
  \det(B'')^2=\det(C)^2=
  \frac{\det(A_1B_{0}^\top (B_{0}^\top)^tA_1)}{\det(M^\top (M^\top)^t)}\gg 
\frac{t_1^{-2}\cdots t_{g+1}^{-2}}{L_1^2\cdots L_{g+1}^2} \det(B_{0}^\top (B_{0}^\top)^t).
\end{equation}
By Equations \eqref{eq:qupbd1} and \eqref{eq:qupbd1},
\begin{equation*}
  |q|(A,B)\ll Y^{(g+1)(g+2)}
  \frac{L_1\cdots L_g}{\sqrt{\det(B_{0}^\top (B_{0}^\top)^t)}}.
\end{equation*}
The result now follows from Lemma \ref{lem:usingkappa}.
\end{proof}

\smallskip

\noindent{\bf Proof of Theorem \ref{thm:distcusp}:} We write
\begin{equation}\label{eq:thmdistcusp1}
  \cI_X^\dcusp(\LLM )\ll \sum_{L}
  \sum_{\substack{\cZ\\a_{g+1,g+1}\in\cZ\\\ell_{g+1,g+1}\in\cZ}}
  N(\LLM ,L,\cZ,X),
\end{equation}
and obtain upper bounds on $N(\LLM ,L,\cZ,X)$ for each
$\cZ\subset\cZ_1$ with $a_{g+1,g+1},\ell_{g+1,g+1}\in\cZ$. Fix such a
set $\cZ$ with $N(\LLM ,L,\cZ,X)>0$ and an element $s\in
T_\cZ(L,Y)$. Then
\begin{equation*}
  (s(Y\D)\times s(Y\D))\cap \LLM 
  \subset (s(Y\D)\cap S(\Z)) \times (s(Y\D)\cap \cS(L,s).
\end{equation*}

We begin by bounding the number of elements in $\#(s(Y\D)\cap
S(\Z))$. Let $\cK_\dist := \{a_{ij}\mid 1\leq i\leq j\leq g+1\}$. By
assumption, $\cK_\dist$ is a subset of $\cZ\cap\cK$. Define $\pi_\cK:
\cZ_1\cap\cK\rightarrow\cK\backslash\cZ_1$ by $$\pi_\cK(a_{ij}) :=
a_{n+1-j,j}.$$ This agrees with the $\pi_k$
as defined in \S\ref{sec:6.3.1} when  restricted to $\cK$. For any
$\alpha\in \cZ_1\cap\cK$, we have $Yw(\pi_\cK(\alpha))\gg1$ and
$w(\pi_\cK(\alpha))\gg w(\alpha)$. For any $a_{ij}\in
(\cZ_1\cap\cK)\backslash \cK_\dist$, we have $i< g+1<j$. Thus 
\begin{equation*}
  \prod_{\alpha\in (\cZ\cap\cK)\backslash \cK_\dist}
  \frac{w(\pi_\cK(\alpha))}{w(\alpha)} \ll
  \prod_{\alpha\in (\cZ_1\cap\cK)\backslash \cK_\dist}
  \frac{w(\pi_\cK(\alpha))}{w(\alpha)}=\prod_{\substack{1\leq i < g+1 < j \\ i + j \leq n}}\frac{t_i}{t_{n+1-j}}=\prod_{1\leq i\leq j\leq g+1}\frac{t_i}{t_j}.
\end{equation*}
Therefore, 
\begin{eqnarray}
\nonumber\#\bigl(s(Y\D)\cap S(\Z)\bigr) &\ll &  Y^{(n+1)(n+2)/2-\#(\cZ\cap\cK)}\prod_{\alpha\in \cZ\cap\cK}\frac{1}{w(\alpha)}\\
\nonumber&\ll&  Y^{(n+1)(n+2)/2-\#(\cZ\cap\cK)}\Big(\prod_{\alpha\in \cK_\dist}\frac{1}{w(\alpha)}\Big)\Big(\prod_{\alpha\in (\cZ\cap\cK)\backslash \cK_\dist} \frac{Yw(\pi_q(\alpha))}{w(\alpha)}\Big)\\
\nonumber&\ll&  \frac{Y^{(n+1)(n+2)/2}}{Y^{\#\cK_\dist}}\Big(\prod_{1\leq i\leq j \leq g+1} t_it_j\Big)\Big(\prod_{1\leq i\leq j\leq g+1}\frac{t_i}{t_j}\Big)\\
\label{eq:wow}&=&  \frac{Y^{(n+1)(n+2)/2}}{Y^{(g+1)(g+2)/2}}(t_1\cdots t_{g+1})^{g+2} \Big(\prod_{1\leq i\leq j\leq g+1}\frac{t_i}{t_j}\Big).
\end{eqnarray}

We now obtain an upper bound on  $\#(s(Y\D)\cap\cS(L,s))$. Recall that 
\begin{equation}\label{eq:distBcount1}
\#\bigl(s(Y\D)\cap\cS(L,s)\bigr)=\sum_{\Lambda\in\Sigma(L,s)}\#\big(Y\D \cap s^{-1}S(\Lambda)\big).
\end{equation}
Let $\Lambda\in\Sigma(L,s)$ be a lattice such that $s^{-1}(\Lambda)$ has reduced basis $\{\ell_1,\ldots,\ell_n\}$ with $L_i\leq |\ell_i| < 2L_i$ for each $i=1,\ldots,n$. Suppose there exists $(A,B)\in (s(Y\D)\times s(Y\D))\cap \LLM $ with $B\in s(Y\D)\cap S(\Lambda)$. By Proposition \ref{prop:qsize}, 
\begin{equation}\label{eq:qMweight}
M\ll Y^{(g+1)^2+\frac{n+1}{2}\kappa}\prod_{i=1}^{g+1} L_i.
\end{equation} 
Recall also from the proof of Proposition \ref{prop:qsize} that
$$\text{Span}_\R\{e_{g+2},\ldots,e_{n+1}\}\cap C(B)=\text{Span}_\R\{s\ell_1,\ldots,s\ell_{g+1}\}.$$
Hence 
$$\text{Span}_\R\{e_{g+2},\ldots,e_{n+1}\}\cap \text{Span}_\R\{s\ell_{g+2},\ldots,s\ell_{n}\}=\{0\}.$$
It follows that the set $\{p_1(s\ell_{g+2}),\ldots,p_1(s\ell_n)\}$, and thus the set $\{p_1(\ell_{g+2}),\ldots,p_1(\ell_n)\}$, are both linearly independent. There then exist vectors $v_{g+2},\ldots,v_n\in\text{Span}_\R\{e_1,\ldots,e_{g+1}\}$ such that $$\begin{pmatrix}v_{g+2}&\cdots&v_n\end{pmatrix}^t\begin{pmatrix}\ell_{g+2}&\cdots&\ell_n\end{pmatrix} = I_{g+1}$$
is the identity matrix.
Let $B'\in s(Y\D)\cap S(\Lambda)$ be any element and write $$s^{-1}B' = \sum_{1\leq i\leq j\leq n} \beta_{ij}\ell_i\ast\ell_j,$$
where $\ell_i\ast \ell_j$ is as defined in (\ref{stardef}). 
Then for $g+2\leq i\leq j\leq n$, since $v_i,v_j\perp \ell_1,\ldots,\ell_{g+2}$, we have $$v_i^t (s^{-1}B') v_j = \begin{cases} 2\beta_{ij} &\mbox{if }i\neq j,\\
\beta_{ii}&\mbox{if }i=j\end{cases}. $$
Since the top left $(g+1)\times (g+1)$ block of $B'\in s(Y\D)\cap S(\Lambda)$ is $0$, the same is true for $s^{-1}B'$. Hence $\beta_{ij} = 0$ whenever $g+2\leq i\leq j\leq n$. In other words, $$Y\D \cap s^{-1}S(\Lambda)\subset \mbox{Span}_\Z \{\ell_i\ast\ell_j\mid 1\leq i\leq j\leq n\mbox{ and }i\leq g+1\}.$$
By Proposition \ref{prop:sch}, we have
\begin{eqnarray}
\nonumber\#(Y\D \cap s^{-1}S(\Lambda))&\ll&
\displaystyle\prod_{\substack{1\leq i \leq j\leq n\\i\leq g+1\\L_iL_j\ll Y}}\frac{Y}{L_iL_j}\\
\nonumber&\ll&\displaystyle
\Bigl(\prod_{\substack{1\leq i \leq j\leq n\\i\leq g+1}}\frac{Y}{L_iL_j}\Bigr)\prod_{\substack{1\leq i\leq j\leq n\\i\leq g+1\\L_iL_j\gg Y}}\Bigl(\frac{L_iL_j}{Y}\frac{Y}{L_{n+1-j}{L_j}}\Bigr)\\
\nonumber&\ll&\displaystyle
\frac{Y^{n(n+1)/2}}{(L_1\cdots L_n)^{n+1}}\frac{(L_{g+2}\cdots L_n)^{g+2}}{Y^{(g+1)(g+2)/2}}\prod_{\substack{1\leq i\leq j\leq n\\i\leq g+1}}\frac{L_i}{L_{n+1-j}}\\
\label{eq:Bcountdist1}&\ll&\frac{Y^{n(n+1)/2}}{(L_1\cdots L_n)^{n+1}}\frac{(L_{g+2}\cdots L_n)^{g+2}}{Y^{(g+1)(g+2)/2}}\prod_{1\leq i< j\leq g+1}\frac{L_j}{L_i},
\end{eqnarray}
where the second bound follows since $L_{n+1-j}L_j\ll Y$ for all $j$ by Proposition \ref{prop:lilj}; and the last bound follows because the map from $\{(i,j):1\leq i\leq j\leq n\mbox{ and }i\leq g+1\}$ to $\{(k,\ell)\}$ sending $(i,j)$ to $(n+1-j,i)$ is one-to-one with its image contained within the set of pairs $(k,\ell)$ with $k<\ell\leq g+1$, and because the $L_i$'s are nondecreasing. 

To obtain a bound on the size of $\Sigma(L,s)$, we use \eqref{eq:Lambda}:
\begin{equation*}
\#\Sigma(L,s)\ll (L_1L_2\cdots L_n)^{n+1}
\Bigl(\prod_{1\leq i<j\leq n}\frac{L_i}{L_j}\Bigr)
\Bigl(\prod_{\alpha\in\cZ\cap\cM}
\frac{1}{w(\alpha)}\Bigr).
\end{equation*}
Let $\cM_\dist = \{\ell_{i,j}\mid 1\leq i\leq g+1,\,1\leq j\leq g+1\}$. Recall that elements $\ell_{i,j}\in\cZ_1$ satisfy $i+j\leq n+1$. Hence for any $\ell_{ij}\in (\cZ_1\cap\cM)\backslash \cM_\dist$, exactly one of $i$ and $j$ is $\leq g+1$. Define
\begin{equation*}
\begin{array}{rcl}
\pi_\cM&:&(\cZ_1\cap\cM)\backslash \cM_\dist\to\cM\\[.1in]
\pi_\cM(\ell_{i,j}) &=& \displaystyle\left\{
  \begin{array}{rl}
    \ell_{i,n+2-i}&\;{\rm if}\; i\leq g+1;\\
    \ell_{n+1-j,j} &\;{\rm if}\; j \leq g+1.
  \end{array}\right.
\end{array}
\end{equation*}
We claim that the image of $\pi_\cM$ is disjoint from $\cZ$. Indeed, when $i\leq g+1$, we have $\pi_\cM(\ell_{i,j})\not\in\cZ_1$, and when $j\leq g+1$, we have $\pi_\cM(\ell_{i,j})\not\in\cZ$ by Lemma \ref{lem:NUX} and the fact that $a_{g+1,g+1}\in\cZ$. Thus $w(\pi_\cM(\alpha))\gg 1$ and $w(\pi_\cM(\alpha))\gg w(\alpha)$ for every $\alpha\in(\cZ_1\cap\cM)\backslash \cM_\dist$. It follows that
\begin{equation*}
\begin{array}{rcl}
\displaystyle \prod_{\alpha\in\cZ_\cM}\frac{1}{w(\alpha)}
&\ll&\displaystyle
\Bigl(\prod_{\ell\in \cM_\dist}\frac{1}{w(\ell)}
\Bigr)\Bigl(\prod_{(\cZ_1\cap\cM)\backslash\cM_\dist}\frac{w(\pi(\ell))}{w(\ell)}\Bigr)
\\[.2in]&\ll&\displaystyle
\frac{(t_1\cdots t_{g+1})^{g+1}}{(L_1\cdots L_{g+1})^{g+1}}
\Bigl(\prod_{g+2\leq i< j\leq n+1}\frac{t_i}{t_j}\Bigr)
\Bigl(\prod_{g+2\leq i< j\leq n+1}\frac{L_j}{L_i}\Bigr),
\end{array}
\end{equation*}
so that
\begin{equation}\label{eq:lbounddist}
\#\Sigma(L,s) \ll (L_1\cdots L_n)^{n+1}\frac{(t_1\cdots t_{g+1})^{g+1}}{(L_1\cdots L_{g+1})^{g+1}}\Bigl(\prod_{g+2\leq i< j\leq n+1}\frac{t_i}{t_j}\Bigr)\Bigl(\prod_{i=1}^{g+1}\prod_{j=i}^n \frac{L_i}{L_j}\Bigr) .
\end{equation}
Combining \eqref{eq:distBcount1}, \eqref{eq:Bcountdist1}, and \eqref{eq:lbounddist} and the identity 
\begin{equation*}
\Bigl(\prod_{1\leq i< j\leq g+1}\frac{L_j}{L_i}\Bigr)\Bigl(\prod_{i=1}^{g+1}\prod_{j=i}^n \frac{L_i}{L_j}\Bigr) = \frac{(L_1\cdots L_{g+1})^{g+1}}{(L_{g+2}\cdots L_n)^{g+1}}
\end{equation*}
now yields
\begin{eqnarray}
\nonumber\#\bigl(s(Y\D)\cup\cS(L,s)\bigr)
&\ll& \frac{Y^{n(n+1)/2}}{(L_1\cdots L_n)^{n+1}}
\frac{(L_{g+2}\cdots L_n)^{g+2}}{Y^{(g+1)(g+2)/2}}
\Bigl(\prod_{1\leq i< j\leq g+1}\frac{L_j}{L_i}\Bigr)\\ 
\nonumber&&\cdot \,\,(L_1\cdots L_n)^{n+1}
\frac{(t_1\cdots t_{g+1})^{g+1}}{(L_1\cdots L_{g+1})^{g+1}}
\Bigl(\prod_{g+2\leq i< j\leq n+1}\frac{t_i}{t_j}\Bigr)
\Bigl(\prod_{i=1}^{g+1}\prod_{j=i}^n \frac{L_i}{L_j}\Bigr)\\
\nonumber&=& \frac{Y^{n(n+1)/2}}{Y^{(g+1)(g+2)/2}}(t_1\cdots t_{g+1})^{g+1}(L_{g+2}\cdots L_n)
\prod_{g+2\leq i< j\leq n+1}\frac{t_i}{t_j}\\
\nonumber&\ll& \frac{Y^{n(n+1)/2}}{Y^{(g+1)(g+2)/2}}(t_1\cdots t_{g+1})^{g+1}(L_{g+2}\cdots L_n)
\Bigl(\prod_{g+2\leq i< j\leq n+1}\frac{t_i}{t_j}\Bigr)
\Bigl(\prod_{j=g+2}^n\frac{Y}{L_jL_{n+1-j}}\Bigr)\\
\label{eq:Bcounttt}&=&
 \frac{Y^{n(n+1)/2+g+1}}{Y^{(g+1)(g+2)/2}}\frac{(t_1\cdots t_{g+1})^{g+1}}{L_1\cdots L_{g+1}}\prod_{g+2\leq i< j\leq n+1}\frac{t_i}{t_j},
\end{eqnarray}
where the fourth line follows since $L_{n+1-j}L_j\ll Y$ for all $j$ by Proposition \ref{prop:lilj}. Finally, note that 
\begin{equation*}
\begin{array}{rcl}
\displaystyle(t_1\cdots t_{g+1})^{2g+3}
\prod_{1\leq i< j\leq g+1}\frac{t_i}{t_j}
\prod_{g+2\leq i< j\leq n+1}\frac{t_i}{t_j}
&=&\displaystyle\frac{(t_1\cdots t_{g+1})^{g+2}}{(t_{g+2}\cdots t_{n+1})^{g+1}}
\prod_{1\leq i< j\leq g+1}\frac{t_i}{t_j}
\prod_{g+2\leq i< j\leq n+1}\frac{t_i}{t_j}
\\[.2in]&=&\displaystyle
\prod_{1\leq i< j\leq n+1}\frac{t_i}{t_j} = \delta(s)^{-1}.
\end{array}
\end{equation*}
Therefore, combining \eqref{eq:wow}, \eqref{eq:qMweight} and \eqref{eq:Bcounttt} gives
\begin{equation*}
\begin{array}{rcl}
N(\LLM ,L,\cZ,X)&\ll&\displaystyle
\int_{T_\cZ(L,Y)}\#\bigl(s(Y\D)\cap S(\Z)\bigr)\cdot\#\bigl( s(Y\D)\cap\cS(L,s)\bigr)\,
\delta(s)d^\times s
\\[.2in]&\ll&\displaystyle
\frac{Y^{(n+1)^2-(g+1)^2}}{L_1\cdots L_{g+1}}\int_{T_\cZ(L,Y)}d^\times s
\\[.2in]&\ll&\displaystyle
\frac{X^{(n+1)+\frac12\kappa}}{M}\log^n Y.
\end{array}
\end{equation*}
Theorem \ref{thm:distcusp} now follows immediately from \eqref{eq:thmdistcusp1} by summing over the $O(1)$ different possible $\cZ$'s and the
$O(\log^n Y)$ different possible $L$'s. $\Box$

\subsection{Proof of the main uniformity estimates} \label{sec:6.5}
\smallskip 

\noindent{\bf Proof of Theorem \ref{thm:mainestimate}:} 
Case (a) of Theorem~\ref{thm:mainestimate} follows from an application of the quantitative version of the Ekedahl geometric sieve developed in \cite[Theorem
  3.3]{geosieve}. Case (b) follows from \eqref{eq:odddegreeprelim}, \eqref{eq:avg}, and Theorem \ref{thm:distirrodd}. 
  
  We now use the results of this section to prove the most intricate case, namely, Case (c).
  For $m\leq X^{1/2}$, we have the immediate bound
$$\#\{f\in\W_m:H(f)<X\}\ll X^{n+1}/m^2.$$ Using this bound for $m\leq X^{1/2}$, we may assume that $M>X^{1/2}$. We first note that from \eqref{eq:evencontainment}, \eqref{eq:evenbreakup}, and Lemmas \ref{lem:discsmall} and \ref{lem:geoeven}, we have
\begin{equation*}
\#\bigcup_{\substack{m>M\\
m\;\mathrm{ squarefree} }}\{f\in\U_m^\w:H(f)<X\}
\ll \cI_X(\L(\sqrt{M}))+\frac{X^{n+1}}{\sqrt{M}}+X^{n}+X^{n+1-\frac{\kappa}{2n-2}}.
\end{equation*}
Applying Theorems \ref{th:evenmainbody}, \ref{th:evenmaincusp}, and \ref{thm:distcusp}, we obtain
\begin{equation*}
\begin{array}{rcl}
\displaystyle\cI_X(\L(\sqrt{M}))&=&\displaystyle\cI_X^\main(\L(\sqrt{M}))+\cI_X^\scusp(\L(\sqrt{M}))+\cI_X^\dcusp(\L(\sqrt{M}))
\\[.05in]&\ll&\displaystyle
X^{n+1-1/(10n)}+X^{n+1-1/(88n^6)}+\frac{X^{n+1+\frac12\kappa}}{\sqrt{M}}\log^{2n}X.
\end{array}
\end{equation*}
Setting $\kappa = (2n-2)/(88n^6)$ 
yields the desired result. 
$\Box$.

\medskip

Theorem \ref{thm:mainestimate} has the following immediate consequence. For a positive squarefree integer $m$, let $\W_m:=\W_m^\s\cup\W_m^\w$.
\begin{corollary}\label{cor:unifWm}
For a positive integer $N\geq 3$, and positive real numbers $M$ and
$X$, we have
\begin{equation*}
\sum_{\substack{m>M\\m{\rm squarefree}}}\#\{f\in \W_{m}:H(f)<X\}
\ll_\epsilon \frac{X^{n+1+\xi_n+\epsilon}}{M^{\delta_n}}+X^{n+1-{\eta_n}+\epsilon},
\end{equation*}
where $\delta_n=1/2,\,\xi_n = 0,\,\eta_n=1/(5n)$ when $n$ is odd and
$\delta_n=1/3,\,\xi_n = 1/(88n^5),\,\eta_n=1/(88n^6)$ when $n$ is even.
\end{corollary}
\begin{proof}
Suppose $f\in \W_{m}$ for some $m>M$. When $n$ is odd, this implies
that $f$ belongs to either $\W_{m_1}^{(1)}$ or $\W_{m_1}^{(2)}$ for
some $m_1>\sqrt{M}$. When $n$ is even, $f$ belongs either to
$\W_{m_2}^{(1)}$ for $m_2>M^{1/3}$ or to $\W_{m_3}^{(2)}$ for
$m_3>M^{2/3}$. A direct application of Theorem \ref{thm:mainestimate}
now yields the result.
\end{proof}

\section{Proofs of the main results}\label{latticearg}

We begin by proving a more general form of Theorem
\ref{thm:powersave}. Let $N$ be a positive squarefree integer, and for
each $p\mid N$, let $\Sigma_p\subset V_n(\Z/p^2\Z)$ be a nonempty
subset. Denote the collection $(\Sigma_p)_{p\mid N}$ by~$\Sigma$. Let $V_n(\Sigma)$ be the set of all $f\in V_n(\Z)$ such that the
reduction of $f$ modulo $p^2$ lies in $\Sigma_p$ for all~$p\mid
N$. For $p\mid N$, let $\alpha_n(\Sigma,p)$
(resp., $\beta_n(\Sigma,p)$) denote the density of elements $f\in
V_n(\Z)$ such that $p^2\nmid\Delta(f)$ (resp., $R_f$ is maximal at $p$)
and such that the reduction of $f$ modulo $p^2$ lies in
$\Sigma_p$. For $p\nmid N$, simply set
$\alpha_n(\Sigma,p)=\alpha_n(p)$ and
$\beta_n(\Sigma,p)=\beta_n(p)$. Finally, define
\begin{equation*}
  \alpha_n(\Sigma)=\prod_p\alpha_n(\Sigma,p);
  \quad\quad\beta_n(\Sigma)=\prod_p\beta_n(\Sigma,p).
\end{equation*}

We are now ready to carry out our sieve.

\begin{theorem}\label{th:sqfreesieve}
We have,  
\begin{equation*}
  \begin{array}{ccl} \displaystyle \#\{f\in V_n(\Sigma): H(f)<X\mbox{ and
      $\Delta(f)$ squarefree}\}&\!\!=\!\!&
    \alpha_n(\Sigma)(2X)^{n+1} +
    O_\epsilon(\mathcal{E}(X,N,\epsilon)),
    \\[.115in] \displaystyle \#\{f\in V_n(\Sigma): H(f)<X\mbox{ and
      $R_f$ maximal}\}&\!\!=\!\!&
    \beta_n(\Sigma)(2X)^{n+1} + O_\epsilon(\mathcal{E}(X,N,\epsilon))
\end{array}
\end{equation*}
where the error term is given by
$$
  \mathcal{E}(X,N,\epsilon):=
  X^{n+1-\eta_n+\epsilon}+N^2X^{n+3(\eta_n+\xi_n)+\epsilon}+N^{2n+2}X^{(6n+3)(\eta_n+\xi_n)+\epsilon},
$$
where $\eta_n = 1/(5n),\, \xi_n = 0$ when $n$ is odd, and $\eta_n = 1/(88n^6),\,\xi_n = 1/(88n^5)$ when $n$ is even.
\end{theorem}
\begin{proof}
For any
squarefree integer $m$ that is relatively prime to $N$, let
$\W_{m}(\Sigma)$ denote the set of elements $f\in V(\Z)$ such that
$m^2\mid\Delta(f)$, and such that the reduction of
$f$ modulo $p^2$ belongs to $\Sigma_p$ for every $p\mid N$. Note that $\W_m(\Sigma)$ is a union of $$\gamma(\Sigma,N,m) := N^{2n+2}m^{2n+2}\prod_{p\mid m}\Bigl(\frac{\#\Sigma_p}{p^{2n+2}} - \alpha_n(\Sigma,p)\Bigr) = O_\epsilon(N^{2n+2}m^{2n+\epsilon})$$ translates of $m^2N^2V(\Z)$. By inclusion-exclusion and Corollary \ref{cor:unifWm}, we have for any $M>0$,
\begin{eqnarray*}
\nonumber&&\#\{f\in V_n(\Sigma)\colon H(f)<X\mbox{ and
      $\Delta(f)$ squarefree}\}\\[.025in]
\nonumber  &=&\sum_{\substack{(m,N)=1\\m\leq M}}
\mu(m)\#\{f\in \W_m(\Sigma)\colon H(f)<X\} + O_\epsilon\Bigl(\frac{X^{n+1+\xi_n+\epsilon}}{M^{\delta_n}}+X^{n+1-{\eta_n}+\epsilon}\Bigr)\\[-.025in]
\nonumber&=&\sum_{\substack{(m,N)=1\\m\leq M}}
\mu(m)\gamma(\Sigma,N,m)\Bigl(\frac{2X}{N^{2}m^{2}} + O(1)\Bigr)^{n+1}+ O_\epsilon\Bigl(\frac{X^{n+1+\xi_n+\epsilon}}{M^{\delta_n}}+X^{n+1-{\eta_n}+\epsilon}\Bigr)\\[-.025in]
\nonumber&=&\sum_{\substack{(m,N)=1\\m\leq M}}\Biggl(
(2X)^{n+1}\mu(m)\prod_{p\mid m}\Bigl(\frac{\#\Sigma_p}{p^{2n+2}} - \alpha_n(\Sigma,p)\Bigr) + O\Bigl(N^2X^{n}m^\epsilon+N^{2n+2}m^{2n+\epsilon}\Bigr)\Biggr)\\&& \nonumber+\,O_\epsilon\Bigl(\frac{X^{n+1+\xi_n+\epsilon}}{M^{\delta_n}}+X^{n+1-{\eta_n}+\epsilon}\Bigr)\\
\nonumber&=& (2X)^{n+1}\alpha_n(\Sigma) + O\Bigl(\frac{X^{n+1}}{M^{1-\epsilon}} + N^2M^{1+\epsilon}X^n + N^{2n+2}M^{2n+1+\epsilon} + \frac{X^{n+1+\xi_n+\epsilon}}{M^{\delta_n}}+X^{n+1-{\eta_n}+\epsilon}\Bigr).
\end{eqnarray*}
Recalling that $\delta_n = 1/2$ or $1/3$, we may take $M = X^{3\eta_n+3\xi_n}$ to obtain the first claim in Theorem~\ref{th:sqfreesieve}. 
The second claim follows identically.
\end{proof}
\pagebreak 

Taking $N=1$ in Theorem~\ref{th:sqfreesieve} yields Theorem
\ref{thm:powersave}. Theorems \ref{polydisc2} and
\ref{polydiscmax2} are then immediate consequences of Theorem
\ref{thm:powersave}.

Next, we prove lower bounds on the number of $S_n$-fields
having bounded discriminant. Let $f(x,y) = a_0x^n + a_1x^{n-1}y + \cdots + a_ny^n$ be a real binary $n$-ic form with $a_0\neq 0$ and nonzero discriminant. 
Let~$\theta$ be the image of $x$ in $\R[x]/(f(x,1))$, and write $R_f$ for the lattice spanned by 
$$1,\quad\zeta_1=a_0\theta,\quad \zeta_2 = a_0\theta^2 + a_1\theta,\quad\ldots,\quad\zeta_{n-1}=a_0\theta^{n-1} + \cdots + a_{n-1}$$ in $\R[x]/(f(x,1))$. Here we identify $\R[x]/(f(x,1))$ with $\R^n$  via its real and complex embeddings and by identifying $\C=\R\oplus i\R$ with~$\R^2$.

We say that $f(x,y)$ is
\textit{Minkowski-reduced} if the basis
$\{1,\zeta_1,\ldots,\zeta_{n-1}\}$ of $R_f$ is
Minkowski-reduced. We say that $f(x,y)$, or its 
$\SL_2(\Z)$-orbit, is \textit{quasi-reduced} if there exists
$\gamma\in\SL_2(\Z)$ such that $\gamma.f$ is Minkowski-reduced. We add
the prefix ``strongly'' if the relevant lattice has a unique
Minkowski-reduced basis. The relevance of being strongly quasi-reduced
is contained in the following lemma.

\begin{lemma}\label{lem:semireduced}
Let $n\geq 3$ and let $f(x,y)$ and $f^*(x,y)$ be strongly
quasi-reduced integral binary $n$-ic forms. Suppose the corresponding 
rank-$n$ rings $R_f$ and $R_{f^*}$ are isomorphic. Then $f(x,y)$ and
$f^*(x,y)$ are $\SL_2(\Z)$-equivalent.
\end{lemma}

\begin{proof}
It suffices to assume $f(x,y)=a_0x^n + \cdots + a_ny^n$ and $f^*(x,y)
= a_0^*x^n + \cdots + a_n^*y^n$ are strongly Minkowski-reduced with
$R_f\simeq R_{f^*}$. We show $f(x,y) = f^*(x,y)$. Let
$\phi:R_f\rightarrow R_{f^*}$ be a ring isomorphism. By the
uniqueness of Minkowski-reduced bases, $\phi$ must map the basis
elements $1,\zeta_1,\ldots,\zeta_{n-1}$ for $R_f$ to the corresponding
basis elements $1,\zeta^*_1,\ldots,\zeta^*_{n-1}$ for $R_{f^*}$. Let
$\theta$ denote the image of $x$ in $\Q[x]/(f(x,1))$ and 
$\theta^*$ the image of $x$ in $\Q[x]/(f^*(x,1))$. Then
$\phi(a_0\theta) = a_0^*\theta^*$ and $$a_0^*\theta^{*2} + a_1^*\theta^* =\phi(a_0\theta^2 + a_1\theta)
=(a_0^{*2}/a_0)\theta^{*2} + (a_1a_0^*/a_0)\theta^*.$$ 
Since $\theta^*$ and $\theta^{*2}$
are linearly independent, we have $a_0 = a_0^*$, $a_1 = a_1^*$, and $\phi(\theta) = \theta^*$, where we extend $\phi$ naturally to $R_f\otimes\Q = \Q[x]/(f(x,1))$. Then
since $\phi(\zeta_{n-1}) = \zeta^*_{n-1}$, we have $a_i=a_i^*$ for
$i=0,\ldots,n-2$. Finally $\phi(-a_{n-1}\theta - a_n) =
\phi(\theta\zeta_{n-1}) =\phi(\theta^*\zeta_{n-1}^*)=
-a_{n-1}^*\theta^* - a_n^*$. Hence $a_{n-1} = a^*_{n-1}$ and $a_n =
a^*_n$.
\end{proof}

\smallskip

\noindent {\bf Proof of Theorem \ref{monogenic}:} The condition of
being strongly quasi-reduced is open in $V_n(\R)$. Therefore, given a
strongly quasi-reduced element $f\in V_n(\R)$, there exists an open
neighbourhood $\B$ of $f$ in which every element is strongly
quasi-reduced. Moreover, since the action of $\SL_2(\Z)$ on $V_n(\R)$ is
discrete, we may ensure that no two elements of $\B$ are
$\SL_2(\Z)$-equivalent. We may further scale $\B$ in order to assume
that every element in $\B$ has discriminant bounded by $1$.

Consider the set $\B_X:=X^{1/(2n-2)}\cdot \B$. No two elements in it
are $\SL_2(\Z)$-equivalent and every element in it is strongly
quasi-reduced. Therefore the rings corresponding to any two elements
in $\B_X$ are nonisomorphic. On the other hand, applying Theorem
\ref{th:sqfreesieve}, we see that $\gg X^{(n+1)/(2n-2)}$ integral
elements in $\B_X$ have discriminant less than $X$ and correspond to
maximal orders in degree-$n$ number fields. Since these rings are
pairwise nonisomorphic, so are their fields of fractions. Hence we
have constructed $\gg X^{(n+1)/(2n-2)}$ nonisomorphic degree-$n$
number fields of absolute discriminant less than $X$. Restricting to counting
forms that have squarefree discriminant yields $\gg X^{(n+1)/(2n-2)}$
nonisomorphic $S_n$-number fields. $\Box$

\bigskip

We note that Theorem \ref{th:sqfreesieve} also allows us to construct
$\gg X^{(n+1)/(2n-2)}$ $S_n$-number fields satisfying any finite set
of splitting conditions. 

\appendix

\section{Computations of the local densities $\alpha_n(p),\beta_n(p)$}\label{sec:sieve}

Let $n\geq 2$ be a fixed integer.  For a prime $p$, let $\alpha_n(p)$
denote the density of the set of binary $n$-ic forms having
discriminant indivisible by $p^2$, and let $\beta_n(p)$ denote
the density of binary $n$-ic forms whose associated rank-$n$ rings are
maximal at $p$.  In this section, we compute $\alpha_n(p)$ and
$\beta_n(p)$ for all integers $n\geq 2$ and all primes $p$.

\begin{proposition}
We have $\alpha_2(2) = 1/2$ and $\alpha_n(2) = 3/8$ for $n\geq 3$. For odd primes $p$, we have  
$$\alpha_n(p) = \begin{cases}\displaystyle
\Bigl(1-\frac{1}{p}\Bigr)
\Bigl(1+\frac{1}{p}-\frac{1}{p^3}\Bigr)
&\mbox{ if }n=2,
\\[.15in] \displaystyle
\Bigl(1-\frac{1}{p}\Bigr)^2\Bigl(1+\frac{1}{p}\Bigr)^2
&\mbox{ if }n=3,
\\[.15in] \displaystyle
\Bigl(1-\frac{1}{p}\Bigr)^2
\Bigl(1+\frac{2}{p}-\frac{2}{p^4}+\frac{1}{p^5}\Bigr)
&\mbox{ if }n=4,
\\[.15in] \displaystyle
\Bigl(1-\frac{1}{p}\Bigr)^2
\Bigl(1+\frac{1}{p}\Bigr)\Bigl(1+\frac{1}{p}-\frac{1}{p^2}\Bigr)
&\mbox{ if }n\geq5.\end{cases}$$
\end{proposition}
\begin{proof}
For $j\geq 0$, $n\geq 1$, and $p$ prime, we let $\nu_{j}(n,p)$ denote
the density within monic degree-$n$ integer polynomials of the set of
those whose discriminants have $p$-adic valuation $j$. Then
$\nu_0(n,p)$ and $\nu_1(n,p)$ are computed in \cite[Proposition 6.4 and  
  Theorem 6.8]{ABZ}:
\begin{eqnarray*}
\nonumber  \nu_0(n,p) &=& \begin{cases}1&\mbox{ if }n= 1;\\
    1-p^{-1}&\mbox{ if }n\geq 2.\end{cases}\\[.035in]
  \nonumber\nu_1(n,p) &=& \begin{cases}0&\mbox{ if }p=2
    \mbox{ or }n= 1;\\ p^{-1}(1-p^{-1})&\mbox{ if }n=2,p\neq2;\\
    p^{-1}(1-p^{-1})^2&\mbox{ if }n=3,p\neq2;\\
    (1-p^{-1})^2(1-(-p)^{-n})(1+p)^{-1}&\mbox{ if }n\geq4,p\neq2.
\end{cases}
\end{eqnarray*}
To compute the densities $\alpha_n(p)$, we partition the set of
integral binary $n$-ic forms $f(x,y)=a_0x^n+a_1x^{n-1}y+\cdots+a_ny^n$
whose discriminants are not divisible by $p^2$ into three subsets, and
compute each of their densities. For any binary form
$f(x,y)$ in $\Z[x,y]$ or in $(\Z/p^2\Z)[x,y]$, we write $\bar{f}(x,y)$
for its reduction modulo $p$.

\medskip

\noindent {\bf Subset 1:} The set of $f(x,y)$ with $p\nmid a_0$ and
$p^2\nmid \Delta(f)$. Here, for any fixed leading coefficient
$a_0\not\equiv 0\pmod p$, the density of $f(x,y)$ having discriminant 
indivisible by $p^2$ is simply given by
$\nu_0(n,p)+\nu_1(n,p)$. Therefore, the $p$-adic density of this
subset is equal to
\begin{equation*}
    \Bigl(1-\frac{1}{p}\Bigr)\big(\nu_0(n,p)+\nu_1(n,p)\big).
\end{equation*}

\medskip

\noindent {\bf Subset 2:} The set of $f(x,y)$ with $p\mid a_0$,
$p\nmid a_1$, and $p^2\nmid \Delta(f)$. In this case, we begin by
proving that the density of elements $f$ with fixed $a_0$ and $a_1$
and with $p^2\nmid \Delta(f)$ is the same as the density of binary
$(n-1)$-ic forms $g$, with fixed leading coefficient $a_1$ such that
$p^2\nmid\Delta(g)$. Indeed, given any
$(a_2,\ldots,a_n)\in(\Z/p^2\Z)^{n-1}$, we write
$$\begin{array}{ccrcc}
f_{a_2,\ldots,a_n}(x,y) &=& a_0x^n + a_1x^{n-1}y + a_2x^{n-2}y^2 + \cdots + a_ny^n&\;\,\in&(\Z/p^2\Z)[x,y],\\[.025in]
g_{a_2,\ldots,a_n}(x,y) &=& \,\!a_1x^{n-1} \!+\, a_2x^{n-2}y\; + \,\!\cdots\,\! + a_ny^{n-1}\!\!\!\!\!\!&\;\,\in&(\Z/p^2\Z)[x,y].
\end{array}
$$
Define
\begin{eqnarray*}
\nonumber S_f^{(1)} &=& \{(a_2,\ldots,a_n)\in(\Z/p^2\Z)^{n-1} \colon p^2 \mbox{ strongly divides }\Delta(f_{a_2,\ldots,a_n})\},\\
\nonumber S_f^{(2)} &=& \{(a_2,\ldots,a_n)\in(\Z/p^2\Z)^{n-1} \colon p^2 \mbox{ weakly divides }\Delta(f_{a_2,\ldots,a_n})\},\\
\nonumber S_g^{(1)} &=& \{(a_2,\ldots,a_n)\in(\Z/p^2\Z)^{n-1} \colon p^2 \mbox{ strongly divides }\Delta(g_{a_2,\ldots,a_n})\},\\
\nonumber S_g^{(2)} &=& \{(a_2,\ldots,a_n)\in(\Z/p^2\Z)^{n-1} \colon p^2 \mbox{ weakly divides }\Delta(g_{a_2,\ldots,a_n})\}.
\end{eqnarray*}
Recall that $p^2$ strongly divides the discriminant of $f$ if and only
if $\bar{f}(x,y)$ has a factor of the form $h(x,y)^3$ for some linear
form $h$ or a factor of the form $j(x,y)^2$ where $j$ is a binary form
of degree at least 2. Since $\overline{f_{a_2,\ldots,a_n}}(x,y)\equiv
y\,\overline{g_{a_2,\ldots,a_n}}(x,y)$, and since $y$ does not divide
$\overline{g_{a_2,\ldots,a_n}}(x,y)$, we see that
$\overline{f_{a_2,\ldots,a_n}}(x,y)$ admits such a factor if and only
if $\overline{g_{a_2,\ldots,a_n}}(x,y)$ does. Hence  $S_f^{(1)} =
S_g^{(1)}$. On the other hand, we have
\begin{eqnarray*}
  \nonumber\#S_f^{(2)} &=& \#\{(a_2,\ldots,a_n)\in(\Z/p^2\Z)^{n-1}\backslash S_f^{(1)}
  \colon \exists r\in\Z/p\Z, (x-r)^2\mid
  \overline{f_{a_2,\ldots,a_n}}(x,1), p^2\mid f_{a_2,\ldots,a_n}(r,1)\}\\
  \nonumber&=& \textstyle\frac{1}{p}\#\{(a_2,\ldots,a_n)\in(\Z/p^2\Z)^{n-1}\backslash S_f^{(1)}
  \colon \exists r\in\Z/p\Z, (x-r)^2\mid
  \overline{f_{a_2,\ldots,a_n}}(x,1), p\mid f_{a_2,\ldots,a_n}(r,1)\}\\
  \nonumber&=& \textstyle\frac{1}{p}\#\{(a_2,\ldots,a_n)\in(\Z/p^2\Z)^{n-1}\backslash S_g^{(1)}
  \colon \exists r\in\Z/p\Z, (x-r)^2\mid
  \overline{g_{a_2,\ldots,a_n}}(x,1), p\mid g_{a_2,\ldots,a_n}(r,1)\}\\
  \nonumber&=& \#\{(a_2,\ldots,a_n)\in(\Z/p^2\Z)^{n-1}\backslash S_g^{(1)}
  \colon \exists r\in\Z/p\Z, (x-r)^2\mid
  \overline{g_{a_2,\ldots,a_n}}(x,1), p^2\mid g_{a_2,\ldots,a_n}(r,1)\}\\
  &=& \#S_g^{(2)}.
\end{eqnarray*}
The density $(\#S_g^{(1)}+\#S_g^{(2)})/p^{2(n-1)}$ is
$\nu_0(n-1,p)+\nu_1(n-1,p)$. Taking into account that $p\mid a_0$ and
$p\nmid a_1$, we see that the density of this second subset is
\begin{equation*}
    \frac{1}{p}\Bigl(1-\frac{1}{p}\Bigr)\big(\nu_0(n-1,p)+\nu_1(n-1,p)\big).
\end{equation*}

\medskip

\noindent {\bf Subset 3:} The set of $f(x,y)$ with $p\mid a_0$, $p\mid
a_1$, and $p^2\nmid \Delta(f)$. Note that we already have $p\mid\Delta(f)$
in this case. To ensure that $p^2\nmid \Delta(f)$, we must have $p>2$,
$p^2\nmid a_0$, and $p\nmid a_2$. Indeed, if $p=2$, then since $2\mid
\Delta(f)$, we have $4\mid\Delta(f)$; if $p^2\mid a_0$, then $p^2$
(weakly) divides $\Delta(f)$; and if $p\mid a_2$, then $y^3\mid \bar{f}$
and so $p^2$ (strongly) divides $\Delta(f)$. As polynomials in
$a_0,\ldots,a_n$, we have
\begin{equation*}
\Delta(a_0x^n + \cdots + a_n) \equiv
-4a_0a_{2}^3\Delta(a_2x^{n-2} + \cdots + a_n)
\pmod{a_0^2,a_0a_1,a_1^2}.
\end{equation*}
Hence if $p>2$, $p^2\nmid a_0$, and
$p\nmid a_2$, then $p^2\nmid \Delta(f)$ if and only if
$p\nmid\Delta(a_2x^{n-2}+a_3x^{n-3}y+\cdots a_ny^{n-2})$. Hence the density of this third subset is
\begin{equation*}
    \frac{1}{p^2}\Bigl(1-\frac{1}{p}\Bigr)^2\nu_0(n-2,p).\\[.025in]
\end{equation*}
Adding together these three densities yields the proposition.
\end{proof}

Next, we compute the value of $\beta_n(p)$ for integers $n\geq 2$ and
primes $p$.
\begin{proposition}\label{prop:betacomp}
We have
\begin{equation*}
\beta_n(p) =  \begin{cases}\displaystyle
\Bigl(1-\frac{1}{p}\Bigr)\Bigl(1+\frac{1}{p}-\frac{1}{p^3}\Bigr)&\mbox{ if }n=2;
\\[.125in] \displaystyle \Bigl(1-\frac{1}{p^2}\Bigr)\Bigl(1-\frac{1}{p^3}\Bigr)
&\mbox{ if }n\geq3.\end{cases}
\end{equation*}
\end{proposition}
\begin{proof}
The density of monic degree-$n$ integer polynomials that are maximal at $p$ was computed in \cite[Proposition
  3.5]{ABZ} to be $1-p^{-2}$  for all $n\geq 2$ and all primes $p$.

We compute $\beta_n(p)$ by working over $\Z_p$. Fix a
binary $n$-ic form $f(x,y)\in V_n(\Z_p)$. Suppose $f(x,y)$ (mod $p$) factors as $y^kg(x,y)$, where $g(x,y)$ is a binary
$(n-k)$-ic form over $\F_p$ with nonzero $x^{n-k}$-term for some
$k\in\{0,\ldots,n\}$. Then, by Hensel's lemma, $f(x,y)$ factors as
$h_1(x,y)h_2(x,y)$ where $h_1(x,y)\in\Z_p[x]$ is a binary $k$-ic form
such that $h_1(x,y)$ (mod $p$) is $y^k$ and $h_2(x,y)\in\Z_p[x]$ is a
binary $(n-k)$-ic such that $h_2(x,y)$ (mod $p$) is $g(x,y)$. By
scaling $h_1$ and $h_2$, we may further assume that the leading
coefficient of $h_2(x,y)$ is $1$.

Since $h_1(x,y)$ and $h_2(x,y)$
share no common factors (mod $p$), the rank-$n$ ring over $\Z_p$
associated to $f(x,y)$ is isomorphic to the product of the rings
associated to $h_1(x,y)$ and $h_2(x,y)$. Since $h_1(x,y)$ reduces to a
unit times $y^k$ modulo $p$, the rank-$k$ ring associated to
$h_1(x,y)$ is always maximal when $k\leq1$ and is maximal when
$k\geq2$ if and only if $p^2$ does not divide the
$x^k$-coefficient. On the other hand, $h_2(x,y)$ is monic, and so the
probability that it is maximal is exactly $1-p^{-2}$ when $n-k\geq 2$,
and $1$ when $n-k=1$. When $k=n$, $f(x,y)$ is a multiple of $p$ and
is automatically nonmaximal. Summing over $k$, we have for $n\geq3$,
\begin{eqnarray*}
  \nonumber\beta_n(p) &=& \sum_{k=0}^1 \frac1{p^k}\Bigl(1-\frac{1}{p}\Bigr)\Bigl(1-\frac{1}{p^2}\Bigr) 
  +
 \sum_{k=2}^{n-2} \frac{1}{p^k}\Bigl(1-\frac{1}{p}\Bigr)^2
  \Bigl(1-\frac{1}{p^2}\Bigr) + \sum_{k=n-1}^n
 \frac{1}{p^{k}}\Bigl(1-\frac{1}{p}\Bigr)^2
 \\
&=& \Bigl(1-\frac{1}{p^2}\Bigr)\Bigl(1-\frac{1}{p^3}\Bigr).
\end{eqnarray*}
When $n=2$, we have
\begin{equation*}
  \beta_2(p) = \Bigl(1-\frac{1}{p}\Bigr)\Bigl(1-\frac{1}{p^2}\Bigr) +
  \frac{1}{p}\Bigl(1-\frac{1}{p}\Bigr) + \frac{1}{p^2}\Bigl(1-\frac{1}{p}\Bigr)^2 = 
  \Bigl(1-\frac{1}{p}\Bigr)\Bigl(1+\frac{1}{p}-\frac{1}{p^3}\Bigr).
\end{equation*}
This concludes the proof of Proposition \ref{prop:betacomp}.
\end{proof}

\vspace{-.1in}

\subsection*{Acknowledgments}

The first-named author was supported by a Simons Investigator Grant
and NSF Grant~DMS-1001828. The second-named author was supported by an
NSERC Discovery Grant and Sloan Research Fellowship. The third-named
author was supported by an NSERC Discovery Grant.

\end{document}